\documentclass[11pt, reqno]{amsart}

\usepackage[top=3.75cm, bottom=3cm, left=3cm, right=3cm]{geometry}
\usepackage{amsmath}
\usepackage{amsthm}
\usepackage{amsfonts}
\usepackage{amssymb}
\usepackage{eucal}
\usepackage{bm}
\usepackage{hyperref}
\usepackage{color}
\usepackage[all,cmtip]{xy}
\usepackage{enumitem}
\usepackage{tikz}  
\usepackage{graphicx}
\usepackage{scalerel}

\numberwithin{equation}{section}

\theoremstyle{plain}
\newtheorem{theorem}{Theorem}[section]
\newtheorem{lemma}[theorem]{Lemma}
\newtheorem{proposition}[theorem]{Proposition}
\newtheorem{corollary}[theorem]{Corollary}
\newtheorem{conjecture}[theorem]{Conjecture}

\theoremstyle{definition}
\newtheorem{definition}[theorem]{Definition}
\newtheorem{example}[theorem]{Example}
\newtheorem{remark}[theorem]{Remark}

\newtheorem{question}[theorem]{Question}

\setlist[itemize]{leftmargin=*, itemsep={2pt}}
\setlist[enumerate]{leftmargin=*, itemsep={2pt}}

\makeatletter
\newcommand*{\da@rightarrow}{\mathchar"0\hexnumber@\symAMSa 4B }
\newcommand*{\da@leftarrow}{\mathchar"0\hexnumber@\symAMSa 4C }
\newcommand*{\xdashrightarrow}[2][]{%
  \mathrel{%
    \mathpalette{\da@xarrow{#1}{#2}{}\da@rightarrow{\,}{}}{}%
  }%
}
\newcommand{\xdashleftarrow}[2][]{%
  \mathrel{%
    \mathpalette{\da@xarrow{#1}{#2}\da@leftarrow{}{}{\,}}{}%
  }%
}
\newcommand*{\da@xarrow}[7]{%
  \sbox0{$\ifx#7\scriptstyle\scriptscriptstyle\else\scriptstyle\fi#5#1#6\m@th$}%
  \sbox2{$\ifx#7\scriptstyle\scriptscriptstyle\else\scriptstyle\fi#5#2#6\m@th$}%
  \sbox4{$#7\dabar@\m@th$}%
  \dimen@=\wd0 %
  \ifdim\wd2 >\dimen@
    \dimen@=\wd2 %
  \fi
  \count@=2 %
  \def\da@bars{\dabar@\dabar@}%
  \@whiledim\count@\wd4<\dimen@\do{%
    \advance\count@\@ne
    \expandafter\def\expandafter\da@bars\expandafter{%
      \da@bars
      \dabar@ 
    }%
  }%
  \mathrel{#3}%
  \mathrel{%
    \mathop{\da@bars}\limits
    \ifx\\#1\\%
    \else
      _{\copy0}%
    \fi
    \ifx\\#2\\%
    \else
      ^{\copy2}%
    \fi
  }%
  \mathrel{#4}%
}
\makeatother

\newcommand{\bijartop}[1][]{%
 \ar[#1]
 \ar@<0.7ex>@{}[#1]|-*=0[@]{\sim}} 
 
 \newcommand{\bijarbottom}[1][]{%
 \ar[#1]
 \ar@<-0.95ex>@{}[#1]|-*=0[@]{\sim}} 

\newcommand{\st}{\mid} 
\newcommand{\set}[1]{\left\{ \, #1 \, \right\}}

\newcommand{\Perf}{\mathrm{Perf}}
\newcommand{\Db}{\mathrm{D^b_{coh}}}

\newcommand{\cKb}{\cK^\mathrm{b}_\mathrm{coh}}

\newcommand{\perf}{\mathrm{perf}}

\DeclareMathOperator{\length}{\mathrm{length}}

\newcommand{\hpd}{{\natural}}
\newcommand{\svee}{\scriptscriptstyle\vee}
\newcommand{\cAd}{\cA^\hpd}

\newcommand{\cJd}{\mathcal{J}^{\hpd}}

\newcommand{\llangle}{\left \langle}
\newcommand{\rrangle}{\right \rangle}
\newcommand{\sotimes}{\otimes}

\newcommand{\QCoh}{\mathrm{D_{qc}}}

\newcommand{\HH}{\mathrm{HH}}

\newcommand{\Sing}{\mathrm{Sing}}
\newcommand{\Gr}{\mathrm{Gr}}
\newcommand{\OGr}{\mathrm{OGr}}
\newcommand{\OGrp}{\mathrm{OGr_+}}
\newcommand{\OGrm}{\mathrm{OGr_-}}
\newcommand{\Spin}{\mathrm{Spin}}

\newcommand{\br}[1]{{#1}_{\textrm{br}}}
\newcommand{\cov}[1]{{#1}_{\textrm{cov}}}

\newcommand{\barbH}{\bar{\bH}}
\newcommand{\wtilde}{\widetilde}
\newcommand{\tC}{{\tilde{\bC}}}
\newcommand{\tCC}{{\widetilde{\bC\bC}}}

\newcommand{\tJ}{{\tilde\bJ}}
\newcommand{\tJv}{{{\tilde\bJ}^{\svee}}}
\newcommand{\tJJ}{{\widetilde{\bJ\bJ}}}

\newcommand{\barQ}{\bar{Q}}

\newcommand{\eps}{\varepsilon}
\newcommand{\tp}{{\tilde{p}}}
\newcommand{\barp}{\bar{p}}

\DeclareMathOperator{\Sym}{Sym}
\DeclareMathOperator{\Proj}{Proj}
\DeclareMathOperator{\Pic}{Pic}

\newcommand{\tQ}{{\wtilde{Q}}}
\newcommand{\tX}{{\wtilde{X}}}

\newcommand{\stimes}{\times}

\newcommand{\cUv}{\mathcal{U}^{\svee}}

\DeclareMathOperator{\Cl}{\mathsf{Cliff}}

\newcommand{\barV}{\bar{V}}
\newcommand{\vV}{V^{\svee}}

\newcommand{\tV}{\tilde{V}} 

\newcommand{\tW}{\wtilde{W}}

\newcommand{\barH}{\bar{H}}

\newcommand{\cHom}{\mathcal{H}\!{\it om}}
\newcommand{\cEnd}{\mathcal{E}\!{\it nd}}
\DeclareMathOperator{\Tor}{Tor}

\newcommand{\id}{\mathrm{id}}
\newcommand{\pr}{\mathrm{pr}}

\newcommand{\rank}{\mathrm{rank}}

\newcommand{\cO}{\mathcal{O}}
\newcommand{\cA}{\mathcal{A}}
\newcommand{\cB}{\mathcal{B}}
\newcommand{\cC}{\mathcal{C}}
\newcommand{\cD}{\mathcal{D}}

\newcommand{\cJ}{\mathcal{J}}
\newcommand{\cK}{\mathcal{K}}
\newcommand{\cL}{\mathcal{L}}

\newcommand{\cQ}{\scalerel*{\mathcal{Q}}{Q}}
\newcommand{\cR}{\mathcal{R}}
\newcommand{\cS}{\mathcal{S}}

\newcommand{\cU}{\mathcal{U}}
\newcommand{\cV}{\mathcal{V}}

\newcommand{\rH}{\mathrm{H}}
\newcommand{\rS}{\mathrm{S}}

\newcommand{\rT}{\mathrm{T}}

\newcommand{\fa}{\mathfrak{a}}

\newcommand{\fq}{\mathfrak{q}}

\newcommand{\bC}{\mathbf{C}}

\newcommand{\bE}{\mathbf{E}}
\newcommand{\bH}{\mathbf{H}}
\newcommand{\bJ}{\mathbf{J}}
\newcommand{\bZ}{\mathbf{Z}}
\newcommand{\bP}{\mathbf{P}}

\newcommand{\bk}{\mathbf{k}}

\newcommand{\sS}{\mathsf{S}}
\newcommand{\sSs}{\mathsf{S}_{16}}

\begin{document}

\title{Categorical cones and quadratic homological projective duality} 

\author{Alexander Kuznetsov}
\address{{\sloppy
\parbox{0.9\textwidth}{
Steklov Mathematical Institute of Russian Academy of Sciences,\\
8 Gubkin str., Moscow 119991 Russia
\\[5pt]
Interdisciplinary Scientific Center J.-V. Poncelet (CNRS UMI 2615), Moscow, Russia
\hfill\\[5pt]
National Research University Higher School of Economics, Moscow, Russia
}\bigskip}}
\email{akuznet@mi.ras.ru \medskip}

\author{Alexander Perry}
\address{Department of Mathematics, Columbia University, New York, NY 10027 \smallskip}
\email{aperry@math.columbia.edu}

\thanks{A.K. was partially supported by the Russian Academic Excellence Project ``5-100''. 
A.P. was partially supported by an NSF postdoctoral fellowship, DMS-1606460.}

\begin{abstract}
We introduce the notion of a categorical cone, which provides
a categorification of the classical cone over a projective variety, 
and use our work on categorical joins to describe its behavior under homological projective duality. 
In particular, our construction provides well-behaved categorical resolutions of singular quadrics, 
which we use to obtain an explicit quadratic version of the main theorem of 
homological projective duality. 
As applications, we 
prove the duality conjecture for Gushel--Mukai varieties, and 
produce interesting examples of conifold transitions between 
noncommutative and honest Calabi--Yau threefolds.   
\end{abstract}

\maketitle



\section{Introduction}
\label{section-intro}

This paper is a sequel to \cite{categorical-joins}, where we introduced categorical joins 
in the context of homological projective duality (HPD). 
Building on that work, our goals here are to study a categorical version of the classical cone over 
a projective variety, 
to use categorical quadratic cones to give a powerful method for studying derived categories of quadratic sections of 
varieties, and to give several applications. 

\subsection{Background} 

The basic object of HPD is a \emph{Lefschetz variety}, which consists of a 
variety mapping to a projective space $X \to \bP(V)$ equipped with a 
\emph{Lefschetz decomposition} of its derived category (a special type of semiorthogonal decomposition). 
The theory in this form was introduced and developed in~\cite{kuznetsov-hpd}. 
However, already at that point it was clear that the theory is more categorical in nature, 
and that for applications it is useful to replace the (perfect) derived category $\Perf(X)$ of $X$
by a more general (suitably enhanced) triangulated category~$\cA$ equipped with a Lefschetz decomposition; 
the structure of a map $X \to \bP(V)$ is then replaced by a~$\bP(V)$-linear structure (an action of the monoidal category $\Perf(\bP(V))$) on $\cA$. 
We call such data a \emph{Lefschetz category} over $\bP(V)$ and think of it as of a \emph{noncommutative} Lefschetz variety. 
The reader is encouraged to focus on the case where $X \to \bP(V)$ is 
an ordinary morphism of varieties for this introduction, and to consult 
\cite{NCHPD, categorical-joins} for more details on the noncommutative situation. 

The HPD of a (noncommutative) Lefschetz variety~$X \to \bP(V)$ is another (noncommutative) Lefschetz variety
\begin{equation*}
X^{\hpd} \to \bP(\vV)
\end{equation*}
over the dual projective space, 
which governs the derived categories of linear sections of $X$ and can be thought of as a 
categorical version of the classical projective dual. 
For details and applications of this theory, see \cite{kuznetsov-hpd, NCHPD, kuznetsov2014semiorthogonal, thomas2015notes}.

In~\cite{categorical-joins} given a pair of (noncommutative) Lefschetz varieties $X_1 \to \bP(V_1)$ and $X_2 \to \bP(V_2)$, 
we constructed a (noncommutative) Lefschetz variety 
\begin{equation*}
\cJ(X_1, X_2) \to \bP(V_1 \oplus V_2) 
\end{equation*}
called their \emph{categorical join}, which can be thought of as a noncommutative resolution of singularities 
of the classical join of $X_1$ and~$X_2$.
Moreover, we proved that various classical properties of joins can be lifted to this  
level; in particular, (under suitable assumptions) there is an equivalence of Lefschetz varieties 
\begin{equation}
\label{HPD-joins}
\cJ(X_1, X_2)^{\hpd} \simeq \cJ(X_1^{\hpd}, X_2^{\hpd}) 
\end{equation} 
over $\bP(\vV_1 \oplus \vV_2)$, i.e.\ the HP dual of a categorical join is the categorical join of HP duals.
This leads to numerous applications, including a nonlinear HPD theorem 
(see also \cite{categorical-plucker}) giving an equivalence between the ``essential parts'' of the derived categories of 
fiber products
\begin{equation*}
X_1 \times_{\bP(V)} X_2 \qquad \text{and} \qquad X^\hpd_1 \times_{\bP(\vV)} X^\hpd_2 . 
\end{equation*} 
The simplest case of this result --- when $X_2$ is a linear subspace of $\bP(V)$ 
and $X_2^\hpd$ is its orthogonal linear subspace of $\bP(\vV)$ --- reduces to the main theorem of HPD,
and other examples of HPD pairs $(X_2,X_2^\hpd)$ provide a generalization of this.
Such generalizations are more useful in cases when $X_2$ and $X_2^\hpd$ both have a nice geometric description.
One of the goals of this paper is to produce such pairs where both $X_2$ and $X_2^\hpd$ 
are \emph{categorical resolutions} of singular quadrics and to relate in this way quadratic sections of $X_1$ and $X_1^\hpd$.
Allowing the quadrics to be singular is crucial for applications, as we will explain below in~\S\ref{subsection:importance}. 

\subsection{Categorical cones}
Assume given an exact sequence of vector spaces 
\begin{equation}
\label{V0VbarV}
0 \to V_0 \to V \to \barV \to 0
\end{equation} 
and a closed subvariety $X \subset \bP(\barV)$. 
Recall that the \emph{classical cone} over $X$ with vertex $\bP(V_0)$ is the strict transform 
\begin{equation*}
\bC_{V_0}(X) \subset \bP(V) 
\end{equation*}
of $X$ under the linear projection $\bP(V) \dashrightarrow \bP(\barV)$ from $\bP(V_0)$. 
Note that $\bC_{V_0}(X)$ is usually highly singular along its vertex $\bP(V_0) \subset \bC_{V_0}(X)$. 

In this paper, given a (noncommutative) Lefschetz variety $X \to \bP(\barV)$, we construct a (noncommutative) Lefschetz variety 
\begin{equation*}
\cC_{V_0}(X) \to \bP(V)
\end{equation*}
called the \emph{categorical cone} which provides (if $X$ is smooth) a categorical resolution of~$\bC_{V_0}(X)$. 
The basic idea of the construction is to first replace the classical cone with 
the \emph{resolved cone}~$\tC_{V_0}(X) \to \bP(V)$ given by the blowup along 
$\bP(V_0) \subset \bC_{V_0}(X)$; the resolved cone is the projectivization of the 
pullback of a natural vector bundle on $\bP(\barV)$, and hence makes sense 
even when $X \to \bP(\barV)$ is not an embedding. 
The categorical cone is then defined as a certain 
triangulated subcategory of $\Perf(\tC_{V_0}(X))$ following a construction in~\cite{kuznetsov2008lefschetz}, 
and can be thought of as a 
noncommutative birational modification of $\tC_{V_0}(X)$ along its exceptional divisor. 

The categorical cone
has several advantages over its classical counterpart: 
\begin{itemize}
\item $\cC_{V_0}(X)$ is defined when $X \to \bP(\barV)$ is not an embedding, 
and even when $X$ is noncommutative (Definition~\ref{definition-cat-cone}). 
\item $\cC_{V_0}(X)$ naturally has the structure of a Lefschetz variety over $\bP(V)$ induced by that of~$X$
(Theorem~\ref{theorem-cone-lef-cat}). 
\item $\cC_{V_0}(X)$ is smooth and proper if $X$ is (Lemma~\ref{lemma-cC-smooth-proper}). 
\end{itemize}

For us, however, the main advantage of the categorical cone is its compatibility with HPD:
our first main result 
is the identification of the HPD of a categorical cone as another categorical cone. 
In fact, we work in a more general setup than above, that simultaneously allows for extensions 
of the ambient projective space, because this extra generality is useful in applications (see~\S\ref{subsection:importance}). 
Namely, let $V$ be a vector space and assume given a pair of subspaces 
\begin{equation*}
V_0 \subset V \qquad \text{and} \qquad V_{\infty} \subset \vV 
\end{equation*}
such that~$V_0 \subset V_\infty^\perp$, or equivalently $V_\infty \subset V_0^\perp$, 
where the orthogonals are taken with respect to the natural pairing between $V$ and $\vV$. 
Let 
\begin{equation*}
\barV = V_\infty^\perp/V_0,\qquad\text{so that}\qquad
\barV^{\svee} \cong V_0^\perp/V_\infty. 
\end{equation*}
For~\mbox{$V_{\infty} = 0$} this reduces to the situation~\eqref{V0VbarV} above. 
Let $X \to \bP(\barV)$ be a 
Lefschetz variety, with HPD variety $X^\hpd \to \bP(\barV^{\svee})$. 
The categorical cone $\cC_{V_0}(X)$ is then a Lefschetz variety over~$\bP(V_{\infty}^{\perp})$. 
Via the inclusion $\bP(V_{\infty}^{\perp}) \to \bP(V)$ we can regard~$\cC_{V_0}(X)$ 
as a Lefschetz variety over $\bP(V)$, which we write as 
$\cC_{V_0}(X)/\bP(V)$ for emphasis. 
Similarly, we have a Lefschetz variety $\cC_{V_{\infty}}(X^\hpd)/\bP(\vV)$ over $\bP(\vV)$. 

\begin{theorem}[{Theorem~\ref{theorem-cones-HPD}}]
\label{cone-theorem-intro} 
In the above situation if $X$ is a right strong, moderate Lefschetz variety over $\bP(\barV)$
and $X^\hpd$ is its HPD over $\bP(\barV^{\svee})$, then 
there is an equivalence 
\begin{equation*}
(\cC_{V_0}(X)/\bP(V))^\hpd \simeq \cC_{V_\infty}(X^\hpd)/\bP(\vV)
\end{equation*}
of Lefschetz varieties over $\bP(\vV)$,
i.e., $\cC_{V_\infty}(X^\hpd)$ is the HPD of $\cC_{V_0}(X)$ over $\bP(V)$.
\end{theorem}

In the statement of the theorem ``right strong'' and ``moderate'' refer to technical assumptions on a Lefschetz variety 
(see Definitions~\ref{definition:strong} and~\ref{def:moderateness}) which are essentially always satisfied in practice.
The theorem
categorifies an analogous classical relation between cones and projective duality: 
for a variety $X \subset \bP(\barV)$ we have 
\begin{equation*}
(\bC_{V_0}(X) \subset \bP(V))^{\svee} = \bC_{V_\infty}(X^{\svee}) \subset \bP(\vV) , 
\end{equation*} 
where $(-)^{\svee}$ denotes the operation of classical projective duality. 

Our proof of Theorem~\ref{cone-theorem-intro} is based on the 
HPD result~\eqref{HPD-joins} for categorical joins mentioned above, and the following relation between 
categorical cones and joins.  
Given an exact sequence~\eqref{V0VbarV} 
with $V_0 \neq 0$, we show that if $X \to \bP(\barV)$ is a Lefschetz variety, 
then the choice of a splitting of~\eqref{V0VbarV} induces a natural equivalence 
\begin{equation}
\label{cones-joins-equivalence}
\cC_{V_0}(X) \simeq \cJ(\bP(V_0), X)
\end{equation}
of Lefschetz categories over $\bP(V$) (Proposition~\ref{proposition-cJ-cC}). 

\begin{remark}
For $V_0 = 0$ the identification \eqref{cones-joins-equivalence}
fails, since then $\cC_{V_0}(X) \simeq X$ whereas $\cJ(\bP(V_0), X) = 0$. 
Moreover, even if~$V_0 \neq 0$, we need to choose a splitting of~\eqref{V0VbarV} 
to be able to form $\cJ(\bP(V_0), X)$. 
When working over a field (as we tacitly do in the introduction) 
this is not a problem, but it is typically not possible when working 
over a general base scheme, as we do in the body of the paper with a view 
toward applications. 
Finally, when~\eqref{cones-joins-equivalence} holds, there is an advantage of 
working with the categorical cone description: $\cC_{V_0}(X)$ becomes isomorphic 
to the classical cone $\bC_{V_0}(X)$ over an a priori bigger Zariski open locus than 
the categorical join $\cJ(\bP(V_0), X)$, which is also important for geometric applications. 
\end{remark} 

\subsection{Quadratic HPD}

We use categorical cones and results from~\cite{kuznetsov-perry-HPD-quadrics} to develop HPD for singular quadrics.
By a quadric, we mean an integral scheme isomorphic to a degree~$2$ 
hypersurface in a projective space. 
Any quadric $Q$ can be expressed as a classical cone $Q = \bC_K(\barQ)$ 
over a smooth quadric $\barQ$, where $\bP(K) = \Sing(Q)$. 
We consider the categorical cone
\begin{equation*}
\cQ = \cC_{K}(\barQ), 
\end{equation*}
where $\barQ$ is equipped with a natural Lefschetz decomposition involving spinor bundles, see Lemma~\ref{lemma-cQ-lc}. 
This $\cQ$ is in fact a crepant categorical resolution of singularities of $Q$, see 
Lemma~\ref{lemma-cQ-categorical-resolution}. 
We call it the \emph{standard categorical resolution} of~$Q$.

We deduce from Theorem~\ref{cone-theorem-intro} and~\cite[Theorem~1.1]{kuznetsov-perry-HPD-quadrics} that
the class of standard categorical resolutions of quadrics is closed under HPD. 
Namely, we consider pairs $(Q, f)$ where $Q$ is a quadric and~$f \colon Q \to \bP(V)$ is a \emph{standard morphism}, 
i.e. such that $f^*\cO_{\bP(V)}(1)$ is the ample line bundle that realizes $Q$ as a quadric hypersurface in a projective space. 
In other words, $f$ is either an embedding as a quadric hypersurface into a linear subspace of $\bP(V)$,
or a double covering of a linear subspace of~$\bP(V)$ branched along a quadric hypersurface.
We define in Definition~\ref{definition:generalized-duality} a \emph{generalized duality} operation $(Q, f) \mapsto (Q^{\hpd}, f^{\hpd})$ 
on such pairs, where the target of $f^{\hpd} \colon Q^{\hpd} \to \bP(\vV)$ is the dual projective space. 
This generalized duality reduces to classical projective duality when 
$Q$ has even rank and $f \colon Q \to \bP(V)$ is an embedding, 
and involves passing to a double covering or branch divisor in other cases. 

\begin{theorem}[{Theorem~\ref{theorem:hpd-quadrics-general}}]
\label{intro-hpd-quadrics} 
Let $(Q, f)$ and $(Q^{\natural}, f^{\natural})$ be a generalized dual pair 
as above. 
Then the HPD of the standard categorical resolution of $Q$ over $\bP(V)$  
is equivalent to the standard categorical resolution of $Q^{\natural}$ over $\bP(\vV)$.  
\end{theorem}

By combining Theorem~\ref{intro-hpd-quadrics} with our  
nonlinear HPD Theorem from \cite{categorical-joins}, we prove the 
following quadratic HPD theorem. 

\begin{theorem}[{Theorem~\ref{theorem-quadric-intersection} and Lemma~\ref{lemma-cQ-resolution}}] 
\label{intro-quadratic-HPD}
Let $X \to \bP(V)$ be a right strong, moderate Lefschetz variety. 
Let $f \colon Q \to \bP(V)$ and $f^{\hpd} \colon Q^{\hpd} \to \bP(\vV)$
be a generalized dual pair of quadrics, 
with standard categorical resolutions $\cQ$ and $\cQ^\natural$, respectively.  
Then there are induced semiorthogonal decompositions of 
\begin{equation*}
\Perf(X) \otimes_{\Perf(\bP(V))} \cQ 
\qquad\text{and}\qquad
\Perf(X^\natural) \otimes_{\bP(\vV)} \cQ^\natural 
\end{equation*} 
which have a distinguished component in common. 

Moreover, when $X$ and $X^{\natural}$ are supported away from the singular loci of 
$Q$ and $Q^{\hpd}$, the above tensor product categories 
are identified with the derived categories of the fiber products 
\begin{equation*}
X \times_{\bP(V)} Q \qquad \text{and} \qquad 
X^\natural \times_{\bP(\vV)} Q^\natural 
\end{equation*} 
of $X$ and $X^\natural$ with the underlying quadrics~$Q$ and~$Q^\hpd$.
\end{theorem}

The semiorthogonal decompositions mentioned above are described in Theorem~\ref{theorem-quadric-intersection}.

\subsection{The importance of being singular}
\label{subsection:importance}

An interesting feature of generalized duality of quadrics is that the dimension of~$Q^{\hpd}$ 
may be very different from the dimension of~$Q$; 
in fact, the dimension of~$Q^{\hpd}$ decreases as the dimension of the singular locus of $Q$ increases, see~\eqref{dimQd-quadric}.
This observation has interesting consequences.

Indeed, imagine we are interested in a fiber product 
\begin{equation}
\label{eq:z-x-q}
Z = X \times_{\bP(V)} Q, 
\end{equation} 
where $X$ is a Lefschetz variety over $\bP(V)$, whose homological projective dual variety $X^\hpd$ is known,
and $Q \subset \bP(V)$ is a quadric hypersurface (the case where $Q \to \bP(V)$ is a standard morphism of a quadric of other type works similarly). 
Imagine also that $X$ itself is an intersection of quadrics 
(or at least there is a big family of quadrics in~$\bP(V)$ containing the image of~$X$);
note that this assumption is usually satisfied in applications, since most of varieties 
for which the HPD is known are homogeneous, and every homogeneous variety (in an equivariant embedding) is an intersection of quadrics.

Under this assumption the quadric~$Q$ such that $Z$ is defined by a fiber product~\eqref{eq:z-x-q} is not unique;
indeed, it can be replaced by any quadric in the affine space of quadrics which contain $Z$ but not~$X$.
Typically, the rank of~$Q$ varies in this family.
From this we obtain a family of ``dual fiber products''~\mbox{$X^\hpd \times_{\bP(\vV)} Q^\hpd$} parameterized by the same affine space,
which have varying dimension, but all contain the distinguished component $\cK(Z) \subset \Perf(Z)$.
If we want to use these varieties to understand the structure of~$\cK(Z)$, it is natural to chose 
a fiber product $X^\hpd \times_{\bP(\vV)} Q^\hpd$ of smallest possible dimension 
(hence the most singular quadric~$Q$ defining~$Z$) and use its geometry.

To show how this works consider for example the Fermat quartic surface
\begin{equation*}
Z = \{ x_0^4 + x_1^4 + x_2^4 + x_3^4 = 0 \} \subset \bP^3 =: \bP(W) . 
\end{equation*}
Note that it can be realized as a fiber product~\eqref{eq:z-x-q}, 
where \mbox{$X = \bP(W)$}, \mbox{$V = \Sym^2W$}, the map \mbox{$X \to \bP(V)$} is the double Veronese embedding, 
and~\mbox{$Q \subset \bP(V)$} is any quadric hypersurface, corresponding to a point of an affine space 
over the vector space \mbox{$\ker(\Sym^2\Sym^2W^\vee \to \Sym^4W^\vee$)} of quadrics containing $X$.
The most singular quadric among these is the quadric
\begin{equation*}
Q_0 = \{ x_{00}^2 + x_{11}^2 + x_{22}^2 + x_{33}^2 = 0 \} \subset \bP^9 = \bP(V),
\end{equation*}
where $x_{ij}$ is the coordinate on $\bP(V)$ corresponding to the quadratic function $x_i \cdot x_j$ on $\bP(W)$.
Note that the kernel space $K$ of the corresponding quadratic form on $V$ is 6-dimensional. 
In this case, the generalized dual $Q_0^{\hpd}$ of $Q_0$ 
coincides with the classical projective dual $Q_0^{\svee}$ of $Q_0$, which is 
a smooth quadric surface in the linear space 
\begin{equation*}
\bP(\langle x_{00}, x_{11}, x_{22}, x_{33} \rangle) = \bP(K^\perp) \subset \bP(\vV) = \bP^9
\end{equation*}
of codimension~6.  
On the other hand, in this case $X^\hpd = (\bP(\vV),\Cl_0)$ is the noncommutative variety 
whose derived category is the category of coherent sheaves of $\Cl_0$-modules on $\bP(\vV)$, 
where $\Cl_0$ is the universal sheaf of even parts of Clifford algebras on $\bP(\vV)$ \cite{kuznetsov08quadrics}.
Therefore, the dual fiber product can be rewritten as 
\begin{equation*}
X^\hpd \times_{\bP(\vV)} Q_0^\hpd = \left({Q_0^{\svee}},\Cl_0\vert_{{Q_0^{\svee}}}\right)
\end{equation*}
and Theorem~\ref{intro-quadratic-HPD} gives an equivalence of categories
\begin{equation}
\label{eq:fermat}
\Perf(Z) \simeq \Perf\left({Q_0^{\svee}},\Cl_0\vert_{{Q_0^{\svee}}}\right). 
\end{equation}
(See Remark~\ref{remark:rhs-explanation} below for a more precise description of the right hand side).

Note that if we replace $Q_0$ with a general quadric cutting out $Z$ in $X$, 
then instead of the above equivalence we would obtain a fully faithful embedding of $\Perf(Z)$
into the derived category of sheaves of $\Cl_0$-modules over an 8-dimensional quadric in $\bP(\vV)$,
which is definitely less effective.

\begin{remark}
\label{remark:rhs-explanation}
In fact, the equivalence~\eqref{eq:fermat} can be made more precise as follows.
Consider the union of coordinate hyperplanes in the above space $\bP(K^\perp) =  \bP^3$ (this is a reducible quartic hypersurface)
and the double covering $Z' \to {Q_0^{\svee}}$ branched along the intersection of ${Q_0^{\svee}}$ with these hyperplanes.
Then $Z'$ is a K3 surface with 12 ordinary double points.
Then the sheaf of algebras $\Cl_0$ defines a Brauer class of order~2 on the resolution of singularities of $Z'$,
and the right hand side of~\eqref{eq:fermat} is equivalent to the corresponding twisted derived category.
\end{remark}

\subsection{Duality of Gushel--Mukai varieties}
\label{subsection-intro-applications} 

As an application of our results,
we prove the duality conjecture for \emph{Gushel--Mukai} (GM) varieties 
from~\cite{kuznetsov2016perry}. 
Abstractly, the class of smooth GM varieties consists of smooth Fano varieties of Picard number~$1$, coindex~$3$, 
and degree~$10$, together with Brill--Noether general polarized K3 surfaces of degree~$10$; 
concretely, any such variety can be expressed as an intersection of the cone over the Pl\"{u}cker 
embedded Grassmannian $\Gr(2,5) \subset \bP^9$ with a linear space and a quadric~$Q$,
or equivalently, as a fiber product of~$\Gr(2,5)$ with a standard morphism $Q \to \bP^9$.

In~\cite[Definitions~3.22 and~3.26]{debarre2015kuznetsov} the notions of \emph{period partnership} and \emph{duality} 
for a pair of GM varieties of the same dimension were introduced, 
and in~\cite[Proposition~3.28]{debarre2015kuznetsov} the notion of duality was related to projective duality of quadrics.
Moreover, in~\cite[Corollary~4.16 and Theorem~4.20]{debarre2015kuznetsov} it was shown 
that smooth period partners and dual GM varieties are always birational.
Finally, \cite[Theorem~5.1 and Remark~5.28]{debarre-kuznetsov-periods} combined with~\cite[Theorem~1.3]{ogrady2015periods} 
proved that period partners of any smooth GM variety of even dimension form the fiber 
of the period map from the moduli space of smooth GM varieties to the appropriate period domain 
(a similar result is also conjectured for GM varieties of odd dimension).

On the other hand, in~\cite[Proposition~2.3]{kuznetsov2016perry} a semiorthogonal decomposition
of the derived category $\Db(Y)$ of a smooth GM variety $Y$ consisting of exceptional vector bundles 
and the \emph{GM category} $\cK(Y) \subset \Db(Y)$ was constructed.
The GM category was shown to be a noncommutative K3 or Enriques surface according to whether $\dim(Y)$ is even or odd. 
In~\cite[Definition~3.5]{kuznetsov2016perry} the notions of period partnership and duality 
were generalized to allow GM varieties of different dimension (but of the same parity!). 
The following result settles the duality conjecture~\cite[Conjecture~3.7]{kuznetsov2016perry}, 
which previously was only known in a very special case by~\cite[Theorem~4.1 and Corollary~4.2]{kuznetsov2016perry}.

\begin{theorem}[{Corollary~\ref{corollary-duality-GM}}] 
\label{theorem-duality-conjecture} 
Let $Y_1$ and $Y_2$ be smooth GM varieties whose associated Lagrangian subspaces do not contain decomposable vectors. 
If $Y_1$ and $Y_2$ are generalized partners or duals, then there is an equivalence $\cK(Y_1) \simeq \cK(Y_2)$.  
\end{theorem}

For the notion of the Lagrangian subspace associated to a GM variety see~\cite[\S3]{debarre2015kuznetsov} 
and the discussion in~\S\ref{subsection:duality-gm} below. 
For now we just note that, with the exception of some GM surfaces, the assumption of the theorem holds for all smooth GM varieties. 

Let us explain some consequences of this result. 
In combination with the period results from~\cite{debarre-kuznetsov-periods} mentioned above, 
Theorem~\ref{theorem-duality-conjecture} shows that the assignment $Y \rightsquigarrow \cK(Y)$ 
is constant on the fibers of the period morphism (at least in the even-dimensional case); 
since these fibers are positive-dimensional, this is an interesting phenomenon connecting 
Hodge theory to derived categories.
Moreover, in combination with the birationality results from~\cite{debarre2015kuznetsov} also mentioned above, 
Theorem~\ref{theorem-duality-conjecture} gives strong evidence for the following 
conjecture. 

\begin{conjecture} 
\label{conjecture-equivalent-K3-birational} 
If $Y_1$ and $Y_2$ are GM varieties of the same dimension such that there is 
an equivalence $\cK(Y_1) \simeq \cK(Y_2)$, then $Y_1$ and $Y_2$ are birational. 
\end{conjecture}

Because of the tight parallels between GM fourfolds, cubic fourfolds, and their associated K3 
categories (see \cite[Theorem 1.3]{kuznetsov2016perry}), 
Theorem~\ref{theorem-duality-conjecture} can also be considered as evidence for the analogous 
conjecture for cubic fourfolds suggested by Huybrechts (see \cite[Question 3.25]{macri-K3-survey}). 
We note that every GM fivefold or sixfold is rational \cite[Proposition 4.2]{debarre2015kuznetsov}, so 
Conjecture~\ref{conjecture-equivalent-K3-birational} is of interest specifically for 
GM threefolds and fourfolds. 
As explained in \cite[\S3.3]{kuznetsov2016perry}, Theorem~\ref{theorem-duality-conjecture} also verifies 
cases of the derived category heuristics for rationality discussed in \cite{kuznetsov2010derived, kuznetsov2015rationality}. 

Finally, we note that Theorem~\ref{theorem-duality-conjecture} implies that for certain special GM 
fourfolds and sixfolds~$Y$, there exists a K3 surface $S$ such that $\cK(Y) \simeq \Db(S)$ 
(see \cite[\S3.2]{kuznetsov2016perry}). 
In fact, for some GM fourfolds this is the main result of~\cite{kuznetsov2016perry}, and our proof of 
Theorem~\ref{theorem-duality-conjecture} gives an extension and a conceptual new proof of this result. 
We expect this fact that even GM categories are ``deformation equivalent'' to an ordinary K3 surface to 
be very important for future applications. 
For instance, the analogous property for the K3 category of a cubic fourfold is crucial in \cite{families-stability}, 
where a structure theory for moduli of Bridgeland stable 
objects in such a category is developed, giving (among other 
results) infinitely many new locally-complete unirational families of polarized hyperk\"{a}hler varieties. 

\subsection{Other applications}
\label{subsection-intro-other-applications}

For another application of the quadratic HPD theorem, we introduce a class of \emph{spin GM varieties}. 
Roughly, these are obtained by replacing the role of the Grassmannian~$\Gr(2,5) \subset \bP^9$ in the definition of 
GM varieties with the connected component~$\OGr_+(5,10) \subset \bP^{15}$ of the orthogonal Grassmannian~$\OGr(5,10)$ 
in its spinor embedding. 
The \emph{spin GM category} $\cK(Y) \subset \Db(Y)$ corresponding to such a variety can be thought of as 
a $3$-dimensional analogue of a GM category, as it is (fractional) Calabi--Yau of dimension $3$. 
In this setting, we deduce from Theorem~\ref{intro-quadratic-HPD} a spin analogue 
of Theorem~\ref{theorem-duality-conjecture} (see Theorem~\ref{theorem-sGM-duality}). 

Going further, we consider the case where $Y$ is a fivefold, which is particularly interesting 
from the perspective of rationality. 
The heuristics of \cite{kuznetsov2010derived, kuznetsov2015rationality} 
lead to the following conjecture: if such a $Y$ is rational, then 
$\cK(Y) \simeq \Db(M)$ for a Calabi--Yau threefold $M$. 
We show that such an equivalence cannot exist if $Y$ is smooth (Lemma~\ref{lemma-sGM-not-geometric}), 
and hence we expect $Y$ to be irrational. 
On the other hand, we use Theorem~\ref{intro-quadratic-HPD} to prove the following result 
(stated somewhat imprecisely here), which verifies the conjecture in a mildly degenerate case.

\begin{theorem}[{Theorem~\ref{theorem-singular-sGM} and Corollary~\ref{corollary-singular-sGM-rational}}]
\label{theorem-nodal-spinGM-intro}
For certain nodal spin GM fivefolds $Y$, the variety $Y$ is rational 
and there exists a smooth Calabi--Yau threefold $M$ which gives a crepant categorical resolution of $\cK(Y)$. 
\end{theorem} 

Finally, we note that Theorem~\ref{theorem-nodal-spinGM-intro} can be regarded as giving a 
noncommutative conifold transition from a smooth spin GM category to the Calabi--Yau threefold $M$. 
This suggests a noncommutative version of Reid's fantasy \cite{reid-fantasy}: by degenerations and 
crepant resolutions, can we connect any noncommutative Calabi--Yau threefold to the 
derived category of a smooth projective Calabi--Yau threefold? 
When the answer to this question is positive, it opens the way to proving results 
by deforming to a geometric situation. 
For instance, using the methods of \cite{families-stability}, this gives a potential way to reduce the 
construction of stability conditions on noncommutative Calabi--Yau threefolds to the geometric 
case. 
Further, once stability conditions are known to exist, one can try to analyze the corresponding 
moduli spaces of semistable objects by relating them to the case of geometric Calabi--Yau threefolds; 
this would be a higher-dimensional version of the approach to studying moduli spaces of objects in 
the K3 category of a cubic fourfold carried out in \cite{families-stability}. 

\subsection{Conventions} 
\label{subsection:conventions}

In this paper, we follow the conventions laid out in \cite[\S1.7]{categorical-joins}, which we 
briefly summarize here. 
All schemes are quasi-compact and separated, and we work relative to a 
fixed base scheme $S$. 
For the applications in~\S\ref{section-HPD-for-quadrics} and~\S\ref{section:applications}, 
we assume the base scheme $S$ is the spectrum of an algebraically closed field $\bk$ of characteristic $0$. 
A vector bundle $V$ on a scheme $T$ means a finite locally free $\cO_T$-module; we use the 
convention that 
\begin{equation*}
\bP(V) = {\bP_T(V) = {}} \Proj(\Sym^\bullet(V^{\svee})) \to T
\end{equation*}
with $\cO_{\bP(V)}(1)$ normalized so that its pushfoward to $T$ is $V^{\svee}$. 
A subbundle $W \subset V$ is an inclusion of vector bundles whose cokernel is a vector bundle. 
Given such a $W \subset V$, its orthogonal is the subbundle of $\vV$ given by 
\begin{equation*}
W^{\perp} = \ker(V^{\svee} \to W^{\svee}). 
\end{equation*}
By abuse of notation, given a line bundle $\cL$ or a divisor class $D$ on a scheme~$T$, 
we denote still by $\cL$ or $D$ its pullback to any variety mapping to $T$. 
Similarly, if $X \to T$ is a morphism and $V$ is a vector bundle on $T$, we sometimes 
write $V \otimes \cO_X$ for the pullback of $V$ to $X$.

Given morphisms of schemes $X \to T$ and $Y \to T$, the symbol 
$X \times_T Y$ denotes their \emph{derived} fiber product (see \cite{lurie-SAG, gaitsgory-DAG}), 
which agrees with the usual fiber product of schemes whenever $X$ and $Y$ are 
$\Tor$-independent over $T$. 
We write fiber products over our fixed base $S$ as absolute fiber products, i.e. 
$X \times Y := X \times_S Y$.

We work with linear categories as reviewed in~\cite[\S1.6 and Appendix A]{categorical-joins}.
In particular, given a scheme $X$ over $T$, we denote by $\Perf(X)$ its category of perfect 
complexes and by~$\Db(X)$ its bounded derived category of coherent sheaves, regarded as $T$-linear categories. 

If $\cC$ is a~$T$-linear category and $T' \to T$ is a morphism of schemes, 
we denote by 
\begin{equation*}
\cC_{T'} = \cC \otimes_{\Perf(T)} \Perf(T')
\end{equation*}
the base change of $\cC$ along $T' \to T$.
If $Z \subset T$ is a closed subset, we say $\cC$ is \emph{supported over $Z$} if $\cC_U \simeq 0$, where $U = T \setminus Z$.
If $U \subset T$ is an open subset, we say $\cC$ is \emph{supported over $U$} 
if the restriction functor $\cC \to \cC_U$ is an equivalence.

All functors considered in this paper (pullback, pushforward, tensor product) 
will be taken in the derived sense. 
Recall that for a morphism of schemes $f \colon X \to Y$ the pushforward~$f_*$ is right adjoint to the pullback $f^*$. 
Sometimes, we need other adjoint functors as well.
Provided they exist, we denote by $f^!$ the right adjoint of $f_* \colon \Perf(X) \to \Perf(Y)$ 
and by~$f_!$ the left adjoint of~$f^* \colon \Perf(Y) \to \Perf(X)$, so that $(f_!,f^*,f_*,f^!)$ is an adjoint sequence. 

\begin{remark}
\label{remark-adjoints-exist}
The above adjoint functors all exist if $f \colon X \to Y$ is a morphism between schemes which 
are smooth and projective over $S$ 
(see~\cite[Remark 1.14]{categorical-joins}); 
this will be satisfied in all of the cases where we need $f^!$ and $f_!$ in the paper.
\end{remark} 

\subsection{Organization of the paper} 
In \S\ref{section:preliminaries} we review preliminaries on HPD.
In \S\ref{section-categorical-cones} we define categorical cones, study their 
basic properties, and relate them to categorical joins. 
In \S\ref{section-cones-HPD} we prove Theorem~\ref{cone-theorem-intro} on 
HPD for categorical cones. 
In \S\ref{section-HPD-for-quadrics} we introduce standard categorical resolutions of 
quadrics, and prove the HPD result Theorem~\ref{intro-hpd-quadrics} for them 
and the quadratic HPD theorem stated as Theorem~\ref{intro-quadratic-HPD} above. 
Finally, in~\S\ref{section:applications} we establish the applications 
discussed in~\S\ref{subsection-intro-applications} and~\S\ref{subsection-intro-other-applications}. 
In Appendix~\ref{appendix}, we prove some results in the context of HPD that are used in the paper. 

\subsection{Acknowledgements} 
We are grateful to Roland Abuaf and Johan de Jong for useful conversations.


\section{Preliminaries on HPD}
\label{section:preliminaries}

In this section, we discuss preliminary material on HPD that will be needed 
in the rest of the paper. 
We fix a vector bundle $V$ over our base scheme $S$. 
We denote by $N$ the rank of $V$ and by 
$H$ the relative hyperplane class on the projective bundle $\bP(V)$ 
such that~$\cO(H) = \cO_{\bP(V)}(1)$. 

\subsection{Lefschetz categories} 
\label{subsection-lef-cats}

The fundamental objects of HPD are Lefschetz categories. 
We summarize the basic definitions following \cite[\S6]{NCHPD}. 

\begin{definition}
\label{definition:lefschetz-category}
Let $T$ be a scheme over $S$ with a line bundle $\cL$.
Let $\cA$ be a $T$-linear category. 
An admissible
{$S$-linear} subcategory $\cA_0 \subset \cA$ is called a \emph{Lefschetz center} of $\cA$ with respect to~$\cL$ 
if the subcategories $\cA_i \subset \cA$, $i \in \bZ$, determined by 
\begin{align}
\label{Ai-igeq0}
\cA_i & = \cA_{i-1} \cap {}^\perp(\cA_0 \otimes \cL^{-i})\hphantom{{}^\perp}, \quad 
i \ge 1 \\ 
\label{Ai-ileq0}
\cA_i & = \cA_{i+1} \cap \hphantom{{}^\perp}(\cA_0 \otimes \cL^{-i})^\perp, \quad 
i \le -1 
\end{align} 
are right admissible in $\cA$ for $i \geq 1$, left admissible in $\cA$ for $i \leq -1$,  
vanish for all $i$ of sufficiently large absolute value, say for $|i| \geq m$, and 
provide {$S$-linear} semiorthogonal decompositions 
\begin{align}
\label{eq:right-decomposition}
\cA & = \langle \cA_0, \cA_1 \otimes \cL, \dots, \cA_{m-1} \otimes \cL^{m-1} \rangle, \\ 
\label{eq:left-decomposition}
\cA & = \langle \cA_{1-m} \otimes \cL^{1-m}, \dots, \cA_{-1} \otimes \cL^{-1}, \cA_0 \rangle. 
\end{align} 
The categories $\cA_i$, $i \in \bZ$, are called the \emph{Lefschetz components} of the Lefschetz center $\cA_0 \subset \cA$. 
The semiorthogonal decompositions~\eqref{eq:right-decomposition} 
and~\eqref{eq:left-decomposition} are called \emph{the right Lefschetz decomposition} and 
\emph{the left Lefschetz decomposition} of $\cA$. 
The minimal $m$ above is called the \emph{length} of the Lefschetz decompositions. 
\end{definition}

The Lefschetz components form two (different in general) chains of admissible subcategories
\begin{equation}
\label{eq:lefschetz-chain}
0 \subset \cA_{1-m} \subset \dots \subset \cA_{-1} \subset \cA_0 \supset \cA_1 \supset \dots \supset \cA_{m-1} \supset 0.
\end{equation} 
Note that the assumption of right or left admissibility of $\cA_i$ in $\cA$ 
is equivalent to the assumption of right or left admissibility in $\cA_0$. 

\begin{definition} 
\label{definition-lc} 
A \emph{Lefschetz category} $\cA$ over $\bP(V)$ is a 
$\bP(V)$-linear category equipped with a Lefschetz center $\cA_0 \subset \cA$ with respect to $\cO(H)$. 
The \emph{length} of $\cA$ is the length of its Lefschetz decompositions, and 
is denoted by $\length(\cA)$. 

Given Lefschetz categories $\cA$ and $\cB$ over $\bP(V)$, an \emph{equivalence of Lefschetz categories} or a \emph{Lefschetz equivalence}  
is a $\bP(V)$-linear equivalence $\cA \simeq \cB$ which induces an $S$-linear 
equivalence $\cA_0 \simeq \cB_0$ of centers.  
\end{definition}

\begin{remark}
By \cite[Lemma 6.3]{NCHPD}, if the subcategories $\cA_i \subset \cA$ are admissible 
for all $i \geq 0$ or all $i \leq 0$, then the length $m$ defined above satisfies 
\begin{equation*}
m = \min \set{ i \geq 0 \st \cA_{i} = 0 } = \min \set{ i \geq 0 \st \cA_{-i} = 0 } . 
\end{equation*}
\end{remark} 
 
\begin{remark}
\label{remark:lc-sp}
If $\cA$ is smooth and proper over $S$, then in order for a subcategory $\cA_0 \subset \cA$ to be a Lefschetz center, 
it is enough to give only one of the semiorthogonal decompositions~\eqref{eq:right-decomposition} or~\eqref{eq:left-decomposition}. 
This follows from \cite[Lemmas 4.15 and 6.3]{NCHPD}.
\end{remark}

For $i \ge 1$ the $i$-th \emph{right primitive component}~$\fa_i$ of a Lefschetz center is 
defined as the right orthogonal to $\cA_{i+1}$ in $\cA_i$, i.e.  
\begin{equation*}
\fa_i = \cA_{i+1}^\perp \cap \cA_{i},
\end{equation*}
so that 
\begin{equation}
\label{eq:ca-fa-ca-plus}
\cA_i = \llangle \fa_i, \cA_{i+1} \rrangle =  \llangle \fa_i, \fa_{i+1}, \dots, \fa_{m-1} \rrangle. 
\end{equation}
Similarly, for $i \le {-1}$ the $i$-th \emph{left primitive component} $\fa_i$ of a Lefschetz center is 
the left orthogonal to $\cA_{i-1}$ in $\cA_i$, i.e. 
\begin{equation*}
\fa_i = {}^\perp\cA_{i-1} \cap \cA_{i}, 
\end{equation*}
so that 
\begin{equation}
\label{eq:ca-fa-ca-minus}
\cA_i = \llangle \cA_{i-1}, \fa_{i} \rrangle = \llangle \fa_{1-m}, \dots, \fa_{i-1}, \fa_{i} \rrangle. 
\end{equation}
For $i = 0$, we have both right and left primitive components, defined by  
\begin{equation*}
\fa_{+0} = \cA_1^{\perp} \cap \cA_0 \quad \text{and} \quad 
\fa_{-0} = {}^\perp\cA_{-1} \cap \cA_0,
\end{equation*}
and then~\eqref{eq:ca-fa-ca-plus} and~\eqref{eq:ca-fa-ca-minus} hold true for $i = 0$ 
with $\fa_{+0}$ taking the place of $\fa_0$ for the first and~$\fa_{-0}$ for the second.

For HPD we will need to consider Lefschetz categories that satisfy 
certain ``strongness'' and ``moderateness'' conditions, defined below. 

\begin{definition}
\label{definition:strong}
A Lefschetz category $\cA$ is called \emph{right strong} if all of its right primitive components 
$\fa_{+0}, \fa_i$, $i \geq 1$, are admissible in $\cA$, 
\emph{left strong} if all of its left primitive components~$\fa_{-0}, \fa_{i}$, $i \leq -1$, are admissible in $\cA$, 
and \emph{strong} if all of its primitive components are admissible. 
\end{definition} 

\begin{remark}
\label{remark:smooth-strong}
If $\cA$ is smooth and proper over $S$, then any Lefschetz structure on $\cA$ is automatically strong, 
see \cite[Remark 6.7]{NCHPD}. 
\end{remark}

By \cite[Corollary 6.19(1)]{NCHPD}, the length of a Lefschetz category $\cA$ over $\bP(V)$ satisfies 
\begin{equation}
\label{length-leq-rank}
\length(\cA) \leq \rank(V). 
\end{equation}

\begin{definition}
\label{def:moderateness}
A Lefschetz category $\cA$ over $\bP(V)$ is called \emph{moderate} if its length satisfies the strict inequality
\begin{equation*}
\length(\cA) < \rank(V) . 
\end{equation*} 
\end{definition}

Moderateness of a Lefschetz category $\cA$ is a very mild condition, see~\cite[Remark~2.12]{categorical-joins}. 

There are many examples of interesting Lefschetz categories, see~\cite{kuznetsov2014semiorthogonal} 
for a survey; the most basic is the following. 

\begin{example} 
\label{example-projective-bundle-lc}
Let $0 \ne W \subset V$ be a subbundle of rank $m > 0$. 
The morphism $\bP(W) \to \bP(V)$ induces a $\bP(V)$-linear structure on $\Perf(\bP(W))$. 
Pullback along the projection $\bP(W) \to S$ gives an embedding $\Perf(S) \subset \Perf(\bP(W))$;
its image is a Lefschetz center in~$\Perf(\bP(W))$ and provides it with the structure of a strong Lefschetz category over $\bP(V)$.
The corresponding right and left Lefschetz decompositions are 
given by Orlov's projective bundle formulas: 
\begin{align*} 
\Perf(\bP(W)) & = \llangle \Perf(S), \Perf(S)(H), \dots, \Perf(S)((m-1)H)  \rrangle , \\
\Perf(\bP(W)) & = \llangle \Perf(S)((1-m)H), \dots, \Perf(S)(-H), \Perf(S) \rrangle . 
\end{align*} 
We call this the \emph{standard Lefschetz structure} on~$\bP(W)$. 
Note that the length of $\Perf(\bP(W))$ is $m$, so it is a moderate Lefschetz category as long as~$W \neq V$. 
\end{example}

\subsection{The HPD category} 
\label{subsection-HPD}
Let $H'$ denote the relative hyperplane class on $\bP(V^{\svee})$ 
such that~$\cO(H') = \cO_{\bP(\vV)}(1)$. 
Let 
\begin{equation*}
\delta \colon \bH(\bP(V)) \to \bP(V) \stimes \bP(\vV). 
\end{equation*} 
be the natural incidence divisor.
We think of $\bH(\bP(V))$ as the universal hyperplane in $\bP(V)$.
If $X$ is a scheme with a morphism $X \to \bP(V)$, 
then the universal hyperplane section of $X$ is defined by 
\begin{equation*}
\bH(X) = X \times_{\bP(V)} \bH(\bP(V)). 
\end{equation*} 
This definition extends directly to linear categories as follows.  
 
\begin{definition}
Let $\cA$ be a $\bP(V)$-linear category. 
The \emph{universal hyperplane section} of $\cA$ is defined by 
\begin{equation*}
\bH(\cA) = \cA \otimes_{\Perf(\bP(V))} \Perf(\bH(\bP(V))). 
\end{equation*} 
\end{definition} 

We sometimes use the more elaborate notation 
\begin{equation*}
\bH(X/\bP(V)) = \bH(X) 
\qquad\text{and}\qquad 
\bH(\cA/\bP(V)) = \bH(\cA) 
\end{equation*}
to emphasize the universal hyperplane section is being taken with respect to $\bP(V)$.

There is a commutative diagram 
\begin{equation}
\label{eq:h-diagram}
\vcenter{\xymatrix@C=6em{
& \bH(\bP(V)) \ar[dl]_\pi \ar[d]^\delta \ar[dr]^h 
\\
\bP(V) &
\bP(V) \times \bP(\vV) \ar[l]^-{\pr_1} \ar[r]_-{\pr_2} & 
\bP(\vV).
}}
\end{equation}
Here we follow the notation of~\cite[\S2.2]{categorical-joins} and
deviate slightly from the notation of \cite{NCHPD}, where the morphisms $\pi$, $\delta$, 
and $h$ are instead denoted $p$, $\iota$, and $f$. 
For a $\bP(V)$-linear category $\cA$
there are canonical identifications 
\begin{equation*} 
\cA \otimes_{\Perf(\bP(V))} \Perf(\bP(V) \times \bP(\vV)) \simeq 
\cA \sotimes \Perf(\bP(\vV)), 
\quad
\cA \otimes_{\Perf(\bP(V))} \Perf(\bP(V)) \simeq \cA, 
\end{equation*} 
by which we will regard the functors induced by morphisms in~\eqref{eq:h-diagram} as functors 
\begin{equation*}
\delta_* \colon \bH(\cA) \to \cA \sotimes \Perf(\bP(\vV)),
\quad 
\pi_* \colon \bH(\cA) \to \cA,
\end{equation*}
and so on.
The following definition differs from the original  
in~\cite{kuznetsov-hpd}, but is equivalent to it by~\cite[Lemma~2.22]{categorical-joins}. 

\begin{definition} 
\label{definition-HPD-category}
Let $\cA$ be a Lefschetz category over $\bP(V)$.
Then the \emph{HPD category} $\cAd$ of $\cA$ is the full $\bP(\vV)$-linear subcategory of 
$\bH(\cA)$ defined by 
\begin{equation} 
\label{eq:hpd-category}
\cAd = \set{ C \in \bH(\cA) \st \delta_*(C) \in \cA_0 \sotimes \Perf(\bP(\vV)) }.
\end{equation}
We sometimes use the notation 
\begin{equation*}
(\cA/\bP(V))^{\hpd}  = \cA^{\hpd}
\end{equation*}
to emphasize the dependence on the $\bP(V)$-linear structure. 
\end{definition}

\begin{remark}
The HPD category $\cAd$ depends on the choice of the Lefschetz center $\cA_0 \subset \cA$, 
although this is suppressed in the notation. 
For instance, for the ``stupid'' Lefschetz center~$\cA_0 = \cA$ we have $\cAd = \bH(\cA)$. 
\end{remark}

A less trivial example of HPD is the following.

\begin{example}
\label{ex:categorical-linear-hpd}
Consider the Lefschetz category $\Perf(\bP(W))$ of Example~\ref{example-projective-bundle-lc}
and assume \mbox{$0 \subsetneq W \subsetneq V$}.
Then by ~\cite[Corollary~8.3]{kuznetsov-hpd} 
there is a Lefschetz equivalence
\begin{equation*}
\Perf(\bP(W))^\hpd \simeq \Perf(\bP(W^\perp)).
\end{equation*}
This is usually referred to as \emph{linear} HPD.
\end{example}

If $\cA$ is a Lefschetz category over $\bP(V)$ of length $m$, 
there is a $\bP(\vV)$-linear semiorthogonal decomposition 
\begin{equation}
\label{HC-sod}
\bH(\cA) = \llangle \cAd, 
\delta^*(\cA_1(H) \sotimes \Perf(\bP(\vV))), 
\dots, 
\delta^*(\cA_{m-1}((m-1)H) \sotimes \Perf(\bP(\vV))) \rrangle . 
\end{equation} 
Moreover, $\cAd$ is an admissible subcategory in $\bH(\cA)$, i.e. 
its inclusion functor 
\begin{equation}
\label{def:gamma}
\gamma \colon \cAd \to \bH(\cA) 
\end{equation} 
has both left and right adjoints $\gamma^*,\gamma^! \colon \bH(\cA) \to \cAd$. 
Further, if $\cA$ is a right strong, moderate Lefschetz category, 
then $\cA^{\hpd}$ is equipped with a natural left strong, moderate Lefschetz structure over $\bP(\vV)$ 
with center~$\cAd_0 = \gamma^*\pi^*(\cA_0)$, see~\cite[Theorem~8.7]{NCHPD}.

\begin{remark}
\label{remark:left-hpd}
When $\cA$ is smooth and proper, the HPD operation is an involution;
in other words, the double dual category $\cA^{\hpd\hpd}$ is naturally Lefschetz equivalent to $\cA$.
In a more general situation, the inverse operation to HPD duality is called ``left HPD'', see~\cite[Definition~7.1]{NCHPD}.
The left HPD category ${}^\hpd\cA$ is defined analogously to Definition~\ref{definition-HPD-category},
one just needs to replace the right adjoint functor $\delta_*$ of $\delta^*$ 
by its left adjoint $\delta_!$ in~\eqref{eq:hpd-category}, see~\cite[(7.4)]{NCHPD}.
Alternatively, one can replace in~\eqref{HC-sod} the right orthogonal to the components coming from the ambient variety
by the left orthogonal, see~\cite[(7.2)]{NCHPD}.
Then there are natural Lefschetz equivalences
\begin{equation*}
{}^\hpd(\cA^\hpd) \simeq \cA \simeq ({}^\hpd\cA)^\hpd. 
\end{equation*}
See~\cite[Theorem~8.9]{NCHPD} for the first equivalence; the second is analogous.
In particular, these equivalences imply that showing a Lefshetz equivalence $\cA^\hpd \simeq \cB$ is equivalent 
to showing $\cA \simeq {}^\hpd\cB$.
We will use this observation several times in the paper.
\end{remark}

\subsection{Categorical joins}
\label{subsection:categorical-joins}

In this section, we summarize some of our results on categorical joins from \cite{categorical-joins}. 
Let $V_1$ and $V_2$ be vector bundles on $S$.
Denote by $H_i$ the relative hyperplane class of $\bP(V_i)$ such that $\cO(H_i) \cong \cO_{\bP(V_i)}(1)$.

The \emph{universal resolved join} is defined as the $\bP^1$-bundle 
\begin{equation}
\label{eq:universal-resolved-join}
\tJ(\bP(V_1), \bP(V_2)) = \bP_{\bP(V_1) \times \bP(V_2)}(\cO(-H_1) \oplus \cO(-H_2)). 
\end{equation}
The canonical embedding of vector bundles on $\tJ(\bP(V_1), \bP(V_2))$
\begin{equation*}
\cO(-H_1) \oplus \cO(-H_2) \hookrightarrow (V_1 \otimes \cO) \oplus (V_2 \otimes \cO) = 
(V_1 \oplus V_2) \otimes \cO 
\end{equation*} 
induces a morphism 
\begin{equation*}
f \colon \tJ(\bP(V_1), \bP(V_2)) \to \bP(V_1 \oplus V_2) 
\end{equation*}
which can be identified with a blowup along $\bP(V_1) \sqcup \bP(V_2) \subset \bP(V_1 \oplus V_2)$
with exceptional divisors
\begin{align*}
\bE_1  & = \bP_{\bP(V_1) \stimes \bP(V_2)}(\cO(-H_1)) \cong \bP(V_1) \stimes \bP(V_2) \stackrel{\eps_1} \hookrightarrow \tJ(\bP(V_1), \bP(V_2)), \\  
\bE_2  & = \bP_{\bP(V_1) \stimes \bP(V_2)}(\cO(-H_2)) \cong \bP(V_1) \stimes \bP(V_2) \stackrel{\eps_2} \hookrightarrow \tJ(\bP(V_1), \bP(V_2)). 
\end{align*} 
This situation is summarized in the following commutative diagram
\begin{equation}
\label{diagram-tJ-blowup}
\vcenter{\xymatrix{
& \bP(V_1) \stimes \bP(V_2)
\\
\bE_1 \ar[r]_-{\eps_1} \ar[d] \bijartop[ur] & 
\tJ(\bP(V_1), \bP(V_2)) \ar[d]_f \ar[u]_p & 
\bE_2 \ar[l]^-{\eps_2} \ar[d] \bijarbottom[ul] 
\\
\bP(V_1) \ar[r]^-{} & 
\bP(V_1 \oplus V_2) & 
\bP(V_2) \ar[l]_-{}
}}
\end{equation} 
where $p$ is the canonical projection morphism.

\begin{definition}
\label{definition-tJ}
Let $\cA^1$ be a $\bP(V_1)$-linear category and $\cA^2$ a $\bP(V_2)$-linear category. 
The \emph{resolved join} of $\cA^1$ and $\cA^2$ is the category 
\begin{equation*}
\tJ(\cA^1, \cA^2) = 
\left( \cA^1 \sotimes \cA^2 \right) \otimes_{\Perf(\bP(V_1) \stimes \bP(V_2))} \Perf(\tJ(\bP(V_1), \bP(V_2))). 
\end{equation*}
Further, for $k=1,2$, we define 
\begin{equation*}
\bE_k(\cA^1, \cA^2) = \left( \cA^1 \sotimes \cA^2 \right) \otimes_{\Perf(\bP(V_1) \stimes \bP(V_2))} \Perf(\bE_k) \cong \cA^1 \sotimes \cA^2 .
\end{equation*}
\end{definition} 

We define the categorical join of Lefschetz categories over $\bP(V_1)$ and $\bP(V_2)$ 
as a certain subcategory of the resolved join. 

\begin{definition}
\label{definition-cat-join} 
Let $\cA^1$ and $\cA^2$ be Lefschetz categories over $\bP(V_1)$ and $\bP(V_2)$ 
with Lefschetz centers $\cA^1_0$ and~$\cA^2_0$. 
The \emph{categorical join} $\cJ(\cA^1,\cA^2)$ of $\cA^1$ and $\cA^2$ is defined by
\begin{equation*}
\cJ(\cA^1,\cA^2) = \left\{ C \in \tJ(\cA^1, \cA^2) \ \left|\ 
\begin{aligned}
\eps_1^*(C) &\in \cA^1 \sotimes \cA^2_0 \subset \bE_1(\cA^1, \cA^2) , \\
\eps_2^*(C) &\in \cA^1_0 \sotimes \cA^2 \subset \bE_2(\cA^1, \cA^2)  
\end{aligned}
\right.\right\}.
\end{equation*} 
\end{definition} 

The categorical join is an admissible subcategory in the resolved join; 
its orthogonal complements are supported on the exceptional divisors~$\bE_k$
and can be explicitly described in terms of Lefschetz components of~$\cA^1$ and~$\cA^2$, see~\cite[Lemma~3.12]{categorical-joins}.
Furthermore, $\cJ(\cA^1,\cA^2)$ is smooth and proper as soon as both $\cA^1$ and $\cA^2$ are \cite[Lemma~3.14]{categorical-joins}.
Note also that the categorical join depends on the choice of Lefschetz centers for $\cA^1$ and $\cA^2$, 
although this is suppressed in the notation. 
Finally, by \cite[Theorem 3.21]{categorical-joins}, 
$\cJ(\cA^1,\cA^2)$ has a natural Lefschetz structure with center
\begin{equation}
\label{eq:cat-join-center}
\cJ(\cA^1, \cA^2)_0 = 
p^* {\left(\cA^1_0 \sotimes \cA^2_0 \right)} \subset \cJ(\cA_1,\cA_2). 
\end{equation} 
It is right or left strong if both $\cA^1$ and $\cA^2$ are, 
its length is equal to $\length(\cA^1) + \length(\cA^2)$,
and its Lefschetz and primitive components can be explicitly described, see~\cite[(3.14), (3.15), (3.16), and Lemma~3.24]{categorical-joins}.
The main property of categorical joins is that they commute with the HPD in the following sense.

\begin{theorem}[{\cite[Theorem~4.1]{categorical-joins}}]
\label{theorem-joins-HPD}
Let $\cA^1$ and $\cA^2$ be right strong, moderate Lefschetz categories over $\bP(V_1)$ and~$\bP(V_2)$.  
Then there is an equivalence 
\begin{equation*}
\cJ(\cA^1, \cA^2)^{\hpd} \simeq \cJ((\cA^1)^{\hpd}, (\cA^2)^{\hpd}) 
\end{equation*}
of Lefschetz categories over $\bP(\vV_1 \oplus \vV_2)$. 
\end{theorem}

By~\cite[Proposition~3.17]{categorical-joins} the fiber product of~$\cJ(\cA^1,\cA^2)$ with any $\bP(V_1 \oplus V_2)$-linear category
supported over the complement of~$\bP(V_1) \sqcup \bP(V_2)$ is equivalent to the fiber product
of the resolved join with the same category.
On the other hand, if $\xi \colon V_1 \xrightarrow{\, \sim \,} V_2$ is an isomorphism, 
the graph of $\xi$ is contained in the complement of~$\bP(V_1) \sqcup \bP(V_2)$ and its
fiber product with~$\tJ(\cA_1,\cA_2)$ is equivalent to $\cA_1 \otimes_{\bP(V_1)} \cA_2$.
A combination of this observation with Theorem~\ref{theorem-joins-HPD} and the main theorem of HPD gives the following. 

\begin{theorem}[{\cite[Theorem 5.5]{categorical-joins}}]
\label{theorem-nonlinear-HPD} 
Let $\cA^1$ and $\cA^2$ be right strong, moderate Lefschetz categories over projective bundles $\bP(V_1)$ and~$\bP(V_2)$,
where $V_1$ and $V_2$ have the same rank 
\begin{equation*}
r = \rank(V_1) = \rank(V_2). 
\end{equation*} 
Let $W$ be a vector bundle on $S$ equipped with isomorphisms $\xi_{k} \colon W \xrightarrow{\, \sim \,} V_k$ for $k =1,2$,
and let $(\xi_{k}^{\svee})^{-1} \colon W^{\svee} \xrightarrow{\, \sim \,}  V_k^{\svee}$ be the inverse dual isomorphisms. 
Set 
\begin{equation*}
m = \length(\cA^1) + \length(\cA^2)
\qquad\text{and}\qquad 
m^\hpd = \length((\cA^1)^{\hpd}) + \length((\cA^2)^{\hpd}). 
\end{equation*}
For $i , j \in \bZ$ let $\cJ_i$ and $\cJd_j$ be the Lefschetz components 
of~$\cJ(\cA^1,\cA^2)$ and~$\cJ((\cA^1)^{\hpd},(\cA^2)^{\hpd})$ respectively. 
Denote by $H$ and $H'$ the relative hyperplane classes on~$\bP(W)$ and $\bP(W^{\svee})$.
Then there are semiorthogonal decompositions 
\begin{align}
\cA^1_{\bP(W)} \otimes_{\Perf(\bP(W))} \cA^2_{\bP(W)} & = \notag\\
\label{CQ} & \hspace{-6em} 
\llangle \cK_W(\cA^1, \cA^2), \cJ_r(H), \dots, \cJ_{m-1}((m-r)H) \rrangle ,    \\ 
(\cA^1)^{\hpd}_{\bP(W^{\svee})} {\otimes_{\Perf(\bP(W^{\svee}))}} (\cA^2)^{\hpd}_{\bP(W^{\svee})} & = \notag \\
\label{DQ} & \hspace{-6em} 
\llangle \cJd_{1-{m^\hpd}}((r-{m^\hpd})H'), \dots, \cJd_{-r}(-H'), \cK'_{W^{\svee}}((\cA^1)^{\hpd}, (\cA^2)^{\hpd}) \rrangle, 
\end{align} 
and an $S$-linear equivalence 
\begin{equation*}
\cK_W(\cA^1, \cA^2) \simeq  \cK'_{W^{\svee}}((\cA^1)^{\hpd}, (\cA^2)^{\hpd}) .  
\end{equation*}
\end{theorem}

In the case where $\cA^2$ and $(\cA^2)^\hpd$ is an HPD pair from Example~\ref{ex:categorical-linear-hpd}, 
this recovers the main theorem of HPD.

\subsection{Categorical resolutions}
\label{subsection:categorical-resolutions}
Finally, we recall the notion of a categorical resolution of singularities developed in~\cite{kuznetsov2008lefschetz}, 
which we will need later. 

\begin{definition}
\label{def:categorical-resolution}
Given a projective variety $Y$ over a field~$\bk$, a $\bk$-linear 
category $\cC$ is called a \emph{categorical resolution} of $Y$ if $\cC$ is smooth and proper
and there exists a pair of functors 
\begin{equation*}
\pi_* \colon \cC \to \Db(Y) \qquad \text{and} \qquad \pi^* \colon \Perf(Y) \to \cC, 
\end{equation*} 
such that $\pi^*$ is left adjoint to $\pi_*$, 
and $\pi^*$ is fully faithful. 
If further $\pi^*$ is right adjoint to $\pi_*$ on $\Perf(Y)$, then the categorical resolution $\cC$ is called \emph{{weakly} crepant}. 
\end{definition}

If $\pi \colon X \to Y$ is a morphism from a smooth proper scheme such that $\pi_*\cO_X \cong \cO_Y$
(for instance, if $Y$ has rational singularities and $\pi$ is a resolution of singularities)
then the pullback and the pushforward functors provide $\Db(X) = \Perf(X)$ with a structure of a categorical resolution of $Y$
(and if the singularities of $Y$ are worse than rational, a categorical resolution of $Y$ was constructed in~\cite{kuznetsov-lunts}).
Such resolution is weakly crepant if and only if $K_{X/Y} = 0$, i.e. if and only if the morphism $\pi$ is crepant.

\begin{remark}
\label{remark-nc-cat-res}
The notion of a categorical resolution and weak crepancy extends to the noncommutative case 
where $Y$ is replaced with an admissible subcategory $\cA \subset \Db(Y)$; 
then~$\cA^{\perf} = \cA \cap \Perf(Y)$ plays the role of $\Perf(Y)$. 
In this case, we say that $\cC$ is a categorical resolution (or weakly crepant categorical resolution) of $\cA^{\perf}$.
\end{remark}


\section{Categorical cones}
\label{section-categorical-cones}

In this section, we introduce the operation of taking the categorical cone of a Lefschetz category. 
This operation is closely related to that of a categorical join reviewed in~\S\ref{subsection:categorical-joins};
in fact, in \S\ref{subsection-cJ-cC-comparison} we show that under a splitting assumption, 
categorical cones can be described in terms of categorical joins. 

We fix an exact sequence 
\begin{equation}
\label{extension-v0-v}
0 \to V_0 \to V \to \barV \to 0 
\end{equation}
of vector bundles on $S$. 
We write $H_0$, $H$, and $\barH$ for the relative hyperplane classes 
on the projective bundles $\bP(V_0)$, $\bP(V)$, and $\bP(\barV)$, and denote by $N_0$ the rank of $V_0$.

\subsection{Resolved cones}
\label{subsection-resolved-cones}

Let $\cV$ be the vector bundle on $\bP(\barV)$ defined as the preimage 
of the line subbundle~$\cO(-\barH) \subset \barV \otimes \cO_{\bP(\barV)}$
under the surjection $V \otimes \cO_{\bP(\barV)} \to \barV \otimes \cO_{\bP(\barV)}$, 
so that on~$\bP(\barV)$ we have a commutative diagram 
\begin{equation}
\label{extension-v0-cv}
\vcenter{
\xymatrix{
0 \ar[r] & 
V_0  \otimes \cO_{\bP(\barV)} \ar[r] \ar@{=}[d] &
\cV \ar[r] \ar[d] &
\cO(-\barH) \ar[r] \ar[d] & 
0
\\
0 \ar[r] & 
V_0 \otimes \cO_{\bP(\barV)} \ar[r] &
V  \otimes \cO_{\bP(\barV)} \ar[r] & 
\barV  \otimes \cO_{\bP(\barV)} \ar[r] & 
0
}}
\end{equation}
with exact rows. 
If~\eqref{extension-v0-v} is split then $\cV \cong V_0 \otimes \cO \oplus \cO(-\barH)$.

Now let $X \to \bP(\barV)$ be a morphism of schemes. 
Then the \emph{resolved cone} over $X$ with vertex~$\bP(V_0)$ is defined as 
the projective bundle
\begin{equation}
\label{resolved-cone-X}
\tC_{V_0}(X) = \bP_X(\cV_X),  
\end{equation} 
where $\cV_X$ denotes the pullback of $\cV$ to $X$. 
The embedding $\cV_X \hookrightarrow V \otimes \cO_X$
induced by the middle vertical arrow in~\eqref{extension-v0-cv} gives a morphism
\begin{equation*}
\tC_{V_0}(X) \to \bP(V). 
\end{equation*}
If $X \to \bP(\barV)$ is an embedding, then 
this morphism factors birationally through the classical cone $\bC_{V_0}(X) \subset \bP(V)$, 
and provides a resolution of singularities if $X$ is smooth.  

Note that there is an isomorphism 
\begin{equation}
\label{resolved-cone-X-base-change}
\tC_{V_0}(X) \cong X \times_{\bP(\barV)} \tC_{V_0}(\bP(\barV)). 
\end{equation}
Motivated by this, we call $\tC_{V_0}(\bP(\barV)) = \bP_{\bP(\barV)}(\cV)$ the \emph{universal resolved cone} 
with vertex~$\bP(V_0)$. 
Denote by 
\begin{equation*}
\barp \colon \tC_{V_0}(\bP(\barV)) \to \bP(\barV) 
\end{equation*}
the canonical projection morphism. 
Note that the rank of $\cV$ is $N_0+1$, so $\barp$ is a $\bP^{N_0}$-bundle. 
Further, denote by 
\begin{equation*}
f \colon \tC_{V_0}(\bP(\barV)) \to \bP(V)
\end{equation*}
the morphism induced by the canonical embedding $\cV \hookrightarrow V \otimes \cO_{\bP(\barV)}$ from~\eqref{extension-v0-cv}. 
Define 
\begin{equation*}
\bE = \bP_{\bP(\barV)}(V_0  \otimes \cO_{\bP(\barV)}) \cong \bP(V_0) \times \bP(\barV)
\end{equation*}
and let  
\begin{equation*}
\eps \colon \bE \to \tC_{V_0}(\bP(\barV))
\end{equation*}
be the canonical divisorial embedding induced by the first map in the top row of~\eqref{extension-v0-cv}.
We have a commutative diagram
\begin{equation}
\label{diagram-tC-blowup}
\vcenter{\xymatrix{
\bE \ar[r]^-\eps \ar[d] &
\tC_{V_0}(\bP(\barV)) \ar[d]^f \ar[r]^-{\barp} &
\bP(\barV) \\
\bP(V_0) \ar[r] & 
\bP(V)
}}
\end{equation}
The isomorphism $\bE \cong \bP(V_0) \times \bP(\barV)$ 
is induced by the product of the vertical arrow and $\barp \circ \eps$.

The next result follows easily from the definitions.  

\begin{lemma}
\label{lemma-tC-divisors}
The following hold: 
\begin{enumerate}
\item \label{tC-f-blowup}
The morphism $f \colon \tC_{V_0}(\bP(\barV)) \to \bP(V)$ is the blowup of 
$\bP(V)$ in $\bP(V_0)$, with exceptional divisor $\bE$.  

\item
The $\cO(1)$ line bundle for the projective bundle  
$\barp \colon \tC_{V_0}(\bP(\barV)) \to \bP(\barV)$ is $\cO(H)$. 

\item \label{E-Hk-H}
We have the following equality of divisors modulo linear equivalence: 
\begin{equation*}
\bE = H - \barH,  \qquad H|_{\bE} = H_0.
\end{equation*}

\item \label{omega-barp}
The relative dualizing complex of the morphism $\barp$ is given by 
\begin{equation*}
\omega_{\barp}  \cong \det(\barp^*\cV^{\svee})(-(N_0+1)H)[N_0].
\end{equation*}
\end{enumerate} 
\end{lemma}

Following \eqref{resolved-cone-X-base-change} we define the resolved cone of a category linear over $\bP(\barV)$ 
by base change from the universal resolved cone.  

\begin{definition}
\label{definition-tbC}
Let $\cA$ be a $\bP(\barV)$-linear category. 
The \emph{resolved cone} over $\cA$ with vertex $\bP(V_0)$ 
is the category 
\begin{equation*}
\tC_{V_0}(\cA) = 
\cA \otimes_{\Perf(\bP(\barV))} \Perf(\tC_{V_0}(\bP(\barV))). 
\end{equation*}
Further, we define 
\begin{equation*}
\bE(\cA) = \cA \otimes_{\Perf(\bP(\barV))} \Perf(\bE).
\end{equation*}
\end{definition}

\begin{remark} 
\label{remark-E}
The isomorphism 
$\bE \cong \bP(V_0) \times \bP(\barV)$ induces a canonical equivalence 
\begin{equation*}
\bE(\cA) \simeq \Perf(\bP(V_0)) \sotimes \cA.
\end{equation*} 
We identify these categories via this equivalence; 
in particular, below we will regard subcategories of the right side as subcategories of the left. 
{Furthermore, using this identification the morphism $\eps$ from~\eqref{diagram-tC-blowup}
induces functors between $\Perf(\bP(V_0)) \sotimes \cA$ and $\tC_{V_0}(\cA)$.}
\end{remark}

\begin{remark}
\label{remark-tbC-spaces}
If $X$ is a scheme over $\bP(\barV)$, 
then by the isomorphism~\eqref{resolved-cone-X-base-change} and \cite[Theorem~1.2]{bzfn} 
the resolved cone satisfies 
\begin{equation*}
\tC_{V_0}(\Perf(X)) \simeq \Perf(\tC_{V_0}(X)). 
\end{equation*} 
\end{remark}

\begin{remark}
\label{remark-tC-functorial}
Resolved cones are functorial. 
Namely, given a $\bP(\barV)$-linear functor $\gamma \colon \cA \to \cB$, 
its base change along $\tC_{V_0}(\bP(\barV)) \to \bP(\barV)$ gives a $\bP(V)$-linear functor 
\begin{equation*}
\tC_{V_0}(\gamma) \colon \tC_{V_0}(\cA) \to \tC_{V_0}(\cB). 
\end{equation*} 
Moreover, 
if $\gamma^* \colon \cB \to \cA$ is a left
adjoint functor to $\gamma$, then $\tC_{V_0}(\gamma^*)$ is left
adjoint to $\tC_{V_0}(\gamma)$, and similarly for right adjoints, see \cite[Lemma 2.12]{NCHPD}.
\end{remark}

\subsection{Categorical cones} 
\label{subsection-categorical-cones}

We define the categorical cone of a Lefschetz category over $\bP(\barV)$ as a certain subcategory of the resolved cone, 
similarly to Definition~\ref{definition-cat-join} of a categorical join.

\begin{definition}
\label{definition-cat-cone} 
Let $\cA$ be a Lefschetz category over $\bP(\barV)$ with Lefschetz center $\cA_0$. 
The \emph{categorical cone} $\cC_{V_0}(\cA)$ over~$\cA$ with vertex $\bP(V_0)$ is the subcategory of $\tC_{V_0}(\cA)$ defined by
\begin{equation*}
\cC_{V_0}(\cA) = \left\{ C \in \tC_{V_0}(\cA) \ \left|\ 
\eps^*(C) \in \Perf(\bP(V_0)) \sotimes \cA_0 \subset \bE(\cA)  
\right.\right\}.
\end{equation*}
Here, we have used the identification of Remark~\ref{remark-E}. 
If $\cA = \Perf(X)$ for a scheme $X$ over~$\bP(\barV)$, we abbreviate notation by writing 
\begin{equation*}
\cC_{V_0}(X) = \cC_{V_0}(\Perf(X)). 
\end{equation*} 
\end{definition} 

\begin{remark}
The categorical cone depends on the choice of a Lefschetz center for~$\cA$, although this is suppressed in the notation. 
For instance, for the ``stupid'' Lefschetz center~$\cA_0 = \cA$, the condition in the definition is void, 
so $\cC_{V_0}(\cA) = \tC_{V_0}(\cA)$. 
\end{remark} 

We note that if $V_0 = 0$, then taking the categorical cone does nothing:

\begin{lemma}
\label{lemma:cone-0}
Let $\cA$ be a Lefschetz category over $\bP(\barV)$. 
If $V_0 = 0$ then $\cC_{V_0}(\cA) \simeq \cA$.
\end{lemma}

\begin{proof}
If $V_0 = 0$ then $\cV \cong \cO_{\bP(\barV)}(-\barH)$ by~\eqref{extension-v0-cv}, 
hence $\tC_{V_0}(\bP(\barV)) = \bP_{\bP(\barV)}(\cV) \cong \bP(\barV)$ and~$\tC_{V_0}(\cA) \simeq \cA$. 
Furthermore, the divisor $\bE$ is empty in this case, hence the defining condition of $\cC_{V_0}(\cA) \subset \tC_{V_0}(\cA)$ is void and $\cC_{V_0}(\cA) = \tC_{V_0}(\cA)$.
\end{proof}

\begin{lemma}
\label{lemma-tC}
Let $\cA$ be a Lefschetz category over $\bP(\barV)$ of length $m$. 
Then the categorical cone $\cC_{V_0}(\cA)$ is an admissible $\bP(V)$-linear 
subcategory of $\tC_{V_0}(\cA)$, and there are $\bP(V)$-linear semiorthogonal decompositions 
\begin{align}
\label{sod-tC}
\tC_{V_0}(\cA) = \Big \langle & \notag
\cC_{V_0}(\cA),  \\
&
\eps_{!}{\left( \Perf(\bP(V_0)) \sotimes \cA_{1}(\barH) \right)}, \dots, 
\eps_{!}{\left( \Perf(\bP(V_0)) \sotimes \cA_{m-1}((m-1)\barH) \right)}  \Big\rangle,
\\
\label{sod-tC-other}
\tC_{V_0}(\cA) = \Big\langle & \notag
\eps_{*}\left( \Perf(\bP(V_0)) \sotimes \cA_{1-m}((1-m)\barH)   \right), \dots, 
\eps_{*}\left( \Perf(\bP(V_0)) \sotimes \cA_{-1}(-\barH) \right),   
\\
& \cC_{V_0}(\cA) 
\Big\rangle, 
\end{align} 
where $\eps_!$ denotes the left adjoint of $\eps_*$. 
\end{lemma}

\begin{proof}
Apply \cite[Proposition 3.11]{categorical-joins} with $T = \bP(\barV)$, $Y = \tC_{V_0}(\bP(\barV))$, and $E = \bE$. 
Then in the notation of that proposition, we have $\cA_Y = \tC_{V_0}(\cA)$ and $\cA_E = \bE(\cA) = \Perf(\bP(V_0)) \otimes \cA$, 
and the result follows. 
\end{proof}

\begin{example}
\label{example-cone-PbarV}
Let $\bar{W} \subset \barV$ be a subbundle, so that $\bP(\bar{W}) \subset \bP(\barV)$.
The classical cone over~$\bP(\bar{W})$ with vertex $\bP(V_0)$ is given by 
$\bC_{V_0}(\bP(\bar{W})) = \bP(W)$, where $W \subset V$ is the preimage of~$\bar{W}$ under the epimorphism $V \to \barV$.  
Consider the Lefschetz structure on $\bP(\bar{W})$ defined in Example~\ref{example-projective-bundle-lc}.
Then it follows from Lemma~\ref{lemma-tC-divisors} and Orlov's blowup formula that 
the pullback functor 
$f^* \colon \Perf(\bP(W)) \to \Perf(\tC_{V_0}(\bP(\bar{W})))$ 
induces an equivalence
$\Perf(\bP(W)) \simeq \cC_{V_0}(\bP(\bar{W}))$. 
Further, Theorem~\ref{theorem-cone-lef-cat} below equips $\cC_{V_0}(\bP(\bar{W}))$ with a canonical Lefschetz 
structure, with respect to which this equivalence is easily seen to be a Lefschetz equivalence.
\end{example}

\begin{lemma}
\label{lemma-cC-smooth-proper}
Let $\cA$ be a Lefschetz category over $\bP(\barV)$ which is smooth and proper over $S$. 
Then the categorical cone~$\cC_{V_0}(\cA)$ is smooth and proper over $S$. 
\end{lemma}

\begin{proof}
Being the base change of $\cA$ along the projective bundle $\barp \colon \tC_{V_0}(\bP(\barV)) \to \bP(\barV)$, 
the resolved cone $\tC_{V_0}(\cA)$ is smooth and proper over $S$ by \cite[Lemma 4.11]{NCHPD}. 
Hence the result follows from Lemma~\ref{lemma-tC} and \cite[Lemma 4.15]{NCHPD}. 
\end{proof}

\begin{proposition}
\label{proposition-cC-tC-agree}
Let $\cA$ be a Lefschetz category over $\bP(\barV)$.  
Let $T \to \bP(V)$ be a morphism of schemes which factors through the complement of $\bP(V_0)$ in $\bP(V)$. 
Then there are $T$-linear equivalences
\begin{equation*}
\cC_{V_0}(\cA)_T \simeq \tC_{V_0}(\cA)_T \simeq \cA_T, 
\end{equation*}
where the base change of $\cA$ is taken along the morphism $T \to \bP(\barV)$ 
obtained by composing~$T \to \bP(V)$ with the linear projection from $\bP(V_0) \subset \bP(V)$. 
\end{proposition} 

\begin{proof}
By Lemma~\ref{lemma-tC-divisors}, the morphism $f \colon \tC_{V_0}(\bP(\barV)) \to \bP(V)$ is an isomorphism over 
the complement of $\bP(V_0)$. 
Hence there is an isomorphism $\tC_{V_0}(\bP(\barV))_T \cong T$. 
The equivalence $\tC_{V_0}(\cA)_T \simeq \cA_T$ then follows from the definition of the 
resolved cone. 
Further, the components to the right of $\cC_{V_0}(\cA)$ in~\eqref{sod-tC} are supported 
over $\bP(V_0)$, hence their base change along $T \to \bP(\barV)$ vanish. 
This shows $\cC_{V_0}(\cA)_T \simeq \tC_{V_0}(\cA)_T$. 
\end{proof}

For future use we fix the following immediate corollary of the proposition.

\begin{corollary}
\label{corollary:cone-trivial}
Let $\cA$ be a Lefschetz category over $\bP(\barV)$.  
Let $T \to \bP(V)$ be a morphism of schemes which factors through the complement of $\bP(V_0)$ in $\bP(V)$, 
and such that the composition $T \to \bP(\barV)$ is an isomorphism.
Then there is an equivalence 
\begin{equation*}
\cC_{V_0}(\cA)_T \simeq \cA. 
\end{equation*}
\end{corollary}

\subsection{Relation to categorical joins} 
\label{subsection-cJ-cC-comparison}
In this subsection, we assume $V_0 \neq 0$ and we are given a splitting of~\eqref{extension-v0-v}:
\begin{equation*}
V = V_0 \oplus \barV.
\end{equation*}
Under these assumptions, we relate the cone operations discussed above (classical, resolved, and categorical) 
to taking a join (in the corresponding senses) with $\bP(V_0)$.

The relation between the classical operations is easy: 
if $X \subset \bP(\barV)$ 
is a closed subscheme, then the classical join of $X$ with $\bP(V_0)$ 
coincides with the cone over $X$ with vertex $\bP(V_0)$, i.e. 
\begin{equation*}
\bJ(\bP(V_0), X) = \bC_{V_0}(X) \subset \bP(V). 
\end{equation*}
Note that the assumption $V_0 \neq 0$ is necessary for this equality; 
if $V_0 = 0$ then $\bP(V_0) = \varnothing$ and hence 
$\bJ(\bP(V_0), X) =\varnothing$, while $\bC_{V_0}(X) = X$. 

Next we compare the universal resolved join~\eqref{eq:universal-resolved-join} 
to the universal resolved cone 
\begin{equation*}
\tC_{V_0}(\bP(\barV))  = \bP_{\bP(\barV)}((V_0  \otimes \cO_{\bP(\barV)}) \oplus \cO(-\barH)) . 
\end{equation*}
The natural embedding $\cO(-H_0) \hookrightarrow V_0 \otimes \cO_{\bP(\barV)}$ induces a 
morphism 
\begin{equation*}
\beta \colon \tJ(\bP(V_0), \bP(\barV)) \to \tC_{V_0}(\bP(\barV)). 
\end{equation*}
Denoting $Z = \bP_{\bP(\barV)}(\cO(-\barH)) \cong \bP(\barV)  \subset \tC_{V_0}(\bP(\barV))$, 
the diagram~\eqref{diagram-tJ-blowup} 
(with~$V_1 = V_0$ and~$V_2 = \barV$) and 
the diagram~\eqref{diagram-tC-blowup} merge to a commutative diagram
\begin{equation*}
\xymatrix{
\bE_1 \ar[r] \bijarbottom[d] & \tJ(\bP(V_0), \bP(\barV)) \ar[d]_{\beta} & \ar[l] \bE_2 \ar[d] \\ 
\bE \ar[r] \ar[d] & \tC_{V_0}(\bP(\barV)) \ar[d] & \ar[l] Z \bijarbottom[d] \\ 
\bP(V_0) \ar[r] & \bP(V_0 \oplus \barV) & \ar[l] \bP(\barV)  
}
\end{equation*}
where under the isomorphisms $\bE_2 \cong \bP(V_0) \times \bP(\barV)$ and 
$Z \cong \bP(\barV)$, the map $\bE_2 \to Z$ is identified with the projection. 

\begin{lemma}
\label{lemma-tJ-tC} 
The morphism $\beta \colon \tJ(\bP(V_0),\bP(\barV) ) \to \tC_{V_0}(\bP(\barV))$ is the blowup of 
$\tC_{V_0}(\bP(\barV))$ in $Z$, with exceptional divisor $\bE_2$. 
\end{lemma}

\begin{proof}
Follows from Lemma~\ref{lemma-tC-divisors}\eqref{tC-f-blowup} and \cite[Lemma 3.1(1)]{categorical-joins}. 
\end{proof}

Using this, we can finally compare categorical joins and cones. 
We consider the categorical join $\cJ(\bP(V_0), \cA)$ of $\Perf(\bP(V_0))$ 
(with the standard Lefschetz structure from Example~\ref{example-projective-bundle-lc}) 
and a Lefschetz category $\cA$ over $\bP(\barV)$.

\begin{proposition}
\label{proposition-cJ-cC}
Let $\cA$ be a Lefschetz category over $\bP(\barV)$, and let $V_0$ be a nonzero vector bundle on $S$. 
Then there is an equivalence 
\begin{equation*}
\cC_{V_0}(\cA) \simeq \cJ(\bP(V_0), \cA)
\end{equation*}
of $\bP(V_0 \oplus \bar{V})$-linear categories. 
More precisely, pullback and pushforward along the blowup morphism 
$\beta \colon \tJ(\bP(V_0), \bP(\barV)) \to \tC_{V_0}(\bP(\barV))$ give functors 
\begin{align*}
\beta^* & \colon \tC_{V_0}(\cA) \to \tJ(\bP(V_0), \cA) , \\ 
\beta_* & \colon \tJ(\bP(V_0), \cA) \to \tC_{V_0}(\cA) , 
\end{align*}
which induce mutually inverse equivalences 
between the subcategories 
\begin{equation*}
\cC_{V_0}(\cA) \subset \tC_{V_0}(\cA) \quad \text{and} \quad 
\cJ(\bP(V_0), \cA) \subset \tJ(\bP(V_0), \cA) . 
\end{equation*}
Moreover, for any $\bP(\barV)$-linear functor $\gamma \colon \cA \to \cB$ there {are commutative diagrams}
\begin{equation}
\label{eq:tj-tc-diagram}
\vcenter{\xymatrix@C=4em{
\tJ(\bP(V_0),\cA) \ar[r]^-{\tJ(\id,\gamma)} \ar[d]_{\beta_*} &
\tJ(\bP(V_0),\cB) \ar[d]_{\beta_*}
\\
\tC_{V_0}(\cA) \ar[r]^-{{\tC_{V_0}(\gamma)}} 
&
\tC_{V_0}(\cB)
}}
\quad\text{and}\quad
\vcenter{\xymatrix@C=4em{
\tJ(\bP(V_0),\cA) \ar[r]^-{\tJ(\id,\gamma)} 
&
\tJ(\bP(V_0),\cB) 
\\
\tC_{V_0}(\cA) \ar[r]^-{{\tC_{V_0}(\gamma)}} \ar[u]_{\beta^*} &
\tC_{V_0}(\cB), \ar[u]_{\beta^*}
}}
\end{equation}
where the functor $\tJ(\id,\gamma)$ is defined for resolved joins in the same way as $\tC_{V_0}(\gamma)$ for cones, see~\cite[(3.4)]{categorical-joins}.
\end{proposition}

\begin{proof}
Diagrams~\eqref{eq:tj-tc-diagram} are obtained from the functor $\gamma$ by base change along the morphism~$\beta$.
Lemma~\ref{lemma-tJ-tC} together with 
Orlov's blowup formula implies  
$\beta^* \colon \tC_{V_0}(\cA) \to \tJ(\bP(V_0), \cA)$ is fully faithful 
and gives an equivalence onto the subcategory 
\begin{equation*}
\beta^* (\tC_{V_0}(\cA)) = 
\set{ C \in \tJ(\bP(V_0), \cA) \st \eps_2^*(C) \in \Perf(S) \otimes \cA \subset \bE_2(\Perf(\bP(V_0)), \cA) }
\end{equation*}
with the inverse functor given by $\beta_*$.
Since $\beta$ maps $\bE_1$ isomorphically onto $\bE$, it thus follows from  
Definition~\ref{definition-cat-cone} that $\beta^*$ induces an equivalence from 
$\cC_{V_0}(\cA)$ onto the subcategory 
\begin{equation*}
\beta^*(\cC_{V_0}(\cA)) = 
\left\{ C \in \tJ(\bP(V_0), \cA)   \ \left|\ 
\begin{aligned}
\eps_1^*(C) &\in \Perf(\bP(V_0)) \sotimes \cA_0 &&\subset \bE_1(\Perf(\bP(V_0)), \cA) , \\
\eps_2^*(C) &\in \Perf(S) \otimes \cA &&\subset \bE_2(\Perf(\bP(V_0)), \cA)
\end{aligned}
\right.\right\} , 
\end{equation*}
with the inverse equivalence induced by $\beta_*$. 
But by~Definition~\ref{definition-cat-join} 
this subcategory coincides with $\cJ(\bP(V_0), \cA)$ since $\Perf(S)$ is the Lefschetz center of $\Perf(\bP(V_0))$. 
\end{proof}

\begin{remark} 
Proposition~\ref{proposition-cJ-cC} does not apply if $V_0 = 0$. 
Indeed, if $V_0 = 0$ then $\bP(V_0) = \varnothing$ and hence $\cJ(\bP(V_0), \cA) = 0$, 
while $\cC_{V_0}(\cA) \simeq \cA$ by Lemma~\ref{lemma:cone-0}. 
\end{remark} 

\begin{remark}
In Theorem~\ref{theorem-cone-lef-cat} we will equip any categorical cone with a canonical Lefschetz structure
in such a way that the equivalence $\cC_{V_0}(\cA) \simeq \cJ(\bP(V_0), \cA)$ of Proposition~\ref{proposition-cJ-cC} 
is an equivalence of Lefschetz categories. 
\end{remark}

\begin{remark}
Let $\cA^1$ and $\cA^2$ be Lefschetz categories over $\bP(V_1)$ and~$\bP(V_2)$, where $V_1$ and~$V_2$ are nonzero. 
Then 
there is a $\bP(V_1 \oplus V_2)$-linear equivalence 
\begin{equation*}
\cJ(\cA^1,\cA^2) \simeq \cC_{V_1}(\cA^2) \otimes_{\Perf(\bP(V_1 \oplus V_2))} \cC_{V_2}(\cA^1). 
\end{equation*}
This can be proved either directly, or (in its dual form) by combining 
Corollary~\ref{corollary:cones-HPD-special} below and~\cite[Corollary~7.6]{categorical-joins}. 
The right side can be endowed with a semiorthogonal decomposition 
by an application of \cite[Corollary 5.3]{categorical-joins}, 
which can be shown to be a Lefschetz decomposition compatible with the Lefschetz structure of the left side. 
Note also that the equivalence of Proposition~\ref{proposition-cJ-cC} is a special case of this.
Indeed, take \mbox{$\cA^1 = \bP(V_1)$} and use the equivalence~$\cC_{V_2}(\bP(V_1)) \simeq \Perf(\bP(V_1 \oplus V_2))$ 
of Example~\ref{example-cone-PbarV}. 
We omit further details as we shall not need this.
\end{remark}

\subsection{The Lefschetz structure of a categorical cone} 
\label{subsection-ld-cC}
Our goal in this subsection is to equip any categorical cone with a canonical Lefschetz structure. 

\begin{lemma}
\label{lemma-cone-A0} 
Let $\cA$ be a Lefschetz category over $\bP(\barV)$. 
Then the image of $\cA_0$ under the functor $\barp^* \colon \cA \to \tC_{V_0}(\cA)$
is contained in the categorical cone $\cC_{V_0}(\cA)$. 
If $\cA_i$ are the Lefschetz components of $\cA$, then 
$\barp^*(\cA_i) \subset \cC_{V_0}(\cA)$ is left admissible for $i < 0$, admissible for $i = 0$, and right admissible for $i > 0$.
\end{lemma}

\begin{proof}
Because the morphism $\barp \colon \tC_{V_0}(\bP(\barV)) \to \bP(\barV)$ is a projective bundle, 
the pullback functor 
$\barp^* \colon \Perf(\bP(\barV)) \to \Perf(\tC_{V_0}(\bP(\barV)))$ 
is fully faithful and admits left and right adjoints. 
Thus the same holds for its base change $\barp^* \colon \cA \to \tC_{V_0}(\cA)$ 
(see \cite[Lemma 2.12]{NCHPD}). 
Further, by Definition~\ref{definition-cat-cone} we see that the image of $\cA_0$ 
under $\barp^*$ is contained in $\cC_{V_0}(\cA)$. 
The result follows. 
\end{proof}

\begin{definition}
\label{definition-lef-center-cone}
Let $\cA$ be a Lefschetz category over $\bP(\barV)$. 
For $i \in \bZ$, we define a subcategory~$\cC_{V_0}(\cA)_i \subset \cC_{V_0}(\cA)$ by  
\begin{equation}
\label{ci-cone} 
\cC_{V_0}(\cA)_i  = 
\begin{cases}
\barp^*(\cA_{i+N_0}) & \text{if $i \le -N_0$,}\\
\barp^*(\cA_0) & \text{if $-N_0 \le i \le N_0$,}\\
\barp^*(\cA_{i-N_0}) & \text{if $i \ge N_0$,}
\end{cases}
\end{equation} 
where $N_0$ is the rank of $V_0$. 
\end{definition}

Note that the containment $\cC_{V_0}(\cA)_i  \subset \cC_{V_0}(\cA)$ holds by Lemma~\ref{lemma-cone-A0}. 

\begin{theorem}
\label{theorem-cone-lef-cat}
Let $\cA$ be a Lefschetz category over $\bP(\barV)$. 
Then the categorical cone $\cC_{V_0}(\cA)$ 
has the structure of a Lefschetz category over $\bP(V)$ with 
Lefschetz components $\cC_{V_0}(\cA)_i$ given by~\eqref{ci-cone}. 
If $\cA$ is either right or left strong, then so is $\cC_{V_0}(\cA)$.  
Moreover, we have 
\begin{equation*}
\length(\cC_{V_0}(\cA)) = \length(\cA) + N_0, 
\end{equation*} 
and $\cC_{V_0}(\cA)$ is moderate if and only if $\cA$ is moderate. 
\end{theorem}

One could prove this directly by an analogue of the argument of~\cite[\S3.4]{categorical-joins}.
However, we prefer to reduce to the case of categorical joins using Proposition~\ref{proposition-cJ-cC}
and the local-to-global result of Lemma~\ref{lemma-sod-local}.

\begin{proof}
We may also assume $V_0 \neq 0$, otherwise the result is trivial. 
The key claim is that we have semiorthogonal decompositions 
\begin{align}
\label{sod-cC}
\cC_{V_0}(\cA) & = \llangle \cC_0 , \cC_1(H), \dots, \cC_{m+N_0-1}((m+N_0-1)H) \rrangle, \\ 
\label{sod-cC-left}
\cC_{V_0}(\cA) & = \llangle \cC_{1-m-N_0}((1-m-N_0)H), \dots, \cC_{-1}(-H), \cC_0 \rrangle ,
\end{align} 
where $\cC_i = \cC_{V_0}(\cA)_i$ and $m = \length(\cA)$. 
By Lemma~\ref{lemma-sod-local} --- whose hypotheses are satisfied by Lemma~\ref{lemma-cone-A0} --- 
it is enough to prove~\eqref{sod-cC} and~\eqref{sod-cC-left} after base change to any fpqc cover of $S$.
Therefore, we may assume that we have a splitting $V = V_0 \oplus \barV$ of~\eqref{extension-v0-v}. 
Then Proposition~\ref{proposition-cJ-cC} gives an 
equivalence
\begin{equation*}
\cC_{V_0}(\cA) \simeq \cJ(\bP(V_0), \cA). 
\end{equation*}
By~\cite[Theorem 3.21]{categorical-joins} the categorical join~$\cJ(\bP(V_0), \cA)$ 
has the structure of a Lefschetz category of length $\length(\cA) + N_0$. 
By Example~\ref{example-projective-bundle-lc} the nonzero primitive components $\mathfrak{p}_i$ of $\Perf(\bP(V_0))$ 
are $\mathfrak{p}_{\pm (N_0-1)} = \Perf(S)$, hence the second formula of~\cite[Lemma~3.24]{categorical-joins} shows that
the Lefschetz components $\cJ_i \subset \cJ(\bP(V_0), \cA)$ are equal to 
\begin{equation*}
\cJ_i = 
\begin{cases}
p^*\pr_2^*(\cA_{i+N_0}) & \text{if $i \le -N_0$,}\\
p^*\pr_2^*(\cA_0) & \text{if $-N_0 \le i \le N_0$,}\\
p^*\pr_2^*(\cA_{i-N_0}) & \text{if $i \ge N_0$,}
\end{cases}
\end{equation*}
Using 
the commutative diagram 
\begin{equation*}
\begin{gathered}[b]
\xymatrix{
\tJ(\bP(V_0), \bP(\barV)) \ar[r]^{p} \ar[d]_{\beta} & \bP(V_0) \times \bP(\barV) \ar[d]^{\pr_2} \\ 
\tC_{V_0}(\bP(\barV)) \ar[r]^{\barp} & \bP(\barV) 
} 
\\[-\dp\strutbox]
\end{gathered}
\end{equation*}
it is easy to see the equivalence $\cC_{V_0}(\cA) \simeq \cJ(\bP(V_0), \cA)$ 
identifies $\cC_i$ with $\cJ_i$; thus the decompositions~\eqref{sod-cC} and~\eqref{sod-cC-left} hold.

By~\cite[Lemma~2.4]{categorical-joins} and Lemma~\ref{lemma-cone-A0}, 
we thus deduce that~$\cC_0 \subset \cC_{V_0}(\cA)$
is a Lefschetz center with $\cC_i$, $i \in \bZ$, the corresponding Lefschetz components. 
The strongness claims follow from the definitions and Lemma~\ref{lemma-cone-A0}, 
and the claims about the length and moderateness of $\cC_{V_0}(\cA)$ follow 
from the definitions. 
\end{proof}


\section{HPD for categorical cones}
\label{section-cones-HPD} 

In this section we show that (under suitable hypotheses) 
the formation of categorical cones commutes with HPD. 
We formulate the theorem in a way that allows for extensions of 
the base projective bundle (in the sense of Definition~\ref{definition-base-extension}), 
because this extra generality is useful in applications. 

\begin{theorem}
\label{theorem-cones-HPD} 
Let $V$ be a vector bundle on $S$, let 
\begin{equation*}
V_0 \subset V 
\quad \text{and} \quad
V_\infty \subset \vV
\end{equation*} 
be subbundles such that the natural pairing $V \otimes \vV \to \cO_S$ is zero on $V_0 \otimes V_{\infty}$, 
so that we have a pair of filtrations 
\begin{equation}
\label{eq:filtrations}
0 \subset V_0 \subset V_\infty^\perp \subset V
\quad \text{and} \quad 
0 \subset V_\infty \subset V_0^\perp \subset \vV. 
\end{equation}
Set 
\begin{equation}
\label{eq:factors}
\barV = V_\infty^\perp /V_0, 
\qquad\text{so that}\quad
V_0^\perp/V_\infty \cong \barV^{\svee}.
\end{equation} 
Let $\cA$ be a right strong, moderate Lefschetz category over $\bP(\barV)$. 
Then there is an equivalence
\begin{equation*}
(\cC_{V_0}(\cA)/\bP(V))^{\hpd} \simeq \cC_{V_\infty}(\cA^{\hpd})/\bP(\vV)
\end{equation*}
of Lefschetz categories over~$\bP(\vV)$. 
\end{theorem} 

\begin{remark}
\label{remark-cones-HPD} 
Let us explain the structure of the categories appearing in Theorem~\ref{theorem-cones-HPD}. 
By Theorem~\ref{theorem-cone-lef-cat} the categorical join $\cC_{V_0}(\cA)$ is a right strong, moderate 
Lefschetz category over~$\bP(V_\infty^{\perp})$. 
By extending the base along the inclusion $\bP(V_\infty^\perp) \to \bP(V)$, we obtain by 
Remark~\ref{remark-base-extension} a right strong, moderate Lefschetz category $\cC_{V_0}(\cA)/\bP(V)$ over~$\bP(V)$. 
Hence by \cite[Theorem~8.7(1)]{NCHPD}, 
the HPD category $(\cC_{V_0}(\cA)/\bP(V))^{\hpd}$ has the 
structure of a Lefschetz category over $\bP(\vV)$. 
The structure of $\cC_{V_\infty}(\cA^{\hpd})/\bP(\vV)$ as a Lefschetz category over $\bP(\vV)$ is similarly 
obtained by a combination of  \cite[Theorem 8.7(1)]{NCHPD}, Theorem~\ref{theorem-cone-lef-cat}, 
and base extension. 
\end{remark}

The case $V_0 = 0$ gives the following (we take $V_\infty = W^\perp \subset \vV$). 

\begin{corollary}
\label{corollary:cones-HPD-special}
Let $W \subset V$ be an inclusion of vector bundles on $S$. 
Let $\cA$ be a right strong, moderate Lefschetz category over $\bP(W)$. 
Then there is an equivalence 
\begin{equation*}
(\cA/\bP(V))^\natural \simeq \cC_{W^\perp}(\cA^\natural)
\end{equation*} 
of Lefschetz categories over $\bP(\vV)$. 
\end{corollary}

Our strategy for proving Theorem~\ref{theorem-cones-HPD} is the following.
First, we use the relation between categorical cones and categorical joins 
described in Proposition~\ref{proposition-cJ-cC} and linear HPD of Example~\ref{ex:categorical-linear-hpd}
to deduce the theorem in case when both $V_0$ and $V_\infty$ are nonzero and the filtrations~\eqref{eq:filtrations} are split
(which always holds locally over the base scheme~$S$).
Then we use a local-to-global argument analogous to the one used in the proof of Theorem~\ref{theorem-cone-lef-cat}
to deduce the theorem without the splitting assumption, 
and finally a relation between HPD and hyperplane sections (Proposition~\ref{proposition-HPD-projection}) 
and duality to deduce the theorem in full generality. 

For the local-to-global argument it is important to define a functor 
between the categories~$\cC_{V_\infty}(\cA^{\hpd})/\bP(\vV)$ and~$(\cC_{V_0}(\cA)/\bP(V))^{\hpd}$ in general.
This is what we start with in~\S\ref{subsection-double-resolved-cones}, where 
we define a functor 
\begin{equation*}
\gamma_{\tC} \colon \tC_{V_\infty}(\cA^\hpd)/\bP(\vV) \to \bH(\tC_{V_0}(\cA)/\bP(V)) 
\end{equation*} 
via a double cone construction, an analogue of the double join construction from \cite[\S4.1]{categorical-joins}. 
Next, in~\S\ref{subsection-cone-HPD-split-case} we check its compatibility with the analogous functor
between resolved joins, deduce the theorem in the split nonzero case, 
and then by the local-to-global argument remove the splitting assumption.
Finally, in~\S\ref{subsection-cone-HPD-general-case} we prove the general case. 

\subsection{Double resolved cones and the HPD functor for categorical cones} 
\label{subsection-double-resolved-cones}

Throughout this section we fix filtrations~\eqref{eq:filtrations}, 
and use~\eqref{eq:factors} to identify their quotients with $(V_0, \barV, \vV_{\infty})$ and~$(V_\infty, \barV^{\svee}, \vV_0)$ respectively.

Let $Y$ be a scheme equipped with a morphism $Y \to \bP(\barV) \times \bP(\barV^{\svee})$.
In this situation, we can form two resolved cones, $\tC_{V_0}(Y)$ and $\tC_{V_\infty}(Y)$, 
using the projection to $\bP(\barV)$ for the first and the projection to $\bP(\barV^{\svee})$ for the second. 
We define the 
\emph{double resolved cone} over~$Y$ as the fiber product 
\begin{equation}
\label{eq:double-cone}
\tCC_{V_0, V_{\infty}}(Y) = \tC_{V_0}(Y) \times_Y \tC_{V_\infty}(Y), 
\end{equation}
which is a $\bP^{N_0} \times \bP^{N_\infty}$-bundle over $Y$, where $N_0 = \dim V_0$ and $N_\infty = \dim V_\infty$.
In particular, we can consider the universal double resolved cone with its natural projection
\begin{equation}
\label{tCC-bc}
\tCC_{V_0,V_\infty}(\bP(\barV) \times \bP(\barV^{\svee})) \to \bP(\barV) \times \bP(\barV^{\svee}).
\end{equation}
Now, given a category $\cB$ which has a $\bP(\barV) \times \bP(\barV^{\svee})$-linear structure, 
we define the \emph{double resolved cone} $\tCC_{V_0, V_{\infty}}(\cB)$ over $\cB$ as 
\begin{equation*}
\tCC_{V_0, V_{\infty}}(\cB) = 
\cB 
\otimes_{\Perf(\bP(\barV) \times \bP(\barV^{\svee}))}
\Perf(\tCC_{V_0,V_\infty}(\bP(\barV) \times \bP(\barV^{\svee}))),
\end{equation*}
that is the base change of $\cB$ along~\eqref{tCC-bc}.

The key case for us is when 
$Y$ is the universal space of hyperplanes in $\bP(\barV)$, which we denote by 
\begin{equation}
\label{eq:barbh}
\barbH = \bH(\bP(\barV)) . 
\end{equation}
Note that $\barbH$ indeed naturally maps to $\bP(\barV) \stimes \bP(\barV^{\svee})$, 
hence we can form the double resolved cone over $\barbH$. 
We write $\bH(\tC_{V_0}(\bP(\barV))/\bP(V))$ for the universal hyperplane section 
of~$\tC_{V_0}(\bP(\barV))$ with respect to the morphism 
$\tC_{V_0}(\bP(\barV)) \to \bP(V_{\infty}^{\perp}) \to \bP(V)$. 
The second projection in~\eqref{eq:double-cone} defines a map~$\tilde{p}$ in~\eqref{tCC-functor-spaces},
and the first projection together with the map~$\tCC_{V_0, V_{\infty}}(\barbH) \to \barbH \to \bP(\vV)$ define the map~$\alpha$.
It is easy to show the following. 

\begin{lemma}
We have a diagram 
\begin{equation}
\label{tCC-functor-spaces}
\vcenter{
\xymatrix{
& \tCC_{V_0, V_{\infty}}(\barbH) \ar[dl]_{\tilde{p}} \ar[dr]^{\alpha}
\\
\tC_{V_\infty}(\barbH) && 
\bH(\tC_{V_0}(\bP(\barV))/\bP(V)) 
}
}
\end{equation}
of schemes over $\bP(\vV)$, where all schemes appearing are smooth and projective over $S$.
\end{lemma}

Let $\cA$ be a $\bP(\barV)$-linear category. 
Then $\bH(\cA) = \bH(\cA/\bP(\barV)) = \Perf(\barbH) \otimes_{\Perf(\bP(\barV))} \cA$ according to the notation~\eqref{eq:barbh}.
Recall the canonical $\bP(\barV^{\svee})$-linear inclusion functor~\eqref{def:gamma}
\begin{equation*}
\gamma \colon \cAd = (\cA/\bP(\barV))^\hpd \to \bH(\cA/\bP(\barV)) = \Perf(\barbH) \otimes_{\Perf(\bP(\barV))} \cA.
\end{equation*}
This induces a $\bP(V_0^{\perp})$-linear functor (see Remark~\ref{remark-tC-functorial}) 
\begin{equation*}
\tC_{V_{\infty}}(\gamma) \colon \tC_{V_\infty}(\cA^\hpd) \to 
\tC_{V_\infty}(\bH(\cA)), 
\end{equation*}
which can be regarded as a $\bP(\vV)$-linear functor 
\begin{equation*}
\tC_{V_{\infty}}(\gamma) \colon \tC_{V_\infty}(\cA^\hpd)/\bP(\vV) \to 
\tC_{V_\infty}(\bH(\cA))/\bP(\vV) .  
\end{equation*}
Here, we have written $\tC_{V_\infty}(\cA^\hpd)/\bP(\vV)$ and $\tC_{V_\infty}(\bH(\cA))/\bP(\vV)$ to emphasize that 
we regard the resolved cones $\tC_{V_\infty}(\cA^\hpd)$ and~$\tC_{V_\infty}(\bH(\cA))$ as~$\bP(\vV)$-linear categories, 
via the inclusion~$\bP(V_0^{\perp}) \subset \bP(\vV)$. 

By base change from diagram~\eqref{tCC-functor-spaces}  
we obtain a diagram of~$\bP(\vV)$-linear functors 
\begin{equation*}
\xymatrix@C=1.6em{
&&& 
\tCC_{V_0,V_{\infty}}(\bH(\cA)) \ar[dr]^{\alpha_*}
\\
\tC_{V_\infty}(\cA^\hpd)/\bP(\vV) \ar[rr]^-{\tC_{V_{\infty}}(\gamma)} &&
\tC_{V_\infty}(\bH(\cA))/\bP(\vV) \ar[ur]^{\tilde{p}^*} && 
\bH(\tC_{V_0}(\cA)/\bP(V)) 
}
\end{equation*}
We define a $\bP(\vV)$-linear functor as the composition 
\begin{equation}
\label{gammatC}
\gamma_{\tC} = \alpha_* \circ \tp^* \circ \tC_{V_{\infty}}(\gamma) 
\colon \tC_{V_\infty}(\cA^\hpd)/\bP(\vV) \to \bH(\tC_{V_0}(\cA)/\bP(V)) . 
\end{equation}

The following fact will be needed later. 
\begin{lemma}
\label{lemma:gtc-adjoints}
The functor~$\gamma_{\tC}$ has both left and right adjoints.
\end{lemma}

\begin{proof}
The functor $\gamma$ has both left and right adjoints by~\cite[Lemma~7.2]{NCHPD}, hence 
so does~$\tC_{V_{\infty}}(\gamma)$ (see Remark~\ref{remark-tC-functorial}). 
Further, $\alpha_*$ and $\tp^*$ have both left and right adjoints, 
by Lemma~\ref{tCC-functor-spaces} and Remark~\ref{remark-adjoints-exist}.  
\end{proof}

\begin{remark}
The functor $\gamma_\tC$ can also be described in terms of Fourier--Mukai kernels, similarly to 
\cite[Remark 4.8]{categorical-joins}. 
We leave this as an exercise.
\end{remark}

Note that by definition the HPD category $(\cC_{V_0}(\cA)/\bP(V))^{\hpd}$  
is a~$\bP(\vV)$-linear subcategory of $\bH(\cC_{V_0}(\cA)/\bP(V))$, and 
hence also of $\bH(\tC_{V_0}(\cA)/\bP(V))$ by \cite[Lemma 2.12]{NCHPD}. 

\subsection{The nonzero case} 
\label{subsection-cone-HPD-split-case} 

The goal of this subsection is to prove the following more precise version of 
Theorem~\ref{theorem-cones-HPD} when $V_0$ and $V_\infty$ are nonzero. 

\begin{proposition}
\label{proposition-cones-HPD-precise}
Let $\cA$ be a right strong, moderate Lefschetz category over $\bP(\barV)$. 
Assume~$V_0$ and $V_{\infty}$ are nonzero.  
Then the functor 
\begin{equation*}
\gamma_{\tC} \colon \tC_{V_\infty}(\cA^\hpd)/\bP(\vV) \to \bH(\tC_{V_0}(\cA)/\bP(V)) 
\end{equation*} 
defined in~\eqref{gammatC}
induces a Lefschetz equivalence between the subcategories 
\begin{equation*}
\cC_{V_\infty}(\cA^{\hpd})/\bP(\vV) \subset  \tC_{V_\infty}(\cA^\hpd)/\bP(\vV)  
 \quad \text{and} \quad 
(\cC_{V_0}(\cA)/\bP(V))^{\hpd} \subset \bH(\tC_{V_0}(\cA)/\bP(V)) . 
\end{equation*}
\end{proposition} 

The proof takes the rest of the subsection.
Let us outline the strategy.

By Lemma~\ref{lemma:gtc-adjoints} the functor $\gamma_\tC$ has adjoints. 
Therefore, by Corollary~\ref{corollary-equivalence-lef-cat-local} 
the claim of Proposition~\ref{proposition-cones-HPD-precise} is fpqc-local,
so it is enough to prove it over a fpqc cover of the base scheme $S$.
Passing to such a cover we may assume that the filtrations~\eqref{eq:filtrations} split, 
so that 
\begin{equation*}
V = V_0 \oplus \barV \oplus \vV_{\infty}. 
\end{equation*}
For the rest of this subsection, we fix such a splitting. 
Using this, we will reduce Proposition~\ref{proposition-cones-HPD-precise} to 
\cite[Theorem 4.9]{categorical-joins}. 
We set 
\begin{equation*}
V_1 = V_0 \oplus \vV_\infty. 
\end{equation*} 
Then the orthogonal to $V_0 \subset V_1$ 
is $V_{\infty} \subset \vV_1$, 
so by Example~\ref{ex:categorical-linear-hpd}
there is an equivalence 
\begin{equation}
\label{PV0-PVinfty-hpd}
\Perf(\bP(V_0))^{\hpd} \simeq \Perf(\bP(V_\infty))
\end{equation} 
of Lefschetz categories over $\bP(\vV_1)$. 
Hence we have a commutative diagram of equivalences 
of Lefschetz categories over $\bP(\vV)$: 
\begin{equation}
\label{eq:cc-cj-diagram}
\vcenter{\xymatrix{
\cJ(\bP(V_\infty), \cA^{\hpd})/\bP(\vV) \bijartop[r] & 
(\cJ(\bP(V_0), \cA)/\bP(V))^{\hpd} \bijartop[d] 
\\ 
\cC_{V_\infty}(\cA^{\hpd})/\bP(\vV) \bijartop[u] \bijartop[r] & 
( \cC_{V_0}(\cA)/\bP(V) )^{\hpd}
}}
\end{equation}
where the vertical equivalences are consequences of Proposition~\ref{proposition-cJ-cC}, 
the top equivalence is given by~Theorem~\ref{theorem-joins-HPD} 
(note that $V = V_1 \oplus \barV$) combined with~\eqref{PV0-PVinfty-hpd}, and 
the bottom equivalence is the composition of the other three. 
To prove Proposition~\ref{proposition-cones-HPD-precise}, 
we check that the bottom equivalence is in fact induced by the functor $\gamma_{\tC}$. 

\begin{remark}
The above argument 
(even without checking the bottom arrow is induced by~$\gamma_{\tC}$)
already proves Theorem~\ref{theorem-cones-HPD} 
under the assumptions that $V_0$ and $V_\infty$ are nonzero and 
the filtrations \eqref{eq:filtrations} are split. 
However, for the local-to-global argument above by which we reduced 
to the split case, it is essential that we verify 
the equivalence is given by a globally defined functor. 
\end{remark}

To check that the bottom equivalence in~\eqref{eq:cc-cj-diagram} is induced by $\gamma_\tC$, 
we prove commutativity of the analogous diagram of resolved joins and cones. 
Below we use the notations $\tJ$ and $\tJv$ for resolved joins introduced in \cite[\S4.1]{categorical-joins}. 

\begin{proposition}
There is a commutative diagram 
\begin{equation}
\label{eq:tc-tj-diagram}
\vcenter{\xymatrix@C=6em{
{\tJv}(\bP(V_\infty), \cA^{\hpd})/\bP(\vV) \ar[r]^-{\gamma_{\tJ_{\infty}}} & 
\bH(\tJ(\bP(V_0), \cA)/\bP(V))
\ar[d]^{\beta_{0*}} 
\\ 
\tC_{V_\infty}(\cA^{\hpd})/\bP(\vV) \ar[u]^{\beta_\infty^*} \ar[r]^-{\gamma_\tC} & 
\bH(\tC_{V_0}(\cA)/\bP(V) ),
}}
\end{equation}
where $\beta_0$ and $\beta_\infty$ are the blowup morphisms from Lemma~\textup{\ref{lemma-tJ-tC}} 
of the cones with vertices~$\bP(V_0)$ and~$\bP(V_\infty)$, respectively, and~$\gamma_{\tJ_{\infty}}$ is the composition 
\begin{equation*}
\gamma_{\tJ_{\infty}} \colon {\tJv}(\bP(V_\infty), \cA^{\hpd})/\bP(\vV) \xrightarrow{\ \sim \ } 
{\tJv}(\Perf(\bP(V_0))^{\hpd}, \cA^{\hpd})/\bP(\vV)
\xrightarrow{ \ \gamma_{\tJ} \ } \bH(\tJ(\bP(V_0), \cA)/\bP(V)) 
\end{equation*}
where the equivalence is induced by~\eqref{PV0-PVinfty-hpd} and $\gamma_{\tJ}$ is 
the functor from \cite[Theorem 4.9]{categorical-joins} 
\textup(with~$V_2 = \barV$\textup). 
\end{proposition}

\begin{proof}
First, we unwind the definition of the functor $\gamma_{\tJ_{\infty}}$. 
By \cite[Corollary~8.3]{kuznetsov-hpd} the functor 
\begin{equation*}
\iota_* \circ \pr_2^* \colon \Perf(\bP(V_\infty)) \to \Perf(\bH(\bP(V_0))) 
\end{equation*} 
induces the HPD between $\bP(V_0)$ and $\bP(V_\infty)$, where 
$\bH(\bP(V_0)) = \bH(\bP(V_0)/\bP(V_1))$ is the universal hyperplane section of the morphism $\bP(V_0) \to \bP(V_1)$, 
$\iota \colon \bP(V_0) \times \bP(V_\infty) \to \bH(\bP(V_0))$ is the 
embedding, and $\pr_2 \colon \bP(V_0) \times \bP(V_\infty) \to \bP(V_\infty)$ is 
the projection. 
It follows from the definitions that we have 
\begin{equation*}
\gamma_{\tJ_{\infty}}  \simeq \alpha_{0*} \circ \tp_0^* \circ \tJ(\iota_* \circ \pr_2^*, \gamma), 
\end{equation*} 
where the morphisms 
\begin{equation*}
\tJv(\bH(\bP(V_0)), \barbH)  \xleftarrow{\ \tp_0 \ } 
\tJJ(\bH(\bP(V_0)), \barbH) \xrightarrow{ \ \alpha_0 \ } \bH(\tJ(\bP(V_0), \bP(\barV))/\bP(V)) 
\end{equation*} 
are the base change along $\bP(V_0) \to \bP(V_1)$ 
of~\cite[diagram (4.4)]{categorical-joins} (with $V_2 = \barV$), and 
$\tJ(\iota_* \circ \pr_2^*, \gamma)$ is the join of the functors $\iota_* \circ \pr_2^*$ 
and $\gamma$, where $\gamma$ is the inclusion~\eqref{def:gamma}.

Further, note that we can write $\tJ(\iota_* \circ \pr_2^*, \gamma)$ as a composition 
\begin{align*}
{\tJv}(\bP(V_\infty), \cA^\hpd) \xrightarrow{\tJ(\id, \gamma)}   
{\tJv}(\bP(V_\infty), \bH(\cA)) 
& \xrightarrow{\tJ( \pr_2^*, \id)}  \tJv(\bP(V_0) \times \bP(V_\infty), \bH(\cA))  \\ 
& \xrightarrow{{\tJ}( \iota_*, \id)} 
\tJv(\bH(\bP(V_0)), \bH(\cA)), 
\end{align*} 
and hence 
\begin{equation}
\label{tJinfty}
\gamma_{\tJ_{\infty}} \simeq \alpha_{0*} \circ \tp_0^* \circ {\tJ}(\iota_*, \id) \circ \tJ(\pr_2^*, \id) \circ \tJ(\id, \gamma) . 
\end{equation}
By definition $\gamma_\tC$ is a composition~\eqref{gammatC} of three functors 
analogous to $\tJ(\id, \gamma)$, $\tp_0^*$, and $\alpha_{0*}$ in~\eqref{tJinfty}.  
To prove the proposition, we will relate the analogous functors appearing in these compositions, 
using the blowup morphisms $\beta_\infty$ and $\beta_0$ and the 
morphisms $\iota$ and $\pr_2$. 

The relation between $\tJ(\id, \gamma)$ and $\tC_{V_\infty}({\gamma})$ is 
provided by the commutative diagram~\eqref{eq:tj-tc-diagram}, that in our case takes the form
\begin{equation}
\label{eq:tc-tj-infinity}
\vcenter{\xymatrix@C=6em{
{\tJv}(\bP(V_\infty), \cA^{\hpd}) \ar[r]^-{\tJ(\id,{\gamma})} & 
{\tJv}(\bP(V_\infty), \bH(\cA))
\\ 
\tC_{V_\infty}(\cA^{\hpd}) \ar[u]^{\beta_\infty^*} \ar[r]^-{\tC_{V_\infty}({\gamma})} & 
\tC_{V_\infty}(\bH(\cA)). \ar[u]_{\beta_\infty^*} 
}}
\end{equation}

To relate the other functors, 
we write down diagrams of schemes that induce diagrams of functors by base change. 
First note that we have a fiber square  
\begin{equation}
\label{eq:iota-tp}
\vcenter{\xymatrix@C=6em{
\tJv(\bP(V_0) \times \bP(V_\infty), \barbH) \ar[d]^{{\tJ}(\iota,\id)}  & \ar[l]_{\tp_{0\infty}}
\tJJ(\bP(V_0) \times \bP(V_\infty), \barbH) \ar[d]^{\tJJ(\iota,\id)} \\ 
\tJv(\bH(\bP(V_0)), \barbH)  & \ar[l]_{\tp_0}
\tJJ(\bH(\bP(V_0)), \barbH) 
}}
\end{equation} 
where $\tJJ(\iota,\id)$ denotes the morphism between the double resolved joins 
induced by the morphisms $\iota \colon \bP(V_0) \times \bP(V_\infty) \to \bH(\bP(V_0))$ 
and $\id \colon \barbH \to \barbH$.
Next observe that  
\begin{align*}
\tJJ(\bP(V_0) \times \bP(V_\infty), \barbH) & = 
\tJ(\bP(V_0) \times \bP(V_\infty), \barbH) \times_{(\bP(V_0) \times \bP(V_\infty)) \times \barbH} 
\tJv(\bP(V_0) \times \bP(V_\infty), \barbH) \\ 
& \cong \tJ(\bP(V_0), \barbH) \times_{\barbH} {\tJv}(\bP(V_{\infty}), \barbH), 
\end{align*}
where the first equality holds by definition. 
We also have by definition 
\begin{equation*}
\tCC_{V_0, V_{\infty}}(\barbH) = \tC_{V_0}(\barbH) \times_{\barbH} \tC_{V_\infty}(\barbH) . 
\end{equation*}
The blowup morphisms $\beta_{0} \colon \tJ(\bP(V_0), \barbH) \to \tC_{V_0}(\barbH)$ and 
$\beta_{\infty} \colon {\tJv}(\bP(V_\infty), \barbH) \to \tC_{V_\infty}(\barbH)$ from 
Lemma~\ref{lemma-tJ-tC} 
thus combine to give a morphism 
\begin{equation*}
\beta_{0\infty} \colon \tJJ(\bP(V_0) \times \bP(V_\infty), \barbH) \to \tCC_{V_0, V_{\infty}}(\barbH). 
\end{equation*}
It is easy see that the morphism $\beta_{0\infty}$ makes the diagrams 
\begin{equation}
\label{eq:8-08}
\vcenter{\xymatrix@C=4.25em{
{\tJv}(\bP(V_\infty), \barbH)  & \ar[l]_-{\tJ(\pr_2,\id)}
\tJv(\bP(V_0) \times \bP(V_\infty), \barbH)  & \ar[l]_-{\tp_{0\infty}}
\tJJ(\bP(V_0) \times \bP(V_\infty), \barbH)
\\
\tC_{V_\infty}(\barbH) \ar[u];[]_{\beta_\infty}  && \ar[ll]_{\tp}
\tCC_{V_0,V_\infty}(\barbH) \ar[u];[]^{\beta_{0\infty}} 
}}
\end{equation} 
and 
\begin{equation}
\label{eq:abi}
\vcenter{\xymatrix@C=3em{
\tJJ(\bP(V_0) \times \bP(V_\infty), \barbH) \ar[r]^-{\tJJ(\iota,\id)} \ar[d]_{\beta_{0\infty}} &
\tJJ(\bH(\bP(V_0)), \barbH) \ar[r]^-{\alpha_0} &
\bH(\tJ(\bP(V_0), \bP(\barV))/\bP(V)) \ar[d]^{\beta_{0}}
\\
\tCC_{V_0,V_\infty}(\barbH)  \ar[rr]^-{\alpha} &&
\bH(\tC_{V_0}(\bP(\barV))/\bP(V)) 
}}
\end{equation}
commutative, where in~\eqref{eq:abi} we abusively write $\beta_0$ for the 
morphism induced by the blowup $\beta_0 \colon \tJ(\bP(V_0), \bP(\barV)) \to \tC_{V_0}(\bP(\barV))$.  
Note also that since $\beta_{0\infty}$ is a product of two blowup morphisms, 
the functor $\beta_{0\infty}^*$ is fully faithful, so
we have an isomorphism of functors 
\begin{equation}
\label{beta0infty}
\beta_{0\infty*} \circ \beta_{0\infty}^* \simeq \id. 
\end{equation}

Finally, combining the above ingredients 
and taking into account that $\tJ(\iota_*, \id) \cong \tJ(\iota, \id)_*$ 
and $\tJ(\pr_2^*, \id) \cong \tJ(\pr_2, \id)^*$,
we can rewrite the composition of the three upper arrows in~\eqref{eq:tc-tj-diagram} as 
\begin{align*}
\beta_{0*} \circ \gamma_{\tJ_\infty} \circ \beta_\infty^* 
& \simeq \beta_{0*}  \circ \alpha_{0*} \circ \tp_0^* \circ {\tJ}(\iota_*, \id) \circ \tJ(\pr_2^*, \id) \circ \tJ(\id, \gamma) \circ \beta_\infty^*  
&& \eqref{tJinfty}\\
& \simeq  \beta_{0*}  \circ \alpha_{0*} \circ \tp_0^* \circ {\tJ}(\iota_*, \id) \circ \tJ(\pr_2^*, \id) \circ \beta_\infty^* \circ \tC_{V_\infty}(\gamma)  && \eqref{eq:tc-tj-infinity} \\ 
& \simeq \beta_{0*} \circ \alpha_{0*} \circ \tJJ(\iota_*,\id) \circ \tp_{0\infty}^* \circ \tJ(\pr_2^*,\id) \circ \beta_\infty^* \circ \tC_{V_\infty}(\gamma) 
&& \eqref{eq:iota-tp} \\
& \simeq \beta_{0*} \circ \alpha_{0*} \circ \tJJ(\iota_*,\id) \circ \beta_{0\infty}^* \circ \tp^* \circ \tC_{V_\infty}(\gamma) 
&& \eqref{eq:8-08} \\
& \simeq \alpha_* \circ \beta_{0\infty*} \circ \beta_{0\infty}^* \circ \tp^* \circ \tC_{V_\infty}(\gamma) 
&& \eqref{eq:abi} \\
& \simeq \alpha_* \circ \tp^* \circ \tC_{V_\infty}(\gamma) && \eqref{beta0infty} \\ 
& = \gamma_{\tC}
&& \eqref{gammatC}  
\end{align*}
which completes the proof. 
\end{proof}

\begin{proof}[Proof of Proposition~\textup{\ref{proposition-cones-HPD-precise}}]
As explained above, we first consider the case where the filtrations~\eqref{eq:filtrations} split. 
Then we have a commutative diagram~\eqref{eq:tc-tj-diagram}, whose vertical arrows and 
top horizontal arrow induce the corresponding arrows of~\eqref{eq:cc-cj-diagram}. 
Hence by commutativity of these diagrams, the functor~$\gamma_\tC$ 
induces the Lefschetz equivalence given by the bottom horizontal arrow of~\eqref{eq:cc-cj-diagram}. 

In the nonsplit case we use Proposition~\ref{proposition-equivalence-local} 
with $\cC = \tC_{V_\infty}(\cA^\hpd)$, $\cD = \bH(\tC_{V_0}(\cA)/\bP(V))$, \mbox{$\phi = \gamma_\tC$},
$\cA = \cC_{V_\infty}(\cA^\hpd)$, and $\cB = (\cC_{V_0}(\cA)/\bP(V))^{\hpd}$.
We note that the assumptions of the proposition are satisfied by Lemmas~\ref{lemma-tC} and~\ref{lemma:gtc-adjoints}.
We take an fpqc cover of our base scheme~$S$ over which the filtrations~\eqref{eq:filtrations} split.
By the argument above the functor $\gamma_\tC$ induces the desired Lefschetz equivalence after base change to this cover.
Hence by Proposition~\ref{proposition-equivalence-local} the functor~$\gamma_{\tC}$ induces an equivalence between 
$\cC_{V_\infty}(\cA^\hpd)/\bP(\vV)$ and $(\cC_{V_0}(\cA)/\bP(V))^{\hpd}$, which is in fact a Lefschetz
equivalence by Corollary~\ref{corollary-equivalence-lef-cat-local}. 
\end{proof}

\subsection{The general case} 
\label{subsection-cone-HPD-general-case}

In this subsection we bootstrap from Proposition~\ref{proposition-cones-HPD-precise} 
to the general case of Theorem~\ref{theorem-cones-HPD}. 

\begin{lemma}
\label{lemma-cones-HPD-V0}
The claim of Theorem~\textup{\ref{theorem-cones-HPD}} holds if $V_0$ is nonzero. 
\end{lemma}

\begin{proof}
By Proposition~\ref{proposition-cones-HPD-precise} we only need to consider the 
case where $V_\infty = 0$. 
Take an auxiliary nonzero vector bundle $\tV_{\infty}$ on $S$, and set $\tV = V \oplus \tV_{\infty}^{\svee}$. 
Then $V_0 \subset \tV$ and~$\tV_{\infty} \subset \tV^{\svee}$ are such 
that the pairing $\tV \otimes \tV^{\svee} \to \cO_S$ is zero on $V_0 \otimes \tV_{\infty}$. 
Hence Proposition~\ref{proposition-cones-HPD-precise}  
gives an equivalence 
\begin{equation}
\label{Vinfty01}
\cC_{\tV_\infty}(\cAd)/\bP(\tV^{\svee}) \simeq  
(\cC_{V_0}(\cA)/\bP(\tV))^{\hpd} 
\end{equation}
of Lefschetz categories over $\bP(\tV^{\svee})$. 
By base change along the embedding $\bP(\vV) \to \bP(\tV^{\svee})$ 
we obtain a $\bP(\vV)$-linear equivalence
\begin{equation*}
\Big( \cC_{\tV_\infty}(\cAd)/\bP(\tV^{\svee}) \Big) \otimes_{\Perf(\bP(\tV^{\svee}))} \Perf(\bP(\vV)) \simeq  
(\cC_{V_0}(\cA)/\bP(\tV))^{\hpd} \otimes_{\Perf(\bP(\tV^{\svee}))} \Perf(\bP(\vV)) ,
\end{equation*}
On the one hand, we have $\bP(\vV)$-linear equivalences  
\begin{equation}
\label{Vinfty02}
\left( \cC_{\tV_\infty}(\cAd)/\bP(\tV^{\svee}) \right) \otimes_{\Perf(\bP(\tV^{\svee}))} \Perf(\bP(\vV)) 
\simeq 
\cAd/\bP(\vV) \simeq 
\cC_{V_\infty}(\cAd) / \bP(\vV),  
\end{equation}
where the first holds by Corollary~\ref{corollary:cone-trivial} and the second by Lemma~\ref{lemma:cone-0} since $V_{\infty} = 0$.
On the other hand, note that $\cC_{V_0}(\cA)/\bP(\tV)$ is supported over the open $\bP(\tV) \setminus \bP(\tV_{\infty}^{\svee})$ 
since this category's $\bP(\tV)$-linear structure is induced from a $\bP(V)$-linear structure via the morphism~$\bP(V) \to \bP(\tV)$
and~$\bP(V) \cap \bP(V_\infty^\vee) = \varnothing$. 
Hence by Proposition~\ref{proposition-HPD-projection} we have a $\bP(\vV)$-linear equivalence 
\begin{equation}
\label{Vinfty03}
(\cC_{V_0}(\cA)/\bP(\tV))^{\hpd} \otimes_{\Perf(\bP(\tV^{\svee}))} \Perf(\bP(\vV)) 
\simeq 
(\cC_{V_0}(\cA)/\bP(V))^{\hpd}.
\end{equation}  
Combining the above equivalences,  
we thus obtain a $\bP(\vV)$-linear equivalence 
\begin{equation}
\label{Vinfty04} 
\cC_{V_\infty}(\cAd) / \bP(\vV) \simeq (\cC_{V_0}(\cA)/\bP(V))^{\hpd}. 
\end{equation}
Finally, tracing through the equivalences~\eqref{Vinfty02} and~\eqref{Vinfty03} 
and using Remark~\ref{remark-HPD-projection} and the fact that~\eqref{Vinfty01} is a Lefschetz equivalence, 
one verifies that \eqref{Vinfty04} identifies the Lefschetz centers on each side. 
\end{proof} 

Now we can handle the general case. 

\begin{proof}[Proof of Theorem~\textup{\ref{theorem-cones-HPD}}] 
By Lemma~\ref{lemma-cones-HPD-V0} it remains to consider the case 
where $V_0 = 0$. We may assume $V_\infty \neq 0$, otherwise there is nothing to prove. 
Using Remark~\ref{remark:left-hpd} we reduce 
the claim of Theorem~\ref{theorem-cones-HPD} 
to the existence of a Lefschetz equivalence 
\begin{equation*}
\cC_{V_0}(\cA)/\bP(V) \simeq {^\hpd} \left( \cC_{V_\infty}(\cA^{\hpd})/\bP(\vV) \right). 
\end{equation*} 
Since $V_\infty \neq 0$ we can apply (the left version of) Lemma~\ref{lemma-cones-HPD-V0} to obtain a 
Lefschetz equivalence   
\begin{equation*}
{^\hpd} \left( \cC_{V_\infty}(\cA^{\hpd})/\bP(\vV) \right) \simeq 
\cC_{V_0}({^\hpd}(\cA^{\hpd}))/\bP(V). 
\end{equation*} 
We conclude by noting that ${^\hpd}(\cA^{\hpd}) \simeq \cA$, again by Remark~\ref{remark:left-hpd}.
\end{proof} 


\section{HPD for quadrics} 
\label{section-HPD-for-quadrics}

In this section, we use categorical cones to describe HPD for quadrics. 
We assume the base scheme $S$ is the spectrum of an algebraically closed field $\bk$ of characteristic not equal to $2$. 
The main reason for this assumption is that our results depend on our work~\cite{kuznetsov-perry-HPD-quadrics}, 
reviewed in~\S\ref{subsection-HPD-smooth-quadrics} below, 
where we described HPD for certain \emph{smooth} quadrics over such a field $\bk$. 
In fact, it is possible to generalize the results of \cite{kuznetsov-perry-HPD-quadrics}  
to the case of families of quadrics with simple degenerations, and then the arguments of this 
section go through in a suitable relative setting, where $S$ is not assumed to be a point; 
however, we will not address this here.

We study the following class of morphisms from a quadric to a projective space.   

\begin{definition}
\label{definition-standard-morphism} 
Let $Q$ be a quadric, i.e. an integral 
scheme over $\bk$ which admits a closed immersion into a projective space as a quadric hypersurface. 
We denote by $\cO_Q(1)$ the restriction of the line bundle $\cO(1)$ from this ambient space.
A morphism $f \colon Q \to \bP(V)$ is \emph{standard} if there is an isomorphism
\begin{equation*}
f^*\cO_{\bP(V)}(1) \cong \cO_Q(1).
\end{equation*}
In other words, $f$ is either an embedding as a quadric hypersurface into a linear subspace of~$\bP(V)$,
or a double covering of a linear subspace of $\bP(V)$ branched along a quadric hypersurface.
{We call $f$ \emph{non-degenerate} if its image is not contained in a hyperplane of~$\bP(V)$.}
\end{definition}

Note that $Q$ is not required to be smooth,
but is required to be integral.
In the preliminary~\S\ref{subsection-HPD-smooth-quadrics}, we recall that if~$Q$ \emph{is} 
smooth then it has a natural Lefschetz structure, and if~\mbox{$f \colon Q \to \bP(V)$} is a non-degenerate standard morphism
the HPD category can be described in terms of classical projective duality. 
In~\S\ref{ssubsection-quadric-resolution} we use categorical cones to construct 
for a general standard morphism \mbox{$f \colon Q \to \bP(V)$} 
a Lefschetz category $\cQ$ over $\bP(V)$ --- called the \emph{standard categorical resolution} of $Q$ --- 
which is smooth and proper over~$\bk$ and agrees with~$\Perf(Q)$ over the complement 
of \mbox{$f(\Sing(Q)) \subset \bP(V)$}. 
In~\S\ref{ssubsection-generalized-quadric-duality} we introduce a ``generalized duality'' operation 
that associates to a standard morphism \mbox{$f \colon Q \to \bP(V)$} of a quadric another such 
morphism~\mbox{$f^{\hpd} \colon Q^\hpd \to \bP(\vV)$}. 
We use HPD for categorical cones to prove that this notation is compatible with the notation for the HPD category, i.e. 
that the HPD of the standard categorical resolution of $Q$ is Lefschetz equivalent 
to the standard categorical resolution of the generalized dual $Q^{\hpd}$ 
(Theorem \ref{theorem:hpd-quadrics-general}). 
By combining this with the nonlinear HPD theorem, we prove in~\S\ref{subsection-quadratic-HPD} a quadratic 
HPD theorem (Theorem~\ref{theorem-quadric-intersection}). 

\subsection{HPD for smooth quadrics} 
\label{subsection-HPD-smooth-quadrics} 

In this subsection, we briefly review HPD for smooth quadrics following \cite{kuznetsov-perry-HPD-quadrics}. 
Given a smooth quadric $Q$, we will denote by $\cS$ a chosen spinor bundle on it. 
Note that there is either one or two choices for $\cS$ depending on whether $\dim(Q)$ is odd or even. 

\begin{lemma}[{\cite[Lemma~1.1]{kuznetsov-perry-HPD-quadrics}}]
\label{lemma-smooth-Q-lc}
Let $f \colon Q \to \bP(V)$ be a standard morphism of a smooth quadric $Q$. 
Let $\cS$ denote a spinor bundle on $Q$. 
Then $\Perf(Q)$ is smooth and proper over $\bk$, and has the structure of a strong,  
moderate Lefschetz category over~$\bP(V)$ with Lefschetz center 
\begin{equation*}
\cQ_0 = \langle \cS , \cO \rangle 
\subset \Perf(Q) 
\end{equation*}
and length $\dim(Q)$. 
Further, if $p \in \set{0,1}$ is the parity of $\dim(Q)$, i.e. $p = \dim(Q) \pmod 2$, 
then the nonzero Lefschetz components of $\Perf(Q)$ are given by 
\begin{equation*}
\cQ_i = \begin{cases}
\langle \cS, \cO \rangle & \text{for $|i| \leq 1-p$} , \\ 
\langle \cO \rangle & \text{for $1-p < |i| \leq \dim(Q)-1$}.
\end{cases}
\end{equation*}
\end{lemma}

\begin{remark}
\label{remark:other-spinor-center}
If $\dim Q$ is even there are two choices of $\cS$, but
up to equivalence, the Lefschetz structure on $\Perf(Q)$ does not depend on this choice, 
see~\cite[Remark 1.2]{kuznetsov-perry-HPD-quadrics}. 
Further, the Lefschetz center $\cQ_0$ of $\Perf(Q)$ can be also written as 
\begin{equation*}
\cQ_0 = \langle \cO, {\cS'}^{\svee}  \rangle
\end{equation*}
where $\cS' = \cS$ if $\dim(Q)$ is not divisible by 4, and the other spinor bundle otherwise. 
The nonzero primitive Lefschetz components (as defined in~\S\ref{subsection-lef-cats}) of $\Perf(Q)$  are given by 
\begin{equation*}
\fq_i = 
\begin{cases}
\langle \cO \rangle & \text{if $i = -(\dim(Q) - 1)$} , \\ 
\langle {\cS'}^{\svee} \rangle & \text{if $i = -(1-p)$} , \\ 
\langle \cS \rangle & \text{if $i = 1-p$} , \\ 
\langle \cO \rangle & \text{if $i = \dim(Q)-1$}. 
\end{cases}
\end{equation*}
\end{remark}

The next result describes HPD for {non-degenerate} standard morphisms of smooth quadrics. 
This will be generalized to arbitrary standard morphisms of quadrics in 
Theorem~\ref{theorem:hpd-quadrics-general}. 
Recall that the classical projective dual of a smooth quadric hypersurface $Q \subset \bP(V)$ is itself 
a smooth quadric hypersurface $Q^{\svee} \subset \bP(\vV)$. 

\begin{theorem}[{\cite[Theorem~1.4]{kuznetsov-perry-HPD-quadrics}}]
\label{theorem-HPD-quadrics} 
Let $f \colon Q \to \bP(V)$ be a {non-degenerate} standard morphism of a smooth quadric $Q$.
Then there is an equivalence 
\begin{equation*}
\Perf(Q)^{\hpd} \simeq \Perf(Q^{\natural}) 
\end{equation*}
of Lefschetz categories over $\bP(V^{\svee})$, where: 
\begin{enumerate}
\item \label{HPD-quadrics-1}
If $f$ is a divisorial embedding and $\dim(Q)$ is even, then $Q^{\hpd} = Q^{\svee}$ is the projective 
dual of~$Q$ and $Q^{\natural} \to \bP(\vV)$ is its natural embedding. 
\item \label{HPD-quadrics-d-odd} 
If $f$ is a divisorial embedding and $\dim(Q)$ is odd, then $Q^{\natural} \to \bP(\vV)$ is the double cover 
branched along the projective dual of~$Q$. 
\item \label{HPD-quadrics-dc-even}
If $f$ is a double covering and~$\dim(Q)$ is even, then $Q^{\natural} \to \bP(\vV)$ is the projective 
dual of the branch locus of~$f$. 
\item \label{HPD-quadrics-dc-odd} 
If $f$ is a double covering and $\dim(Q)$ is odd, then $Q^{\natural} \to \bP(\vV)$ is the double cover 
branched along the projective dual of the branch locus of~$f$. 
\end{enumerate} 
\end{theorem}

\subsection{Standard categorical resolutions of quadrics} 
\label{ssubsection-quadric-resolution}
In this subsection, we will obtain a categorical resolution of a singular quadric 
by expressing it as a cone over a smooth quadric, and then taking a categorical cone. 
To start with, we analyze the general structure of a standard morphism of quadrics.
The following result is clear.

\begin{lemma}
\label{lemma:standard-morphism}
Let $Q \subset \bP(\tW)$ be a quadric hypersurface.
Then there are a unique subspace~\mbox{$K \subset \tW$} and a smooth quadric $\barQ \subset \bP(\tW/K)$ such that
\begin{equation*}
Q \cong \bC_K(\barQ).
\end{equation*}
Moreover, if $f \colon Q \to \bP(V)$ is a standard morphism of~$Q$,
there is a unique commutative diagram of vector spaces
\begin{equation*}
\xymatrix@C=4em{
0 \ar[r] &
K \ar@{^{(}->}[r] \ar@{=}[d] & 
\tW \ar@{->>}[r] \ar@{->>}[d]_{\bC_K(\bar{f})} &
\tW/K \ar[r] \ar@{->>}[d]_{\bar{f}} &
0
\\
0 \ar[r] &
K \ar@{^{(}->}[r] & 
W \ar@{->>}[r] \ar@{^{(}->}[d] &
W/K \ar[r] &
0
\\
&& V
}
\end{equation*}
with surjective morphism $\bar{f}$ such that $f$ is 
the composition $Q \xrightarrow{\ \bC_K(\bar{f})\vert_Q \ } \bP(W) \hookrightarrow \bP(V)$. 
Moreover, one of the following two possibilities hold: 
\begin{enumerate}
\item 
\label{type:embedding}
The map $\bar{f}$ is an isomorphism. In this case, $f$ is an embedding. 
\item 
\label{type:covering}
The spaces $\ker(\bar{f})$ and $\ker(\bC_K(\bar{f}))$ are $1$-dimensional, 
and the corresponding points of the projective spaces~$\bP(\tW)$ and~$\bP(\tW/K)$ 
do not lie on the quadrics $Q$ and $\barQ$ respectively. 
In this case, $f$ is a double covering onto~$\bP(W) \subset \bP(V)$. 
\end{enumerate}
\end{lemma}
\begin{proof}
We define $K$ to be the kernel of the quadratic form on~$\tW$ corresponding to~$Q$ 
and~$\barQ$ to be the quadric corresponding to the induced quadratic form on $\tW/K$.
We set 
\begin{equation*}
W = \operatorname{im}\left(V^\vee = H^0(\bP(V),\cO_{\bP(V)}(1)) \xrightarrow{\ f^*\ } H^0(Q,\cO_Q(1)) = \tW^\vee\right)^\vee.
\end{equation*}
This gives a factorization of $f^*$ as a composition $V^\vee \to W^\vee \to \tW^\vee$
and we define the maps in the middle column of the diagram as the dual maps.
The rest is clear.
\end{proof}

We call the quadric $\barQ$ above the \emph{base quadric} of $Q$.
Moreover, if~\eqref{type:embedding} holds we say $f$ is of \emph{embedding type},
and if~\eqref{type:covering} holds we say $f$ is of \emph{covering type}. 

Next we define some useful numerical invariants of a standard morphism of a quadric. 
In the definition below we use the notation introduced in Lemma~\ref{lemma:standard-morphism}.

\begin{definition}
\label{def:numerical-invariants}
Let $f \colon Q \to \bP(V)$ be a standard morphism of a quadric. Then: 
\begin{itemize} 
\item $r(Q) = \dim\tW - \dim K$ denotes the rank of $Q$, i.e. the rank of the quadratic form on~$\tW$ corresponding to $Q$.  
\item $p(Q) \in \{0,1\}$ denotes the \emph{parity} of $r(Q)$, 
i.e. $p(Q) = r(Q) \pmod 2$. 
\item $k(Q) = \dim K$.  
\item $c(f) = \dim V - \dim W$ denotes the codimension of the linear span
$\langle {f(Q)} \rangle \subset \bP(V)$. 
\item $t(f) = \dim\tW - \dim W \in \{0,1\}$ denotes the \emph{type} of $Q$, 
defined by 
\begin{equation*}
t(f) = 
\begin{cases}
0 & \text{ if $f$ is of embedding type}, \\ 
1 & \text{ if $f$ is of covering type}. 
\end{cases}
\end{equation*}
\end{itemize}
\end{definition}

Note that our convention that $Q$ is integral is equivalent to $r(Q) \ge 3$.

\begin{remark}
\label{remark-Q-invariants}
As indicated by the notation, $r(Q)$, $p(Q)$, and $k(Q)$ depend only on~$Q$, 
while~$c(f)$ and $t(f)$ are invariants of the morphism~$f$. 
We note the relations
\begin{align}
\label{dimQ-quadric}
\dim(Q) & = r(Q) + k(Q) - 2 ,  \\
\label{dimV-quadric}
\dim(V) & = r(Q) + k(Q) + c(f) - t(f) .
\end{align} 
Moreover, if $\barQ$ is the base quadric of $Q$, we have $r(Q) = r(\barQ)$ and $p(Q) = p(\barQ)$.
\end{remark}

Using the identification of Lemma~\ref{lemma:standard-morphism} of a quadric $Q$ with the cone over a smooth quadric, 
we see that the corresponding resolved cone gives a resolution of~$Q$.
We call the induced map
\begin{equation}
\label{equation-tQ}
\pi \colon \tC_K(\barQ) \to \bC_{K}(\barQ) = Q
\end{equation} 
the \emph{standard geometric resolution} of $Q$. 
Note that this map is nothing but the blowup of $Q$ in 
its singular locus $\Sing(Q) = \bP(K)$. 

\begin{definition}
\label{definition-standard-cat-res-Q}
Let $f \colon Q \to \bP(V)$ be a standard morphism of a quadric.   
Using the above notation, the \emph{standard categorical resolution} of $Q$ over~$\bP(V)$ 
is the Lefschetz category $\cQ$ over~$\bP(V)$ defined as the categorical cone over the base quadric $\barQ$: 
\begin{equation*}
\cQ = \cC_K(\barQ)/\bP(V) . 
\end{equation*} 
\end{definition}

In Lemma~\ref{lemma-cQ-lc} we explicitly describe the Lefschetz components of $\cQ$, 
and in Lemma~\ref{lemma-cQ-categorical-resolution} we justify calling $\cQ$ a categorical resolution. 

\begin{lemma}
\label{lemma-cQ-lc}
Let $f \colon Q \to \bP(V)$ be a standard morphism of a quadric. 
Let $\cS$ be a spinor bundle on the base quadric of~$Q$. 
Then the standard categorical resolution $\cQ$ of $Q$ over $\bP(V)$ 
is smooth and proper over~$\bk$, with a strong, moderate Lefschetz structure of 
length $\dim(Q)$.
If~$k = k(Q)$ and $p = p(Q)$, then its nonzero Lefschetz components are given by  
\begin{equation*}
\cQ_i = \begin{cases}
\langle \cS, \cO \rangle & \text{for $|i| \leq k+1-p$} , \\ 
\langle \cO \rangle &  \text{for $k+1-p < |i| \leq \dim(Q) - 1$} ,  
\end{cases}
\end{equation*}
and its nonzero primitive Lefschetz components are given by 
\begin{equation*}
\fq_i = 
\begin{cases}
\langle \cO \rangle & \text{if $i = -(\dim(Q) - 1)$} , \\ 
\langle {\cS'}^{\svee} \rangle & \text{if $i = -(k+1-p)$} , \\ 
\langle \cS \rangle & \text{if $i = k+1-p$} , \\ 
\langle \cO \rangle & \text{if $i = \dim(Q) - 1$}, 
\end{cases}
\end{equation*}
where $\cS'$ is described in Remark~\textup{\ref{remark:other-spinor-center}}.
\end{lemma}

\begin{proof}
Combine Theorem~\ref{theorem-cone-lef-cat}, Lemma~\ref{lemma-cC-smooth-proper}, Lemma~\ref{lemma-smooth-Q-lc}, 
and the formula~\eqref{dimQ-quadric}.
\end{proof}

\begin{remark}
In Lemma~\ref{lemma-cQ-lc} and below, we tacitly identify the objects~$\cO$, $\cS$, and~$\cS'$ 
on the base quadric~{$\barQ$} of~$Q$ with their pullbacks to~{$\cQ \subset \tC_K(\barQ)$}. 
\end{remark}

The next lemma justifies our terminology by showing that $\cQ$ is 
a weakly crepant categorical resolution in the sense of Definition~\ref{def:categorical-resolution}.
Recall that by definition the standard categorical resolution of a quadric is a subcategory 
of the derived category of the standard geometric resolution~\eqref{equation-tQ}. 

\begin{lemma}
\label{lemma-cQ-categorical-resolution}
Let $f \colon Q \to \bP(V)$ be a standard morphism of a quadric, 
with standard geometric resolution $\pi \colon \tQ  = \tC_K(\barQ) \to Q$ and 
standard categorical resolution $\cQ$ over $\bP(V)$. 
Then~$\pi_*$ and~$\pi^*$ restrict to functors 
\begin{equation*}
\pi_* \colon \cQ \to \Db(Q) \qquad \text{and} \qquad \pi^* \colon \Perf(Q) \to \cQ, 
\end{equation*} 
which give $\cQ$ the structure of a {weakly} crepant categorical resolution of $Q$. 

Furthermore, the object $\cR = \pi_* \cEnd( \cO \oplus \cS) \in \Db(Q)$ 
is a coherent sheaf of~$\cO_Q$-algebras on~$Q$ of finite homological dimension, 
and there is an equivalence 
\begin{equation*} 
\cQ \simeq \Db(Q, \cR) , 
\end{equation*} 
where $\Db(Q, \cR)$ denotes the bounded derived category of coherent $\cR$-modules on $Q$. 
\end{lemma}

\begin{proof}
Recall that the morphism $\pi$ is the blowup with center at $\bP(K)$.
Consider the blowup diagram 
\begin{equation*}
\vcenter{\xymatrix{
E \ar[r]  \ar[d]_{\pi_E} & \tQ \ar[d]^\pi \\ 
\bP(K) \ar[r] & Q
}}
\end{equation*} 
where $E$ is the exceptional divisor; note moreover that $E \cong \bP(K) \times \barQ$, where $\barQ$ is the base quadric of $Q$. 
It is easy to see that $\pi$ is a resolution of rational singularities (recall that the rank of $\barQ$ is assumed to be greater than~2). 
Moreover, $\pi_E^*(\Perf(\bP(K)))$ is contained in the Lefschetz center 
\begin{equation*}
\Perf(\bP(K)) \otimes \langle \cS, \cO \rangle \subset \Perf(\bP(K) \times \barQ) \simeq \Perf(E).
\end{equation*} 
Hence by \cite[Theorem 4.4]{kuznetsov2008lefschetz} the functors $\pi^*$ and $\pi_*$ 
indeed give $\cQ$ the structure of a categorical resolution of $Q$. 
Moreover, $Q$ is Gorenstein and a direct computation shows that 
\begin{equation*}
K_{\tQ} = \pi^*(K_Q) + (\dim(\barQ) - 1)E .
\end{equation*} 
Note that $\dim(\barQ)$ is the length of the Lefschetz decomposition of $\Perf(E)$ above. 
Hence \cite[Proposition 4.5]{kuznetsov2008lefschetz} shows that $\cQ$ is a {weakly} crepant categorical resolution of $Q$. 
By an argument similar to \cite[Proposition 7.1]{kuznetsov2008lefschetz}, 
the bundle $\cO \oplus \cS$ on $\tQ$ is tilting over $Q$ (i.e. the 
derived pushforward~$\pi_* \cEnd(\cO \oplus \cS)$ is a pure sheaf) if and only if for all $t \geq 0$ we have 
\begin{equation*}
\rH^{>0}(\barQ, \cEnd( \cO \oplus \cS )(t\barH)) = 0.  
\end{equation*}
A computation shows that this vanishing holds, and then the rest of the lemma 
follows from~\cite[Theorem 5.2]{kuznetsov2008lefschetz}. 
\end{proof}

\begin{remark}
The last statement of Lemma~\ref{lemma-cQ-categorical-resolution} shows that 
$\cQ$ can also be considered as a noncommutative resolution in the sense of 
Van den Bergh~\cite{vdb-flops, vdb-crepant}. 
\end{remark}

The following lemma relates standard categorical resolutions of quadrics to geometry
and shows that~$\cQ$ is ``birational'' to~$Q$ over~$\bP(V)$.

\begin{lemma}
\label{lemma-cQ-resolution}
Let $f \colon Q \to \bP(V)$ be a standard morphism of a quadric. 
Let $\cQ$ be the standard categorical resolution of 
$Q$ over $\bP(V)$. 
Let $U = \bP(V) \setminus f(\Sing(Q))$. 
\begin{enumerate}
\item \label{cQ-resolution} 
There is a $\bP(V)$-linear functor $\Perf(Q) \to \cQ$ whose 
base change to $U$ gives an equivalence
\begin{equation*}
\cQ_U \simeq \Perf(Q_U) . 
\end{equation*}

\item \label{cQ-fiber-product} 
Let $\cA$ be a $\bP(V)$-linear category supported over $U$. 
Then there is an equivalence 
\begin{equation*}
\cA \otimes_{\Perf(\bP(V))} \cQ \simeq \cA \otimes_{\Perf(\bP(V))} \Perf(Q). 
\end{equation*}
In particular, if $\cA = \Perf(X)$ for a scheme $X$ over $\bP(V)$ supported over $U$, then 
\begin{equation*}
\cA \otimes_{\Perf(\bP(V))} \cQ \simeq \Perf{\left(X \times_{\bP(V)} Q \right)}. 
\end{equation*}
\end{enumerate}
\end{lemma}

\begin{proof}
Part \eqref{cQ-resolution} follows from Proposition~\ref{proposition-cC-tC-agree}
because the morphism~$\tC_{K}(\barQ) \to \bP(V)$ factorizes as 
\begin{equation*}
\tC_K(\barQ) \to \bC_{K}(\barQ) = Q \xrightarrow{\, f \,} \bP(V), 
\end{equation*}
where the first map is the blowup in $\bP(K) = \Sing(Q)$. 
Part \eqref{cQ-fiber-product} follows from Lemma~\ref{lemma-C-support-fiber} 
and part~\eqref{cQ-resolution}. 
\end{proof}

\subsection{Generalized quadratic duality and HPD} 
\label{ssubsection-generalized-quadric-duality}  

Our goal in this subsection is to define a geometric duality operation on 
standard morphisms of quadrics, which after passing to  
standard categorical resolutions corresponds to the 
operation of taking the HPD category. 

The desired duality operation will be defined using a 
combination of the following three operations. 

\begin{definition}
\label{def:generalized-duality}
Let $f \colon Q \to \bP(V)$ be a standard morphism of a quadric.  
\begin{itemize}

\item If $f \colon Q \to \bP(V)$ is of embedding type, we denote by 
\begin{equation*}
f^{\svee} : Q^{\svee} \to \bP(V^{\svee})
\end{equation*} 
the embedding of the classical projective dual of $Q \subset \bP(V)$. 

\item If {$f \colon Q \to \bP(V)$} is of embedding type, we define 
\begin{equation*}
\cov{f} \colon \cov{Q} \to \bP(V)
\end{equation*} 
as the composition of 
the double cover $\cov{Q} \to \langle Q \rangle$ branched along $Q \subset \langle Q \rangle$ 
with the embedding $\langle Q \rangle \hookrightarrow \bP(V)$. 

\item If {$f \colon Q \to \bP(V)$} is of covering type, we define 
\begin{equation*}
\br{f} \colon \br{Q} \to \bP(V)
\end{equation*}
as the composition of the inclusion $\br{Q} \hookrightarrow \langle {f(Q)} \rangle$ 
of the branch divisor of the double cover~\mbox{$Q \to \langle {f(Q)} \rangle$} 
with the embedding~$\langle {f(Q)} \rangle \hookrightarrow \bP(V)$.  
\end{itemize}
\end{definition}

\begin{remark}
\label{remark:duality-spaces}
Let $f \colon Q \to \bP(V)$ be a standard morphism of a quadric, and recall the 
canonical diagram of vector spaces associated to $f$ in Lemma~\ref{lemma:standard-morphism}. 
The operations of Definition~\ref{def:generalized-duality} affect this diagram as follows. 
\begin{itemize}
\item If $f \colon Q \to \bP(V)$ is of embedding type, then its classical projective dual 
can be described as follows. 
The filtration $0 \subset K \subset W \subset V$ 
gives a filtration $0 \subset W^{\perp} \subset K^{\perp} \subset \vV$ by taking orthogonals. 
The pairing between $V$ and $\vV$ induces a nondegenerate pairing between~$W/K$ and $K^{\perp}/W^{\perp}$, 
and hence an isomorphism $K^{\perp}/W^{\perp} \cong (W/K)^{\svee}$. 
Via this isomorphism, the base quadric $\barQ \subset \bP(W/K)$ of $Q$ 
corresponds to a quadric in \mbox{$\barQ^{\svee} \subset \bP(K^{\perp}/W^{\perp})$} (its projective dual), 
and $Q^{\svee} = \bC_{W^{\perp}}(\barQ^{\svee}) \subset \bP(\vV)$.  
Thus, the operation $f \mapsto f^{\svee}$ replaces $W$ by $K^{\perp}$ and $K$ by 
$W^{\perp}$. 

\item If $f \colon Q \to \bP(V)$ is of embedding type, then note that $\tW = W$. 
The operation $f \mapsto \cov{f}$ replaces $\tW$ by $W \oplus \bk$, and 
keeps $W$ and $K$ fixed. 

\item Similarly, if $f \colon Q \to \bP(V)$ is of covering type, then 
the operation $f \mapsto \br{f}$ replaces $\tW$ by~$W$, and keeps $W$ and $K$ fixed.  
\end{itemize}
\end{remark}

\begin{remark}
All of the above operations preserve the integrality of $Q$, 
except for the branch divisor operation in case $r(Q) = 3$ and $f$ is a morphism of covering type. 
Indeed, this follows from the formulas: 
\begin{equation*}
r(Q^{\svee}) = r(Q), \quad r(\cov{Q}) = r(Q) + 1, \quad r(\br{Q}) = r(Q) -1. 
\end{equation*} 
Note, however, that the above operations are defined even for non-integral quadrics.
\end{remark} 

The next definition is modeled on the cases considered in Theorem~\ref{theorem-HPD-quadrics}. 

\begin{definition}
\label{definition:generalized-duality}
Let $f \colon Q \to \bP(V)$ be a standard morphism of a quadric. 
The \emph{generalized dual} of $f$ is the 
standard morphism 
\begin{equation*}
f^{\natural} \colon Q^\natural \to \bP(V^{\svee})
\end{equation*} 
of the quadric $Q^{\natural}$ defined as follows:  
\begin{itemize}
\item If $f \colon Q \to \bP(V)$ is of embedding type, then: 
\begin{itemize}
\item[$\diamond$] If $r(Q)$ is even, we set $Q^\natural = Q^{\svee}$ and 
$f^{\natural} = f^{\svee} \colon Q^{\natural} \to \bP(\vV)$. 
\item[$\diamond$] If $r(Q)$ is odd, we set $Q^\natural = \cov{(Q^{\svee})}$ 
and $f^{\natural} = \cov{(f^{\svee})} \colon Q^{\natural} \to \bP(\vV)$. 
\end{itemize}
\item If $f \colon Q \to \bP(V)$ is of covering type, then: 
\begin{itemize}
\item[$\diamond$] If $r(Q)$ is even, we set $Q^\natural = (\br{Q})^{\svee}$ 
and $f^{\natural} = (\br{f})^{\svee} \colon Q^{\natural} \to \bP(\vV)$. 
\item[$\diamond$] If $r(Q)$ is odd, we set $Q^\natural = \cov{((\br{Q})^{\svee})}$ 
and $f^{\natural} = \cov{((\br{f})^{\svee})} \colon Q^{\natural} \to \bP(\vV)$. 
\end{itemize}
\end{itemize}
In other words, we first pass to a morphism of the embedding type (by taking the branch divisor if necessary), 
then apply projective duality, and then if necessary go to the double covering.
\end{definition}

\begin{remark} 
\label{remark-natural-invariants}
Using the description of Remark~\ref{remark:duality-spaces} it is easy to check that
generalized duality affects the numerical invariants of $f$ described in Definition~\ref{def:numerical-invariants} as follows: 
\begin{equation*} 
r(Q^{\natural})  = r(Q) + p(Q) - t(f), \quad 
p(Q^{\natural})  = t(f) , \quad 
k(Q^{\natural})  = c(f), \quad c(f^{\sharp})  = k(Q), \quad  t(f^{\natural}) = p(Q). 
\end{equation*}
In particular, note that generalized duality preserves the integrality of $Q$. 
Note also that by~\eqref{dimQ-quadric} we have  
\begin{equation}
\label{dimQd-quadric}
\dim(Q^{\natural}) = r(Q) + p(Q) + c(f) - t(f) - 2. 
\end{equation} 
\end{remark}

\begin{remark}
\label{remark:parity-dimension}
By~\eqref{dimQ-quadric}, \eqref{dimQd-quadric}, and \eqref{dimV-quadric} we have 
\begin{align*}
\dim(Q^\natural) + \dim(Q)
& = (r(Q) + k(Q) + c(f) - t(f)) + (r(Q) + p(Q) - 4) \\
& = \dim(V) + (r(Q) + p(Q) - 4) , 
\end{align*}
which is congruent to $\dim(V) \bmod 2$ since by definition $p(Q) \equiv r(Q) \bmod 2$.
This means that if~$\dim(V)$ is even, then the parities of the dimensions of $Q^\natural$ and $Q$ are the same,
and if~$\dim(V)$ is odd, then the parities are opposite.
\end{remark}

Now we can bootstrap from Theorem~\ref{theorem-HPD-quadrics} to  
a result for arbitrary standard morphisms. 

\begin{theorem}
\label{theorem:hpd-quadrics-general}
Let $f \colon Q \to \bP(V)$ be a standard morphism of a quadric.  
Let $\cQ$ be the standard categorical resolution of $Q$ over $\bP(V)$.
Then the HPD category $\cQ^{\hpd}$ is Lefschetz equivalent to the 
standard categorical resolution of $Q^{\natural}$ over $\bP(V^{\svee})$. 
\end{theorem} 

\begin{proof}
Follows from Theorem~\ref{theorem-HPD-quadrics}, Theorem~\ref{theorem-cones-HPD}, and the definitions. 
\end{proof}


\subsection{The quadratic HPD theorem}
\label{subsection-quadratic-HPD} 

Now we can prove our quadratic HPD theorem, by combining the above results with 
the nonlinear HPD Theorem~\ref{theorem-nonlinear-HPD}.

\begin{theorem}
\label{theorem-quadric-intersection}
Let $\cA$ be a right strong, moderate Lefschetz category over $\bP(V)$. 
Let 
\begin{equation*}
f \colon Q \to \bP(V) \qquad \text{and} \qquad f^{\natural} \colon Q^{\natural} \to \bP(V^{\svee}) 
\end{equation*} 
be a standard map of a quadric and its generalized dual. 
Let $\cQ$ be the standard categorical resolution of $Q$ over $\bP(V)$, 
and let $\cQ^{\natural}$ be the standard categorical resolution of $Q^{\natural}$ over $\bP(V^{\svee})$. 
Let $\cS \in \cQ$ and $\cS^{\natural} \in \cQ^{\natural}$ be the pullbacks of 
spinor bundles on the base quadrics of~$Q$ and $Q^{\hpd}$. 
Let $H$ and $H'$ denote the hyperplane classes on $\bP(V)$ and $\bP(\vV)$. 
Let 
\begin{equation*}
N = \dim(V), \quad m = \length(\cA), \quad m^\hpd
= \length(\cAd), \quad 
d = \dim(Q), \quad d^\natural = \dim(Q^{\natural}). 
\end{equation*} 
Then there are semiorthogonal decompositions 
\begin{align*}
\begin{split}
\cA \otimes_{\Perf(\bP(V))} \cQ = \Big \langle  &  \cK_Q(\cA) ,  \\
& \cA_{d^\natural}(H) {\otimes} \langle \cS \rangle,  \dots,  \cA_{m-1}((m-d^\natural)H) {\otimes} \langle \cS \rangle, \\
& \cA_{N-d}(H) {\otimes} \langle \cO \rangle,  \dots,    \cA_{m-1}((m+d-N)H) {\otimes} \langle \cO \rangle  \Big\rangle.
\end{split}\\
\begin{split}
\cAd  \otimes_{\Perf(\bP(V^{\svee}))} \cQ^{\natural} = \Big\langle  
& \cAd_{1-m^\hpd}((N-d^\natural-m^\hpd)H') {\otimes} \langle \cO \rangle, \dots,   \cAd_{d^\natural-N}(-H') {\otimes} \langle \cO \rangle ,   \\
& \cAd_{1-m^\hpd}((d-m^\hpd)H') {\otimes} \langle (\cS^{\natural})^{\svee} \rangle, \dots,  
\cAd_{-d}(-H') {\otimes} \langle (\cS^{\natural})^{\svee} \rangle,  \\ 
& \cK'_{Q^{\natural}}(\cAd)   \Big\rangle , 
\end{split}
\end{align*}
and an equivalence of categories $\cK_Q(\cA) \simeq \cK'_{Q^{\natural}}(\cAd)$.
\end{theorem}

\begin{proof}
We apply the nonlinear HPD Theorem~\ref{theorem-nonlinear-HPD} in case ${\cA^1} = \cA$, ${\cA^2} = \cQ$;
it gives semiorthogonal decompositions~\eqref{CQ} and~\eqref{DQ} and the equivalence; 
so we only have to check that the components $\cJ_i = \cJ(\cA, \cQ)_i$ of~\eqref{CQ} 
and $\cJd_j = \cJ(\cAd, \cQ^{\hpd})_j$ of~\eqref{DQ} have the prescribed form.

By \cite[Lemma 3.24]{categorical-joins} and 
Lemma~\ref{lemma-cQ-lc}, for $i \geq N$ we have  
\begin{equation*}
\cJ_i = \big \langle \cA_{i-k+p-2} \otimes \langle \cS \rangle, \cA_{i-d} \otimes \langle \cO \rangle \big \rangle 
\subset \cA \otimes_{\Perf(\bP(V))} \cQ . 
\end{equation*} 
Combined with the observation that $k-p+2 = N - d^\natural$ by~\eqref{dimV-quadric} and~\eqref{dimQd-quadric},  
it follows that the semiorthogonal decomposition~\eqref{CQ} takes the claimed form.
Using the expression for the numerical invariants of $Q^{\natural}$ in terms of
those of $Q$ (Remark~\ref{remark-natural-invariants}), 
it follows similarly that the semiorthogonal decomposition~\eqref{DQ} takes the 
claimed form. 
\end{proof}

It is natural to combine Theorem~\ref{theorem-quadric-intersection} 
with the result of Lemma~\ref{lemma-cQ-resolution}\eqref{cQ-fiber-product} 
that provides the left hand sides of the semiorthogonal decompositions with a clear geometric meaning. 
In the next section we use this to derive the applications promised in~\S\ref{subsection-intro-applications}. 


\section{Applications}
\label{section:applications}

In this section we collect some applications of the quadratic HPD Theorem~\ref{theorem-quadric-intersection} obtained above.
In~\S\ref{subsection:duality-gm} we prove the duality conjecture for Gushel--Mukai varieties and in~\S\ref{subsection-sGM}
we discuss and prove its spin analogue.
We continue to assume the base scheme~$S$ is the spectrum of an algebraically closed field $\bk$ of characteristic not equal to $2$.

\subsection{Duality of Gushel--Mukai varieties}
\label{subsection:duality-gm}

We will prove \cite[Conjecture 3.7]{kuznetsov2016perry} on the duality of Gushel--Mukai varieties. 
For context and consequences of this conjecture, see the discussion in \S\ref{subsection-intro-applications}. 
The definition of this class of varieties from~\cite{debarre2015kuznetsov} can be rephrased as follows; 
note that unlike~\cite{debarre2015kuznetsov}, by convention we require GM varieties to have dimension at least $2$.

\begin{definition}
\label{def:gm}
A \emph{Gushel--Mukai \textup(GM\textup) variety} is a 
dimensionally transverse fiber product
\begin{equation*}
X = \Gr(2,V_5) \times_{\bP(\wedge^2V_5)} Q , 
\end{equation*}
where $V_5$ is a $5$-dimensional vector space, 
$\Gr(2,V_5) \to \bP(\wedge^2V_5)$ is the Pl\"{u}cker embedding of the Grassmannian 
of $2$-dimensional subspaces of $V_5$, and 
$Q \to \bP(\wedge^2V_5)$ is a standard morphism of a quadric with $\dim Q \geq 5$. 
\end{definition}

In~\cite{kuznetsov2016perry} a semiorthogonal decomposition of $\Db(X) = \Perf(X)$ 
(appearing as~\eqref{GMX} below) for any smooth GM variety was constructed, and in particular, an interesting subcategory
\begin{equation*}
\cK(X) \subset \Perf(X),
\end{equation*}
(called the \emph{GM category} of $X$) was defined.
In~\cite{kuznetsov2016perry} GM~categories were thoroughly studied, and in particular, it was shown that a GM category~$\cK(X)$
is either a K3 category or an Enriques type category, depending on whether~$\dim(X)$ is even or odd.

On the other hand, in~\cite{debarre2015kuznetsov} GM varieties were classified.
In particular, in~\cite[Theorem~3.6]{debarre2015kuznetsov} to every GM variety
there was associated its \emph{Lagrangian data set}, which consists of a triple of vector spaces~$(V_6(X),V_5(X),A(X))$,
where: 
\begin{itemize}
\item 
$V_6(X)$ is a 6-dimensional vector space;
\item 
$V_5(X) \subset V_6(X)$ is a hyperplane; and
\item 
$A(X) \subset \wedge^3V_6(X)$ is a Lagrangian subspace.
\end{itemize}
Here we endow the 20-dimensional space $\wedge^3V_6(X)$ with a symplectic form 
given by wedge product $\wedge^3V_6(X) \otimes \wedge^3V_6(X) \to \wedge^6V_6(X) \cong \bk$.
The form depends on a choice of isomorphism above, but the property of being Lagrangian does not.

Conversely, to every triple $(V_6,V_5,A)$ as above two \emph{GM intersections} $X_{A,V_5}^{\mathrm{ord}}$ and $X_{A,V_5}^{\mathrm{spe}}$
were associated. 
Both $X_{A,V_5}^{\mathrm{ord}}$ and $X_{A,V_5}^{\mathrm{spe}}$ have the form
\begin{equation*}
X_{A,V_5} = \Gr(2,V_5) \times_{\bP(\wedge^2V_5)} Q,
\end{equation*}
and their type (\emph{ordinary} or \emph{special}) corresponds 
to the type of the morphism $Q \to \bP(\wedge^2V_5)$ (embedding or covering).
Note that a GM intersection $X_{A,V_5}$ is not necessarily dimensionally transverse (so it is not necessarily a GM variety).

Furthermore, in~\cite[Theorem~3.16]{debarre2015kuznetsov} it was shown that if~$X$ is a smooth GM variety of dimension $d_X \ge 3$ 
then the Lagrangian $A(X) \subset \wedge^3V_6(X)$ \emph{contains no decomposable vectors},
i.e., $\bP(A(X))$ does not intersect $\Gr(3, V_6(X)) \subset \bP(\wedge^3 V_6(X))$.
Conversely, if $A \subset \wedge^3V_6$ contains no decomposable vectors,
then for any $V_5 \subset V_6$ both GM intersections $X_{A,V_5}$ are smooth GM varieties.
Note that the dimension of the two types of $X_{A,V_5}$ differs by 1 and depends on $V_5$.

This already shows that the Lagrangian $A(X)$ controls many important properties of a GM variety $X$.
Motivated by this and a birationality result~\cite[Corollary~4.16 and Theorem~4.20]{debarre2015kuznetsov},  
we introduced in~\cite[Definition~3.5]{kuznetsov2016perry} the notions of generalized duality and partnership of GM varieties 
(generalizing~\cite[Definition~3.22 and~3.26]{debarre2015kuznetsov}). 

\begin{definition}
\label{def:gm-duality}
Let $X_1$ and $X_2$ be GM varieties such that $\dim(X_1) \equiv \dim(X_2) \pmod 2$.
\begin{itemize}
\item 
$X_1$ and $X_2$ are \emph{generalized partners} if there exists an isomorphism $V_6(X_1) \cong V_6(X_2)$
identifying $A(X_1) \subset \wedge^3V_6(X_1)$ with $A(X_2) \subset \wedge^3V_6(X_2)$.
\item 
$X_1$ and $X_2$ are \emph{generalized dual} if there exists an isomorphism $V_6(X_1) \cong V_6(X_2)^{\svee}$
identifying~$A(X_1) \subset \wedge^3V_6(X_1)$ with $A(X_2)^\perp \subset \wedge^3V_6(X_2)^{\svee}$.
\end{itemize}
\end{definition}

The duality conjecture~\cite[Conjecture~3.7]{kuznetsov2016perry} predicted 
that for {(smooth)} GM varieties whose associated Lagrangians contain no decomposable vectors,
being generalized partners or duals implies an equivalence of GM categories.
A special case was proved in~\cite[Theorem~4.1]{kuznetsov2016perry}; 
below we prove the conjecture in full generality. 

The idea of our proof is as follows. 
First, we note that the $\Gr(2,V_5)$ factor in the fiber product defining a GM variety in Definition~\ref{def:gm}
is homologically projectively self-dual, so one can use the quadratic HPD theorem
to relate the derived categories of two GM varieties.
Second, we note that generalized duality of the quadric factors corresponds to generalized duality of the corresponding GM intersections.
Finally, we iterate equivalences of GM categories obtained in this way to prove the conjecture.

Now we start implementing this approach.
First, recall the homological projective self-duality of $\Gr(2,V_5)$: 

\begin{theorem}[{\cite[Section~6.1 and Theorem~1.2]{kuznetsov2006hyperplane}}]
\label{theorem-HPD-Gr} 
Let $\cU$ and $\cU'$ be the tautological rank~$2$ subbundles on $\Gr(2,V_5)$ and $\Gr(2,\vV_5)$. 
Then $\Perf(\Gr(2,V_5))$ and $\Perf(\Gr(2,\vV_5))$ have the structure of 
strong, moderate Lefschetz categories over $\bP(\wedge^2V_5)$ and $\bP(\wedge^2\vV_5)$, respectively, 
of length $5$, with Lefschetz components given by 
\begin{equation*}
\cA_i  = \langle \cO , \cUv \rangle \quad \text{and} \quad 
\cA'_i  = \langle \cU', \cO \rangle   
\end{equation*} 
for $|i| \leq 4$. Moreover, there is an equivalence 
\begin{equation*}
\Perf(\Gr(2,V_5))^{\hpd} \simeq \Perf(\Gr(2,\vV_5))
\end{equation*}
of Lefschetz categories over $\bP(\wedge^2\vV_5)$. 
\end{theorem} 

Now we apply Theorem~\ref{theorem-quadric-intersection} to GM varieties. 

\begin{theorem} 
\label{theorem-GM-duality}
Let 
\begin{equation}
\label{eq:gm-varieties}
X = \Gr(2, V_5) \times_{\bP(\wedge^2V_5)} Q \quad \text{and} \quad 
Y = \Gr(2, \vV_5) \times_{\bP(\wedge^2\vV_5)} Q^\hpd 
\end{equation}
be smooth GM varieties of dimensions $d_X \geq 2$ and $d_Y \geq 2$, 
where $Q \to \bP(\wedge^2V_5)$ is a standard morphism of a quadric 
and $Q^{\hpd} \to \bP(\wedge^2\vV_5)$ is its generalized dual. 
Let $\cU_X$ and $\cU_Y$ denote the pullbacks of $\cU$ and $\cU'$ 
to $X$ and $Y$, and let $\cO_X(1)$ and $\cO_Y(1)$ denote the pullbacks 
of the $\cO(1)$ line bundles on $\bP(\wedge^2V_5)$ and $\bP(\wedge^2\vV_5)$. 
Then there are semiorthogonal decompositions 
\begin{align}
\label{GMX}
\Perf(X) & = \langle \cK(X), \cO_X(1), \cU_X^{\svee}(1), \dots, \cO_X(d_X-2), \cU_X^{\svee}(d_X-2) \rangle , \\ 
\label{GMY} 
\Perf(Y) & = \langle \cU_Y(2-d_Y), \cO_Y(2-d_Y), \dots, \cU_Y(-1), \cO_Y(-1), 
\cK'(Y) \rangle , 
\end{align} 
and an equivalence $\cK(X) \simeq \cK'(Y)$. 
\end{theorem}

Before giving a proof note that if $d_X \le 1$ then
$Y$ is necessarily singular.
Indeed, in this case we have $\dim(Q) \le 4$, hence $c(Q) \ge 4$, hence $k(Q^\hpd) \ge 4$,
hence $\Gr(2,\vV_5)$ intersects the image of the singular locus of $Q^\hpd$, hence $Y$ is singular. 

\begin{proof}
This is a combination of Theorem~\ref{theorem-quadric-intersection}, Theorem~\ref{theorem-HPD-Gr}, 
and Lemma~\ref{lemma-cQ-resolution}\eqref{cQ-fiber-product}. 
Indeed, 
the smoothness of $X$ and $Y$ implies that the Grassmannians in~\eqref{eq:gm-varieties} do not intersect 
the singular loci of the quadrics, so by Lemma~\ref{lemma-cQ-resolution}\eqref{cQ-fiber-product} we have
\begin{align*}
\Perf(\Gr(2, V_5^{\hphantom{\svee}})) \otimes_{\Perf(\bP(\wedge^2V_5^{\hphantom{\svee}}))} \cQ^{\hphantom{\natural}} & \simeq \Perf(X), \\ 
\Perf(\Gr(2, \vV_5)) \otimes_{\Perf(\bP(\wedge^2\vV_5))} \cQ^{\natural} & \simeq \Perf(Y). 
\end{align*}
We just need to show the semiorthogonal decompositions of Theorem~\ref{theorem-quadric-intersection} take the prescribed form.

The length of the Lefschetz decompositions of $\Perf(\Gr(2,V_5))$ is $m = 5$,
and the codimension of $\Gr(2,V_5)$ in $\bP(\wedge^2V_5)$ is~3,
so by dimensional transversality
\begin{equation*}
d^\hpd = \dim (Q^\natural) = d_Y + 3 \ge 5.
\end{equation*}
Thus~$m - d^\hpd \le 0$, hence $\cS$ does not show up in the semiorthogonal decomposition of $\Perf(X)$.
The same argument shows that $(\cS^{\natural})^{\svee}$ does not show up in the decomposition of~$\Perf(Y)$.
Similarly, $N = \dim \wedge^2V_5 = 10$ and $d = \dim(Q) = d_X + 3$, hence
\begin{equation*}
m + d - N = 5 + (d_X + 3) - 10 = d_X - 2,
\end{equation*}
and so the Lefschetz components $\cA_i = \langle \cO,\cU^{\svee} \rangle$ of $\Perf(\Gr(2,\vV_5))$ appear
$d_X - 2$ times in the decomposition of $\Perf(X)$. 
The same argument shows that $\langle \cU_Y, \cO_Y \rangle$ appears $d_Y - 2$ times in the decomposition of $\Perf(Y)$.
Hence the semiorthogonal decompositions of Theorem~\ref{theorem-quadric-intersection} take the prescribed form. 
\end{proof}

Now we are ready to prove the duality conjecture.

\begin{corollary}[{\cite[Conjecture~3.7]{kuznetsov2016perry}}] 
\label{corollary-duality-GM} 
Let $X$ and $Y$ be {smooth} GM varieties
whose associated Lagrangian subspaces $A(X)$ and $A(Y)$ 
do not contain decomposable vectors.
If $X$ and $Y$ are generalized partners or duals, then there is an equivalence $\cK(X) \simeq \cK(Y)$.
\end{corollary}

By~\cite[Theorem~3.16]{debarre2015kuznetsov} the assumption 
that the Lagrangian subspace~$A(X)$ does not contain decomposable vectors
holds automatically unless~$X$ is a special GM surface or an ordinary GM surface with singular Grassmannian hull.

\begin{proof}
First assume $X$ and $Y$ are generalized duals. 
Under the isomorphism $V_6(X) \cong V_6(Y)^{\svee}$, the hyperplane $V_5(Y) \subset V_6(Y)$ 
corresponds to a point ${\mathbf{q}_Y} \in \bP(V_6(X))$. 
Further assume that
\begin{equation}
\label{eq:bq-v5-condition}
{\mathbf{q}_Y} \notin \bP(V_5(X)). 
\end{equation} 
Then $X$ and $Y$ can be obtained as in Theorem~\ref{theorem-GM-duality} 
from an appropriate pair of generalized dual quadrics, 
cf.~\cite[Proposition~3.28]{debarre2015kuznetsov}. 
Twisting the decomposition~\eqref{GMX} by~$\cO_X(-1)$ 
shows that~$\cK(X)$ is equivalent to the GM category of~$X$, 
as defined in~\cite[Definition 2.5]{kuznetsov2016perry}. 
On the other hand, 
twisting the decomposition~\eqref{GMY} by~$\cO_Y(1)$ and 
using~\cite[(2.20) and~(2.21)]{kuznetsov2016perry} shows that~$\cK'(Y)$ is equivalent to the GM category of~$Y$. 
Thus Theorem~\ref{theorem-GM-duality} gives the result under our above assumptions. 

Next assume $X$ and $Y$ are generalized partners. 
Choose a point $\mathbf{p} \in \bP(V_6(X)) = \bP(V_6(Y))$ away from the hyperplanes $\bP(V_5(X))$ and $\bP(V_5(Y))$.
Let $V_5 \subset V_6 := V_6(X)^\vee$ be the corresponding hyperplane,
and let $A = A(X)^{\perp} \subset \wedge^3 V_6$. 
Then by \cite[Theorem 3.10]{debarre2015kuznetsov} there is a GM variety $Z$ 
such that $(V_6(Z), V_5(Z), A(Z)) = (V_6, V_5, A)$.
Since $A$ does not contain decomposable vectors, $Z$ is smooth by \cite[Theorem~3.16]{debarre2015kuznetsov}.
By construction, $X$ and $Y$ are both generalized duals of $Z$ satisfying
the extra assumption considered in the previous paragraph. 
Thus~$X$ and~$Y$ have equivalent GM categories. 

The final case to consider is when $X$ and $Y$ are generalized duals, 
and under the isomorphism $V_6(X) \cong V_6(Y)^{\svee}$, the hyperplane 
$V_5(Y) \subset V_6(Y)$ corresponds to a point $\mathbf{q} \in \bP(V_6(X))$ which 
lies in $\bP(V_5(X))$. 
In this case, arguing as in the previous paragraph
we can construct a GM variety $Y'$ generalized dual to $X$, such that under the isomorphism 
$V_6(X) \cong V_6(Y')^{\svee}$, the hyperplane 
$V_5(Y') \subset V_6(Y')$ corresponds to a point $\mathbf{q}' \in \bP(V_6(X)) \setminus \bP(V_5(X))$. 
Then the GM category of $X$ is equivalent to that of $Y'$ by the first paragraph. 
On the other hand, by construction $Y'$ and $Y$ are generalized partners, so their GM categories 
are equivalent by the previous paragraph. 
\end{proof}

\begin{remark}
{As we already mentioned, the duality conjecture does not cover the case of smooth GM varieties $X$ whose Lagrangians contain decomposable vectors,
i.e. all special GM surfaces and some ordinary GM surfaces.
Note that such $X$ have neither generalized partners nor duals of dimension greater than~2 and moreover $\cK(X) = \Perf(X)$.
Thus, extending~\cite[Conjecture~3.7]{kuznetsov2016perry} it is natural to ask: 
If $X$ and $Y$ are smooth GM surfaces which are generalized partners or duals 
and whose Lagrangians contain decomposable vectors, 
then is there an equivalence $\Perf(X) \simeq \Perf(Y)$?

We expect that the answer is positive, although the argument of Corollary~\ref{corollary-duality-GM} does not work
as the crucial assumption~\eqref{eq:bq-v5-condition} never holds for smooth generalized partners
when the corresponding Lagrangian has decomposable vectors.
There are two possible strategies to work around this. 

First, one can also consider (mildly) singular GM varieties and prove that if~$X$ and~$Y$ are generalized dual 
with~$X$ singular and~$Y$ smooth, then~$\cK(Y)$ is a categorical resolution of~$\cK(X)$
(cf.\ the proof of Theorem~\ref{theorem-singular-sGM} below).
Then, however, one  will have to check that the resolutions~$\cK(Y_1)$ and~$\cK(Y_2)$ 
obtained from two smooth generalized duals~$Y_1$ and~$Y_2$ of the same~$X$ are equivalent.
This seems possible, but does not fit into the scope of this paper.

Another possibility is to extend the arguments of~\cite[Theorem~4.7 and Propositions~4.13 and~4.19]{debarre2015kuznetsov}
to show that if~$X$ and~$Y$ are two smooth generalized partners or duals, then~$X$ is birational to~$Y$, and hence~$X \cong Y$.}
\end{remark}

\subsection{Duality of spin GM varieties} 
\label{subsection-sGM} 

It is well known that the Grassmannian $\Gr(2,V_5)$ shares many properties with its elder brother, the 
spinor tenfold $\OGrp(5, V_{10})$. 
The content of this subsection provides yet another confirmation of this principle.

Let $V_{10}$ be $10$-dimensional vector space. 
Recall the orthogonal Grassmannian $\OGr({5},V_{10})$ of ${5}$-dimensional isotropic subspaces 
for a nondegenerate quadratic form on $V_{10}$ has two connected components, 
$\OGrp({5}, V_{10})$ and $\OGrm({5}, V_{10})$, which are abstractly isomorphic. 
The Pl\"{u}cker embedding $\OGrp({5},V_{10}) \to {\Gr(5,V_{10}) \to {}} \bP(\wedge^{{5}}V)$ is 
given by the square of the generator of $\Pic(\OGrp({5},V_{10}))$; 
the generator itself gives an embedding $\OGrp(5, V_{10}) \subset \bP(\sSs)$, 
where~$\sSs$ is the $16$-dimensional half-spinor representation of $\Spin(V_{10})$. 
Note that
\begin{equation*}
\dim\OGrp(5, V_{10}) = 10 
\qquad \text{and} \qquad 
\operatorname{codim}_{\bP(\sSs)} \OGrp(5, V_{10}) = 5.
\end{equation*}
The spinor tenfold $\OGrp(5, V_{10}) \subset \bP(\sSs)$ shares a very 
special property with the Grassmannian~\mbox{$\Gr(2, V_5) \subset \bP(\wedge^2V_5)$}: 
both are projectively self-dual, and even homologically projectively self-dual.
More precisely, the classical projective dual variety of $\OGrp(5, V_{10}) \subset \bP(\sSs)$ 
is given by the spinor embedding~$\OGrm(5, V_{10}) \subset \bP(\sSs^{\svee})$. 
This lifts to the homological level as follows. 

\begin{theorem}[{\cite[Section~6.2 and Theorem~1.2]{kuznetsov2006hyperplane}}]
\label{theorem-HPD-OGr} 
Let $\cU$ and $\cU'$ be the tautological rank $5$ subbundles on $\OGrp(5,V_{10})$ 
and $\OGrm(5,V_{10})$. 
Then $\Perf(\OGrp(5,V_{10}))$ and $\Perf(\OGrm(5,V_{10}))$ have the structure of 
strong, moderate Lefschetz categories over the spinor spaces~$\bP(\sSs)$ and~$\bP(\sSs^{\svee})$, respectively, 
of length $8$, with Lefschetz components given by 
\begin{equation*}
\cA_i  = \langle \cO , \cUv \rangle \quad \text{and} \quad 
\cA'_i  = \langle \cU', \cO \rangle   
\end{equation*} 
for $|i| \leq 7$. Moreover, there is an equivalence 
\begin{equation*}
\Perf(\OGrp(5,V_{10}))^{\hpd} \simeq \Perf(\OGrm(5,V_{10}))
\end{equation*}
of Lefschetz categories over $\bP(\sSs^{\svee})$.  
\end{theorem} 

This parallel between $\Gr(2,V_5)$ and $\OGrp(5,V_{10})$ motivates the following definition. 
\begin{definition}
A smooth \emph{spin GM variety} is a smooth dimensionally transverse fiber product 
\begin{equation*}
X = \OGrp(5, V_{10}) \times_{\bP(\sS_{16})} Q , 
\end{equation*}
where $Q \to \bP(\sSs)$ is a standard morphism of a quadric. 
\end{definition} 

We note that if $X$ is a smooth spin GM variety of dimension $d$, 
then $X$ is a Fano variety of Picard number~$1$, coindex~$4$, and degree~$24$
for $d \geq 4$, and $X$ is a polarized Calabi--Yau threefold of Picard number~$1$ and degree~$24$ for $d = 3$. 
Using Theorem~\ref{theorem-HPD-OGr} in place of Theorem~\ref{theorem-HPD-Gr}, 
the argument of Theorem~\ref{theorem-sGM-duality} proves the following spin analogue. 

\begin{theorem} 
\label{theorem-sGM-duality}
Let 
\begin{equation*}
X = \OGrp(5,V_{10}) \times_{\bP(\sSs)} Q \quad \text{and} \quad 
Y = \OGrm(5, V_{10}) \times_{\bP(\sSs^{\svee})} Q^\hpd 
\end{equation*}
be smooth spin GM varieties of dimensions $d_X \geq 4$ and $d_Y \geq 4$, 
where $Q \to \bP(\sSs)$ is a standard morphism of a quadric and $Q^{\hpd} \to \bP(\sSs^{\svee})$ is its generalized dual. 
Let $\cU_X$ and $\cU_Y$ denote the pullbacks of $\cU$ and $\cU'$ 
to $X$ and $Y$, and let $\cO_X(1)$ and $\cO_Y(1)$ denote the pullbacks 
of the $\cO(1)$ line bundles on $\bP(\sSs)$ and $\bP(\sSs^{\svee})$. 
Then there are semiorthogonal decompositions 
\begin{align}
\label{sGMX}
\Perf(X) & = \langle \cK(X), \cO_X(1), \cU_X^{\svee}(1), \dots, \cO_X(d_X-3), \cU_X^{\svee}(d_X-3) \rangle , \\ 
\label{sGMY} 
\Perf(Y) & = \langle \cU_Y(3-d_Y), \cO_Y(3-d_Y), \dots, \cU_Y(-1), \cO_Y(-1), 
\cK'(Y) \rangle , 
\end{align} 
and an equivalence $\cK(X) \simeq \cK'(Y)$. 
\end{theorem}

We call the category $\cK(X)$ occurring in \eqref{sGMX} a \emph{spin GM category}. 
Spin GM categories should be thought of as $3$-dimensional counterparts of GM categories. 
Indeed, whereas a GM category is always (fractional) Calabi--Yau of dimension~$2$, 
a spin GM category is (fractional) Calabi--Yau of dimension~$3$ by~\cite[Remark~4.9]{kuznetsov2015calabi}. 
More precisely, if~$X$ is odd-dimensional then~$\cK(X)$ is $3$-Calabi--Yau, while if $X$ is 
even-dimensional then the Serre functor of $\cK(X)$ is given by $\rS_{\cK(X)} = \sigma \circ [3]$ 
where $\sigma$ is an involutive autoequivalence of $\cK(X)$. 

Furthermore, one can develop the notion of a Lagrangian data set for spin GM varieties: 
this should consist of triples $(V_{12},V_{10},A)$, 
where $V_{12}$ is a 12-dimensional space endowed with a non-degenerate quadratic form,
$V_{10} \subset V_{12}$ a 10-dimensional subspace to which the quadratic from restricts non-degenerately,
and $A \subset {\sS_{32}}(V_{12})$ is a Lagrangian subspace in the 32-dimensional half-spinor representation of $\Spin(V_{12})$
(note that ${\sS_{32}}(V_{12})$ has a natural $\Spin(V_{12})$-invariant symplectic form). 
Then the notion of generalized spin partnership and duality for spin GM varieties can 
be defined analogously to Definition~\ref{def:gm-duality}, and 
the argument of Corollary~\ref{corollary-duality-GM} would prove that spin GM categories of 
generalized spin partners or duals are equivalent.

It would be interesting to investigate the rationality question for spin GM varieties 
in relation to Theorem~\ref{theorem-sGM-duality}, following the GM case 
discussed in \cite[\S3]{kuznetsov2016perry}. 
The critical case is when~$X$ has dimension $5$; then $\cK(X)$ is a $3$-Calabi--Yau category, 
which is conjecturally equivalent to the derived category of a Calabi--Yau threefold if and only if $X$ is rational. 
Theorem~\ref{theorem-sGM-duality}, however, 
does \emph{not} give examples of this sort. 
Indeed, if $Y = \OGrm(5, V_{10}) \times_{\bP(\sSs^{\svee})} Q^\hpd$ is a smooth 
GM variety of dimension $3$, then it is easy to see that 
$X = \OGrp(5,V_{10}) \times_{\bP(\sSs)} Q$ cannot be smooth because 
$\OGrp(5,V_{10})$ must meet the singular locus of $Q$.  
More generally, we have the following result. 

\begin{lemma}
\label{lemma-sGM-not-geometric}
Let $X$ be a smooth spin GM variety whose dimension is odd and at least $5$. 
\begin{enumerate}
\item \label{HH-cKX}
The $0$-th Hochschild homology of $\cK(X)$ is given by $\HH_0(\cK(X)) \cong \bk^2$. 
\item \label{cKX-not-geometric}
$\cK(X)$ is not equivalent to the derived category of a projective variety. 
\end{enumerate} 
\end{lemma}

\begin{proof}
We first note that $\OGrp(5,V_{10})$ has cohomology of Tate type, 
and Poincar\'{e} polynomial given by
\begin{equation*}
1 + t^2 + t^4 + 2 t^6 + 2 t^8 + 2 t^{10} + 2 t^{12} + 2 t^{14} + t^{16} + t^{18} + t^{20},  
\end{equation*} 
see~\cite[\S2.2]{double-spinor} or \cite[Corollary 3.8]{kuznetsov2018spinor}. 
The Lefschetz hyperplane theorem combined with the HKR theorem 
then determines $\HH_0(X)$, and the claimed formula for $\HH_0(\cK(X))$ follows from the 
additivity of Hochschild cohomology \cite[Theorem 7.3]{kuznetsov2009hochschild}. 

If $\cK(X) \simeq \Db(M)$ for a projective variety $M$, then $M$ is smooth by \cite[Lemma D.22]{kuznetsov2006hyperplane}. 
Moreover, $M$ must have dimension $3$ since $\cK(X)$ is $3$-Calabi--Yau. 
The HKR theorem then implies $\dim \HH_0(M) \geq 4$, contradicting part~\eqref{HH-cKX}. 
\end{proof}

Nonetheless, by considering a mild degeneration of the situation of Theorem~\ref{theorem-sGM-duality}, 
we can find spin GM fivefolds whose category $\cK(X)$ admits a geometric resolution of singularities. 
Recall from Definition~\ref{def:categorical-resolution} and Remark~\ref{remark-nc-cat-res} 
the notion of a weakly crepant categorical resolution. 

\begin{theorem}
\label{theorem-singular-sGM}
Let $K \subset W \subset \sS_{16}$ be generic subspaces with $\dim(K) = 6$ and~$ \dim(W) = 12$,
and let $\barQ \subset \bP(W/K)$ be a general smooth quadric.
Set $Q = \bC_K(\barQ)$ and let 
\begin{equation*}
f \colon Q \to \bP(W) \to \bP(\sS_{16})
\end{equation*}
be the induced morphism.
Let 
\begin{equation*}
X = \OGrp(5,V_{10}) \times_{\bP(\sSs)} Q \quad \text{and} \quad 
Y = \OGrm(5, V_{10}) \times_{\bP(\sSs^{\svee})} Q^\hpd . 
\end{equation*}
Then $X$ is a spin GM fivefold with $12$ nodal singularities and $Y$ is a smooth spin GM threefold. 
Moreover, there is a semiorthogonal decomposition 
\begin{equation}
\label{cKX-singular}
\Db(X) = \langle \cKb(X), \cO_X(1), \cU_X^{\svee}(1), \cO_X(2), \cU_X^{\svee}(2)  \rangle , 
\end{equation} 
and $\Db(Y)$ is a weakly crepant categorical resolution of $\cK(X) = \cKb(X) \cap \Perf(X)$. 
\end{theorem}

\begin{proof}
The spinor embedding $\OGrp(5,V_{10}) \subset \bP(\sSs)$ has degree $12$ and codimension $5$. 
Thus for general $K$ the intersection $Z = \OGrp(5, V_{10}) \cap \bP(K)$ consists of $12$ reduced points, 
say~$z_1,\dots,z_{12}$,
and the dual intersection $\OGrm(5, V_{10}) \cap \bP(K^\perp)$ is a smooth fourfold.
Furthermore, for general $W$ containing $K$ the intersection $\OGrm(5, V_{10}) \cap \bP(W^\perp)$ is empty,
and the intersection $\OGrp(5, V_{10}) \cap \bP(W)$ is a smooth sixfold containing $Z$.

The embedded tangent space to $\OGrp(5, V_{10})$ at the point $z_i$ corresponds
to an 11-di\-men\-sional subspace $T_i \subset \sSs$ such that $\dim(T_i \cap K) = 1$.
The intersection $T_i \cap W$ corresponds to the embedded tangent space to $\OGrp(5, V_{10}) \cap \bP(W)$ at $z_i$,
hence $\dim(T_i \cap W) = 7$
and the natural map $T_i \cap W \to W/K$ is surjective with kernel $T_i \cap K$.
For any smooth quadric~$\barQ \subset \bP(W/K)$ its strict preimage in $\bP(T_i \cap W)$ 
is the cone over $\barQ$ with vertex $z_i = \bP(T_i \cap K)$
and it is identified with the normal cone to $X$ at $z_i$, hence $z_i$ is a node.
This proves that for $K$ and~$W$ chosen as above and any smooth~$\barQ$ the intersection $X$ has nodes at points of $Z$.
Also, for general $\barQ$ by Bertini's theorem $X$ is smooth away from $Z$ and $Y$ is smooth.
Thus~$Y$ is a smooth spin GM threefold. 

The semiorthogonal decomposition~\eqref{cKX-singular} is induced by the Lefschetz 
decomposition of the spinor tenfold~$\OGrp(5,V_{10})$, cf.~\cite[Lemma 5.5]{cyclic-covers}. 

Let $\cQ$ denote the standard categorical resolution of $Q$ over $\bP(\sSs)$. 
Then arguing as in Theorem~\ref{theorem-sGM-duality}, we see that Theorem~\ref{theorem-quadric-intersection} 
gives a semiorthogonal decomposition 
\begin{equation}
\label{OGrp-cQ}
\Perf(\OGrp(5,V_{10})) \otimes_{\Perf(\bP(\sSs))} \cQ = 
\llangle 
\wtilde{\cK}(X), 
\langle \cO(1) , \cU^{\svee}(1) \rangle \otimes  \langle \cO \rangle, \langle \cO(2), \cU^{\svee}(2) \rangle \otimes \langle \cO \rangle 
\rrangle
\end{equation} 
and an equivalence $\Perf(Y) \simeq \wtilde{\cK}(X)$. 
Note that $\Perf(Y) = \Db(Y)$ since $Y$ is smooth. 
Thus to finish it suffices to show that $\wtilde{\cK}(X)$ is a weakly crepant categorical resolution of $\cKb(X)$. 

The functors $\pi_* \colon \cQ \to \Db(Q)$ and $\pi^* \colon \Perf(Q) \to \cQ$ of Lemma~\ref{lemma-cQ-categorical-resolution} 
induce by base change along $\OGrp(5,V_{10}) \to \bP(\sSs)$ functors
\begin{align*}
\pi_* &\colon \Perf(\OGrp(5,V_{10})) \otimes_{\Perf(\bP(\sSs))} \cQ \to \Db(X),\\
\pi^* &\colon \Perf(X) \to \Perf(\OGrp(5,V_{10})) \otimes_{\Perf(\bP(\sSs))} \cQ,
\end{align*}
such that $\pi^*$ is left and right adjoint to $\pi_*$ and $\pi_* \circ \pi^* \cong \id$.
Thus, these functors provide the category $\Perf(\OGrp(5,V_{10})) \otimes_{\Perf(\bP(\sSs))} \cQ$ 
with the structure of a weakly crepant categorical resolution of $X$.

Furthermore, \eqref{cKX-singular} also induces a semiorthogonal decomposition
\begin{equation}
\label{eq:sod-perf-x}
\Perf(X) = \langle \cK(X), \cO_X(1), \cU_X^{\svee}(1), \cO_X(2), \cU_X^{\svee}(2)  \rangle , 
\end{equation}
where $\cK(X) = \cKb(X) \cap \Perf(X)$. 
Indeed, by \cite[Proposition 4.1]{kuznetsov-base-change} it is enough to show that the components of~\eqref{cKX-singular} are admissible; 
this is clear for the exceptional objects that appear, and then follows for $\cKb(X)$ by Serre duality and 
the fact that $X$ is Gorenstein.
Clearly, $\pi^*$ takes the four exceptional objects from~\eqref{eq:sod-perf-x} 
to the four exceptional objects in~\eqref{OGrp-cQ}.
Therefore, from full faithfulness it follows that $\pi^*$ takes the right orthogonal $\cK(X)$ of the former 
to the right orthogonal $\wtilde{\cK}(X)$ of the latter, and thus defines a functor
\begin{equation*}
\pi^* \colon \cK(X) \to \wtilde{\cK}(X).
\end{equation*}
Similarly, by adjunction it follows that the right adjoint functor $\pi_*$ takes 
$\wtilde{\cK}(X)$ to $\cKb(X)$, and hence defines a functor 
\begin{equation*}
\pi_* \colon \wtilde{\cK}(X) \to \cKb(X).
\end{equation*}
Since we have already shown that $\pi_*$ and $\pi^*$ provide 
$\Perf(\OGrp(5,V_{10})) \otimes_{\Perf(\bP(\sSs))} \cQ$ with the structure of a weakly crepant 
categorical resolution of $X$, it follows that $\wtilde{\cK}(X)$ is a weakly crepant 
categorical resolution of $\cK(X)$ via these functors.
\end{proof} 

The proof of the theorem shows that the resolution $\tX \to X$ given by blowing 
up the singular points of $X$ has a semiorthogonal decomposition consisting of 
exceptional objects and the derived category of the Calabi--Yau threefold $Y$. 
Thus, the philosophy of~\cite{kuznetsov2010derived, kuznetsov2015rationality} suggests 
that $\tX$ (and therefore $X$) should be rational. 
We will prove this as a consequence of the following.

\begin{lemma}
\label{lemma-singular-sGM-fibration}
If $X$ is as in Theorem~\textup{\ref{theorem-singular-sGM}}, 
then there is a resolution of singularities $X' \to X$ and a morphism $X' \to \bP^2$ 
whose general fiber is a smooth Fano threefold of Picard number $1$, degree $12$, and index $1$. 
Moreover, the morphism $X' \to \bP^2$ has~$12$ sections.
\end{lemma}

\begin{proof}
The following argument is inspired by~\cite[Lemma~4.1]{debarre2015kuznetsov}. 

Recall that the kernel space $K$ of the quadric $Q$ defining~$X$ is 6-dimensional and its span~$W$ is 12-dimensional.
Therefore, the maximal isotropic spaces for $Q$ are $9$-dimensional. 
Let $I \supset K$ be a generic such space. 
Then linear projection from $\bP^8 = \bP(I) \subset \bP(W)$ induces a morphism 
\begin{equation*}
q \colon X' \to \bP(W/I) = \bP^2, 
\end{equation*}
where 
$X'$ is the blowup of $X$ along 
\begin{equation*}
X \cap \bP(I) = \OGrp(5,V_{10}) \cap \bP(I). 
\end{equation*} 
The genericity of $I$ guarantees that $X'$ is smooth.

The fibers of $q$ can be described as follows: 
a point $b \in \bP^2$ corresponds to a $\bP^9_b \subset \bP(W)$ containing $\bP(I)$; 
we have $Q \cap \bP^9_b = \bP(I) \cup \bP(I_b)$ where $I_b$ is the residual isotropic 
space for $Q$; and the fiber over $b$ is $q^{-1}(b) = \OGrp(5,V_{10}) \cap \bP(I_b)$. 
Thus the general fiber of $q$ is a smooth threefold given as a codimension $7$ 
linear section of $\OGrp(5,V_{10}) \subset \bP(\sSs)$, i.e. a threefold of the claimed type. 

Furthermore, since any maximal isotropic subspace in $Q$ contains $K$, we have $K \subset I_b$, hence
\begin{equation*}
\OGrp(5,V_{10}) \cap \bP(K) \subset \OGrp(5,V_{10}) \cap \bP(I_b).
\end{equation*}
It remains to note that the left side is a set of~12 reduced points;
each of these points gives a section of the morphism $X' \to \bP^2$.
\end{proof}

\begin{corollary}
\label{corollary-singular-sGM-rational}
If $X$ is as in Theorem~\textup{\ref{theorem-singular-sGM}} and the base field~$\bk$ is algebraically closed of characteristic~$0$, 
then $X$ is rational. 
\end{corollary}

\begin{proof}
By Lemma~\ref{lemma-singular-sGM-fibration}, it suffices to show that a smooth Fano threefold 
of Picard number~$1$, degree~$12$, and index~$1$ is rational if it has a rational point. 
This holds by \cite{kuznetsov-prokhorov}. 
\end{proof}

We note that 
Theorem~\ref{theorem-singular-sGM} can be thought of as giving a 
conifold transition from the noncommutative Calabi--Yau threefold $\cK(X)$
to the Calabi--Yau threefold $Y$. 
In the spirit of Reid's fantasy \cite{reid-fantasy}, we pose the following (loosely formulated) 
question: 

\begin{question}
Can any noncommutative Calabi--Yau 
threefold be connected to a geometric Calabi--Yau threefold via a 
sequence of degenerations and crepant resolutions? 
\end{question} 

Theorem~\ref{theorem-singular-sGM} gives a positive answer to this question for spin GM categories of spin GM fivefolds, 
and similar arguments also give a positive answer for spin GM varieties of dimension~$7$ or~$9$. 
The results of~\cite{favero-kelly} give a positive answer for noncommutative Calabi--Yau threefolds 
associated to cubic sevenfolds (using, however, degenerations with 
worse-than-nodal singularities). 
It would be interesting to investigate more examples, in particular 
the list of noncommutative Calabi--Yau threefolds given in \cite[\S4.5]{kuznetsov2015calabi}. 

\begin{remark}
If $Y$ is a smooth \emph{strict} Calabi--Yau threefold in the sense that $\omega_Y \cong \cO_Y$ 
and~\mbox{$\rH^j(Y, \cO_Y) = 0$} for $j = 1,2$, then the HKR theorem shows 
that~$\HH^2(Y) \cong \rH^1(\rT_Y)$, so~$\Db(Y)$ has no noncommutative infinitesimal deformations. 
Thus to have a hope of connecting a noncommutative Calabi--Yau threefold to a geometric Calabi--Yau 
threefold, we should indeed allow more operations than deformations. 
\end{remark}


\appendix

\section{HPD results} 
\label{appendix} 

In this appendix we provide some material on semiorthogonal decompositions and HPD that is used in the body of the paper.
In~\S\ref{subsection:local-criterion} we establish a local criterion for an equivalence of $T$-linear categories.
In~\S\ref{subsection-HPD-projection} we describe the behavior of HPD under linear projections.

\subsection{A local criterion for an equivalence}
\label{subsection:local-criterion}

The main result of this subsection is the following proposition. 

\begin{proposition}
\label{proposition-equivalence-local} 
Let $\phi \colon \cC \to \cD$ be a $T$-linear functor and let~$\cA \subset \cC$ be a $T$-linear subcategory.
Assume that either $\phi$ has a left adjoint and $\cA$ is left admissible, or $\phi$ has a right adjoint and $\cA$ is right admissible.
Let also~$\cB \subset \cD$ be a $T$-linear subcategory which is either right or left admissible. 
Let ${U \to T}$ 
be an fpqc cover, and let~$\phi_{U} \colon \cC_{U} \to \cD_{U}$ 
denote the induced functor.  
Then $\phi$ induces an equivalence~$\cA \simeq \cB$ if and only if $\phi_{U}$ induces 
an equivalence $\cA_{U} \simeq \cB_{U}$. 
\end{proposition}

As we observe in Corollary~\ref{corollary-equivalence-lef-cat-local}, 
the proposition also implies a local criterion for a functor between Lefschetz categories to be a Lefschetz equivalence.

We build up some preliminary results before giving the proof. 
If $\cC$ is a $T$-linear category and $T' \to T$ is a morphism, 
we write $C \vert_{T'}$ for the image of $C \in \cC$ under the canonical functor~$\cC \to \cC_{T'}$ induced by pullback. 

\begin{lemma}
\label{lemma-object-vanish-locally}
Let $\cC$ be a $T$-linear category, and let $C \in \cC$. 
Let ${U \to T}$ be an fpqc cover. 
Then~$C \simeq 0$ if and only if $C \vert_{U} \simeq 0$. 
\end{lemma}

\begin{proof}
The forward implication is obvious. 
Conversely, by the K\"{u}nneth formula in the form of \cite[Lemma~2.10]{NCHPD}, we have 
\begin{equation*}
\cHom_T(C,C)\vert_{U} \simeq \cHom_{U}(C\vert_{U}, C\vert_{U}) , 
\end{equation*} 
where $\cHom_T(C,C) \in \QCoh(T)$ is the mapping object defined in \cite[\S2.3.1]{NCHPD}. 
Hence if $C \vert_{U} \simeq 0$, we have $\cHom_T(C,C)\vert_{U} \simeq 0$. 
Then $\cHom_T(C,C) \simeq 0$ since the vanishing of an object in $\QCoh(T)$ can 
be checked fpqc locally, and therefore $C \simeq 0$. 
\end{proof} 

\begin{corollary}
\label{corollary-functor-vanish-locally} 
Let $\phi \colon \cC \to \cD$ be a $T$-linear functor. 
Let ${U \to T}$ be an fpqc cover. 
Then~$\phi \simeq 0$ if and only if $\phi_{U} \simeq 0$. 
\end{corollary}

\begin{proof}
The forward implication is obvious. 
Conversely, we must show that $\phi(C) \simeq 0$ for all~$C \in \cC$ if $\phi_{U} \simeq 0$.  
For this, just note that $\phi(C)\vert_{U} \simeq \phi_{U}(C\vert_{U})$ and 
apply Lemma~\ref{lemma-object-vanish-locally}. 
\end{proof} 

\begin{lemma} 
\label{lemma-factor-locally} 
Let $\phi \colon \cC \to \cD$ be a $T$-linear functor. 
Let $\cB \subset \cD$ be a $T$-linear subcategory which is left or right admissible. 
Let ${U \to T}$ be an fpqc cover. 
Then $\phi$ factors through the inclusion $\cB \subset \cD$ if and only if $\phi_{U}$ 
factors through the inclusion $\cB_{U} \subset \cD_{U}$.  
\end{lemma} 

\begin{proof}
We consider the case where $\cB$ is left admissible; the right admissible case is similar. 
Since $\cB$ is left admissible, its left orthogonal ${}^\perp\cB$ is right admissible,
hence its inclusion functor~$j \colon {^\perp}\cB \to \cD$ has a right adjoint $j^!$ whose kernel is $\cB$.
Therefore $\phi$ factors through $\cB \subset \cD$ 
if and only if the composition $j^! \circ \phi$ vanishes. 
By Corollary~\ref{corollary-functor-vanish-locally}, this composition vanishes if and only if its base 
change to $U$ vanishes. 
But this base change identifies with $j_{U}^! \circ \phi_{U}$
where~$j^!_{U}$ is the right adjoint to the inclusion ${^{\perp}}\cB_{U} \subset \cD_{U}$ (see~\cite[Lemma~3.15]{NCHPD}), 
and hence vanishes if and only if $\phi_{U}$ factors through $\cB_{U} \subset \cD_{U}$.
\end{proof}

\begin{proof}[Proof of Proposition~\textup{\ref{proposition-equivalence-local}}]
We consider the left adjoints case of the proposition; the right adjoints case is similar. 
First assume $\cA = \cC$ and $\cB = \cD$. 
Note that a functor with a left adjoint is an equivalence if and only if 
the cones of the unit and counit of the adjunction vanish. 
If~$\psi$ denotes the cone of the unit or counit for the adjoint pair $(\phi, \phi^*)$, 
then $\psi_{U}$ is the cone 
of the unit or counit for the adjoint pair $(\phi_{U}, \phi^*_{U})$ 
(cf.~\cite[Lemma 2.12]{NCHPD} or~\cite[\S2.6]{kuznetsov2006hyperplane}).  
Hence applying Corollary~\ref{corollary-functor-vanish-locally} proves the lemma in this case. 

Now consider the case of general $\cA$ and $\cB$. 
Denote by $\alpha \colon \cA \to \cC$ and $\beta \colon \cB \to \cD$ the inclusions. 
If $\phi_{U}$ induces an equivalence $\cA_{U} \simeq \cB_{U}$, then by Lemma~\ref{lemma-factor-locally} 
the composition of functors $\phi \circ \alpha \colon \cA \to \cD$ factors through $\cB \subset \cD$, i.e. 
there is a functor $\phi_{\cA} \colon \cA \to \cB$ 
such that~$\phi \circ \alpha = \beta \circ \phi_{\cA}$. 
We want to show $\phi_{\cA}$ is an equivalence. 
But $\phi_{\cA}$ admits a left adjoint, namely $\alpha^* \circ \phi^* \circ \beta$, 
and $(\phi_{\cA})_{U} \colon \cA_{U} \to \cB_{U}$ is an equivalence, 
so we conclude by the case handled above. 
\end{proof}

Let $S' \to S$ be a morphism of schemes, and let~$V_{S'}$ denote the pullback 
of a vector bundle~$V$ on~$S$ to~$S'$. 
Then if $\cA$ is a Lefschetz category over $\bP(V)$, the base change~$\cA_{S'}$ 
is naturally a Lefschetz category over $\bP(V_{S'})$ with 
Lefschetz center given by the base change~\mbox{$(\cA_0)_{S'} \subset \cA_{S'}$}. 
This follows from a combination of \cite[Lemma 2.4]{categorical-joins} and 
\cite[Lemmas 3.15 and 3.17]{NCHPD}. 
Proposition~\ref{proposition-equivalence-local} then implies the following.

\begin{corollary}
\label{corollary-equivalence-lef-cat-local} 
Let $\cA$ and $\cB$ be Lefschetz categories over $\bP(V)$. 
Let $\phi \colon \cA \to \cB$ be a $\bP(V)$-linear functor which admits a left or right adjoint. 
Let ${U \to S}$ be an fpqc cover of $S$, and let~$\phi_{U} \colon \cA_{U} \to \cB_{U}$ denote 
the induced functor.  
Then $\phi$ is an equivalence of Lefschetz categories over $\bP(V)$ if and only if $\phi_{U}$ is 
an equivalence of Lefschetz categories over $\bP(V_{U})$. 
\end{corollary}

The following related result is useful for establishing the existence of a semiorthogonal 
decomposition, by reduction to a local situation. 
\begin{lemma}
\label{lemma-sod-local}
Let $\cC$ be a $T$-linear category, and let $\cA_1, \dots, \cA_n \subset \cC$ be a 
sequence of right or left admissible $T$-linear subcategories. 
Let $U \to T$ be an fpqc cover. 
Then $\cC = \langle \cA_1, \dots, \cA_n \rangle$ 
if and only if $\cC_{U} = \langle \cA_{1U}, \dots, \cA_{nU} \rangle$. 
\end{lemma}

\begin{proof}
The forward implication holds by \cite[Lemma 3.15]{NCHPD}. 
Conversely, assume we have a semiorthogonal decomposition 
$\cC_{U} = \langle \cA_{1U}, \dots, \cA_{nU} \rangle$. 
Then the argument of Lemma~\ref{lemma-object-vanish-locally} shows that the 
categories $\cA_i \subset \cC$ are semiorthogonal. 
Assume the categories $\cA_i \subset \cC$ are right admissible (a similar argument works in the 
left admissible case). 
Then setting  
$\cD = \langle \cA_1, \dots, \cA_n \rangle {^\perp}$ we have a 
semiorthogonal decomposition $\cC = \langle \cD, \cA_1, \dots, \cA_n \rangle$. 
But $\cD = 0$ by Lemma~\ref{lemma-object-vanish-locally}. 
\end{proof}

\subsection{HPD over quotients and subbundles} 
\label{subsection-HPD-projection} 

Given a surjective morphism $\tV \to V$ of vector bundles with kernel $K$, 
we consider the corresponding rational map $\bP(\tV) \dashrightarrow \bP(V)$
and denote by~$U = \bP(\tV) \setminus \bP(K) \subset \bP(\tV)$ the open subset on which it is regular.
If $\cA$ is a~$\bP(\tV)$-linear category supported over $U$ (i.e. if the restriction functor $\cA \to \cA_U$ is an equivalence), 
then it inherits a natural $\bP(V)$-linear structure via the linear projection map. 
In this situation, we can ask for a relation between HPD with respect to the two linear structures on $\cA$. 
Before answering this, we make some preliminary observations. 

\begin{definition}
\label{definition-base-extension} 
Let $\cC$ be a $T$-linear category, and let $T \to T'$ be a morphism of schemes.  
We write $\cC/T'$ for $\cC$ regarded as a $T'$-linear category via the pullback functor 
$\Perf(T') \to \Perf(T)$, and say $\cC/T'$ is obtained from $\cC$ by \emph{extending the base scheme} 
along $T \to T'$. 
\end{definition}

\begin{remark}
\label{remark-base-extension}
If $\cA$ is a Lefschetz category over $\bP(V)$ and $V \to V'$ is an embedding of vector bundles, 
then the category $\cA/\bP(V')$ is naturally a Lefschetz category over $\bP(V')$, with the same 
center. Moreover, this operation preserves (right or left) strongness and moderateness of Lefschetz 
categories. 
\end{remark}  

\begin{lemma}
\label{lemma-C-support-fiber} 
Let $T$ be a scheme and let $U \subset T$ be an open subscheme. 
Let $\cC$ be a $T$-linear category which is supported over $U$. 
Then for any $T$-linear category $\cD$, there is a canonical $T$-linear 
equivalence 
\begin{equation*}
\cC \otimes_{\Perf(T)} \cD \simeq \cC \otimes_{\Perf(T)} \cD_U. 
\end{equation*}
\end{lemma}

\begin{proof}
We have equivalences 
\begin{equation*}
\cC \otimes_{\Perf(T)} \cD \simeq 
\cC_U \otimes_{\Perf(T)} \cD \simeq 
\cC \otimes_{\Perf(T)} \Perf(U) \otimes_{\Perf(T)} \cD  \simeq 
\cC \otimes_{\Perf(T)} \cD_U.  \qedhere
\end{equation*}
\end{proof}

Now we can answer the question posed above about HPD under linear projection. 
Note that the surjection $\tV \to V$ induces an embedding of bundles $\vV \to \tV^{\svee}$, so that $\bP(\vV) \subset \bP(\tV^{\svee})$.

\begin{proposition}
\label{proposition-HPD-projection}
Let $\cA$ be a Lefschetz category over $\bP(\tV)$ with center $\cA_0$.
Assume $\tV \to V$ is a surjection of vector bundles with kernel $K$ such that $\cA$ is supported over $\bP(\tV) \setminus \bP(K)$. 
Then $\cA$ has the structure of a Lefschetz category over $\bP(V)$ 
\textup(with the same center $\cA_0$\textup), 
and there is a~$\bP(\vV)$-linear equivalence 
\begin{equation*}
(\cA/{\bP(V)})^{\hpd} \simeq (\cA/\bP(\tV))^{\hpd} \otimes_{\Perf(\bP(\tV^{\svee}))} \Perf(\bP(\vV)). 
\end{equation*}
\end{proposition}

\begin{remark}
The proposition can be generalized to the case where $\cA$ is not assumed to be supported over $\bP(\tV) \setminus \bP(K)$, 
by working with a suitable ``blowup'' of $\cA$. 
In the situation where $\cA$ is geometric, this is the main result of~\cite{carocci2015homological}; 
for general Lefschetz categories, see \cite[Proposition 7.1]{categorical-joins}. 
For convenience, we supply the proof in the simpler case needed in the paper. 
\end{remark}

\begin{proof}
Let $U = \bP(\tV) \setminus \bP(K)$. 
Then by the support assumption, $\cA$ has a $U$-linear structure such that 
the $\bP(\tV)$-linear structure is induced by pullback along $U \to \bP(\tV)$. 
Via the morphism~\mbox{$U \to \bP(V)$} given by linear projection, $\cA$ also carries a $\bP(V)$-linear structure. 
Let $H$ and~$\tilde{H}$ denote the relative hyperplane classes on $\bP(V)$ and $\bP(\tV)$.
Note that $\cO(H)$ and $\cO(\tilde{H})$ both pull back to the same object of $\Perf(U)$, and hence their actions on $\cA$ coincide. 
From this, it follows that the given Lefschetz center $\cA_0 \subset \cA$ is 
also a Lefschetz center with respect to the~$\bP(V)$-linear structure
{with the same Lefschetz components}. 

Consider the induced embedding $\vV \hookrightarrow \tV^{\svee}$.
There is a canonical isomorphism 
\begin{equation}
\label{hyperplane-V-VW}
U \times_{\bP(\tV)} \bH(\bP(\tV)) \times_{\bP(\tV^{\svee})} \bP(\vV) \cong 
U \times_{\bP(V)} \bH(\bP(V)) . 
\end{equation}
Using this, we deduce 
\begin{align*}
\bH(\cA/\bP(\tV)) &\otimes_{\Perf(\bP(\tV^{\svee}))} \Perf(\bP(\vV)) 
\\
& = \cA \otimes_{\Perf(\bP(\tV))} \Perf(\bH(\bP(\tV))) \otimes_{\Perf(\bP(\tV^{\svee}))} \Perf(\bP(\vV))
\\
& \simeq \cA \otimes_{\Perf(\bP(\tV))} \Perf(U) \otimes_{\Perf(\bP(\tV))} \Perf(\bH(\bP(\tV))) \otimes_{\Perf(\bP(\tV^{\svee}))} \Perf(\bP(\vV))
\\
& \simeq \cA \otimes_{\Perf(\bP(\tV))} \Perf\left(U \times_{\bP(\tV)} \bH(\bP(\tV)) \times_{\bP(\tV^{\svee})} \bP(\vV) \right) \\ 
& \simeq \cA \otimes_{\Perf(\bP(\tV))} \Perf \left(U \times_{\bP(V)} \bH(\bP(V)) \right) \\ 
& \simeq \cA \otimes_{\Perf(\bP(V))} \Perf(\bH(\bP(V))) \\ 
& = \bH(\cA/\bP(V)) .
\end{align*}
Indeed, the second line holds by definition of $\bH(\cA/\bP(\tV))$,
the third and the sixth follow from the fact that $\cA$ is supported over~$U$ 
(see Lemma~\ref{lemma-C-support-fiber}), 
the fourth holds by \cite[Theorem~1.2]{bzfn}, 
the fifth holds by~\eqref{hyperplane-V-VW}, and the last holds by definition. 
Using the semiorthogonal decomposition \eqref{HC-sod} defining the HPD category, 
it is easy to check
that this equivalence induces an equivalence between the subcategories 
\begin{align*}
(\cA/\bP(\tV))^{\hpd} \otimes_{\Perf(\bP(\tV^{\svee}))} \Perf(\bP(\vV)) & \subset 
\bH(\cA/\bP(\tV)) \otimes_{\Perf(\bP(\tV^{\svee}))} \Perf(\bP(\vV)) 
\intertext{and}
(\cA/{\bP(V)})^{\hpd} & \subset \bH(\cA/\bP(V)).  
\end{align*}
This completes the proof.
\end{proof}

\begin{remark} 
\label{remark-HPD-projection}
In the situation of Proposition~\ref{proposition-HPD-projection}, 
note that we have $K = (\vV)^\perp$ and $\cA \otimes_{\Perf(\bP(\tV))} \Perf(\bP(K)) = 0$ by the support assumption for $\cA$. 
Assume that $\cA$ is right strong and moderate as a Lefschetz category over $\bP(V)$ (and hence also over $\bP(\tV)$). 
Then~\cite[Theorem 8.7]{NCHPD} implies there is a semiorthogonal decomposition 
\begin{equation*}
(\cA/\bP(\tV))^{\hpd} \otimes_{\Perf(\bP(\tV^{\svee}))} \Perf(\bP(\vV)) = 
\llangle 
\cAd_{1-n}((1+r-n)H') , \dots, \cAd_{-r} \rrangle,
\end{equation*}
where $n = \length(\cAd)$ and $r = \rank(K)$.
This provides the left side with a Lefschetz structure of length $n - r$ and center $\cAd_{-r}$, 
with respect to which the equivalence of Proposition~\ref{proposition-HPD-projection} is a Lefschetz equivalence. 
We also note that $\cAd_{-r} = \cAd_0$. 
To explain this, we use the equivalence $\cA \simeq {^\hpd}(\cAd)$, see Remark~\ref{remark:left-hpd}.
By (the left HPD version of) \cite[Theorem 8.7(1)]{NCHPD} we have 
\begin{equation*}
\length(\cA) = \rank(\tV) - \# \{ \, i \leq 0 \mid \cAd_i = \cAd_0 \, \}. 
\end{equation*} 
On the other hand, by moderateness of $\cA$ over $\bP(V)$ we also have 
\begin{equation*}
\length(\cA) < \rank(V) = \rank(\tV) - r . 
\end{equation*} 
Hence $\# \{ \, i \leq 0 \mid \cAd_i = \cAd_0 \, \} > r$.  
\end{remark}


\providecommand{\bysame}{\leavevmode\hbox to3em{\hrulefill}\thinspace}
\providecommand{\MR}{\relax\ifhmode\unskip\space\fi MR }
\providecommand{\MRhref}[2]{%
  \href{http://www.ams.org/mathscinet-getitem?mr=#1}{#2}
}
\providecommand{\href}[2]{#2}



\begin{thebibliography}{10}

\bibitem{families-stability}
Arend Bayer, Mart\'i Lahoz, Emanuele Macr\`\i, Howard Nuer, Alexander Perry,
  and Paolo Stellari, \emph{Stability conditions in families}, arXiv:1902.08184
  (2019).

\bibitem{bzfn}
David Ben-Zvi, John Francis, and David Nadler, \emph{Integral transforms and
  {D}rinfeld centers in derived algebraic geometry}, J. Amer. Math. Soc.
  \textbf{23} (2010), no.~4, 909--966.

\bibitem{carocci2015homological}
Francesca Carocci and Zak Tur\v{c}inovi\'{c}, \emph{Homological projective
  duality for linear systems with base locus}, International Mathematics
  Research Notices (2018), rny222.

\bibitem{debarre-kuznetsov-periods}
Olivier Debarre and Alexander Kuznetsov, \emph{{{G}ushel--{M}ukai varieties:
  linear spaces and periods}}, to appear in Kyoto J. Math. arXiv:1605.05648
  (2017).

\bibitem{debarre2015kuznetsov}
\bysame, \emph{{Gushel}--{Mukai} varieties: classification and
  birationalities}, Algebr. Geom. \textbf{5} (2018), no.~1, 15--76.

\bibitem{favero-kelly}
David Favero and Tyler Kelly, \emph{Fractional {C}alabi-{Y}au categories from
  {L}andau-{G}inzburg models}, Algebr. Geom. \textbf{5} (2018), no.~5,
  596--649.

\bibitem{gaitsgory-DAG}
Dennis Gaitsgory and Nick Rozenblyum, \emph{A study in derived algebraic
  geometry}, available at \url{http://www.math.harvard.edu/~gaitsgde/GL/},
  2016.

\bibitem{categorical-plucker}
Qingyuan Jiang, Conan~Naichung Leung, and Ying Xie, \emph{Categorical
  {P}l{\"u}cker formula and homological projective duality}, arXiv:1704.01050
  (2017).

\bibitem{kuznetsov2006hyperplane}
Alexander Kuznetsov, \emph{{Hyperplane sections and derived categories}},
  Izvestiya: Mathematics \textbf{70} (2006), no.~3, 447.

\bibitem{kuznetsov-hpd}
\bysame, \emph{Homological projective duality}, Publ. Math. Inst. Hautes
  \'Etudes Sci. (2007), no.~105, 157--220.

\bibitem{kuznetsov08quadrics}
\bysame, \emph{Derived categories of quadric fibrations and intersections of
  quadrics}, Adv. Math. \textbf{218} (2008), no.~5, 1340--1369.

\bibitem{kuznetsov2008lefschetz}
\bysame, \emph{{Lefschetz decompositions and categorical resolutions of
  singularities}}, Selecta Mathematica \textbf{13} (2008), no.~4, 661--696.

\bibitem{kuznetsov2009hochschild}
\bysame, \emph{{Hochschild homology and semiorthogonal decompositions}},
  arXiv:0904.4330 (2009).

\bibitem{kuznetsov2010derived}
\bysame, \emph{{Derived categories of cubic fourfolds}}, {Cohomological and
  geometric approaches to rationality problems}, Springer, 2010, pp.~219--243.

\bibitem{kuznetsov-base-change}
\bysame, \emph{Base change for semiorthogonal decompositions}, Compos. Math.
  \textbf{147} (2011), no.~3, 852--876.

\bibitem{kuznetsov2014semiorthogonal}
\bysame, \emph{Semiorthogonal decompositions in algebraic geometry},
  Proceedings of the {I}nternational {C}ongress of {M}athematicians, {V}ol.
  {II} ({S}eoul, 2014), 2014, pp.~635--660.

\bibitem{kuznetsov2015calabi}
\bysame, \emph{{Calabi}--{Yau} and fractional {Calabi}--{Yau} categories}, to
  appear in J. Reine Angew. Math. arXiv:1509.07657 (2016).

\bibitem{kuznetsov2015rationality}
\bysame, \emph{Derived categories view on rationality problems}, pp.~67--104,
  Springer International Publishing, Cham, 2016.

\bibitem{kuznetsov2018spinor}
\bysame, \emph{On linear sections of the spinor tenfold, {I}}, Izv. Ross. Akad.
  Nauk Ser. Mat. \textbf{82} (2018), no.~4, 53--114.

\bibitem{kuznetsov-lunts}
Alexander Kuznetsov and Valery Lunts, \emph{Categorical resolutions of
  irrational singularities}, Int. Math. Res. Not. IMRN (2015), no.~13,
  4536--4625.

\bibitem{cyclic-covers}
Alexander Kuznetsov and Alexander Perry, \emph{Derived categories of cyclic
  covers and their branch divisors}, Selecta Math. (N.S.) \textbf{23} (2017),
  no.~1, 389--423.

\bibitem{kuznetsov2016perry}
\bysame, \emph{Derived categories of {G}ushel--{M}ukai varieties}, Compos.
  Math. \textbf{154} (2018), no.~7, 1362--1406.

\bibitem{categorical-joins}
\bysame, \emph{Categorical joins}, arXiv:1804.00144 (2019).

\bibitem{kuznetsov-perry-HPD-quadrics}
\bysame, \emph{Homological projective duality for quadrics}, arXiv:1902.09832
  (2019).

\bibitem{kuznetsov-prokhorov}
Alexander Kuznetsov and Yuri Prokhorov, \emph{Rationality of {F}ano threefolds
  over nonclosed fields}, in preparation (2019).

\bibitem{lurie-SAG}
Jacob Lurie, \emph{Spectral algebraic geometry}, available at
  \url{http://www.math.harvard.edu/~lurie/}.

\bibitem{macri-K3-survey}
Emanuele Macr{\`\i} and Paolo Stellari, \emph{Lectures on non-commutative
  {$K3$} surfaces, {B}ridgeland stability, and moduli spaces}, arXiv:1807.06169
  (2018).

\bibitem{double-spinor}
Laurent Manivel, \emph{Double spinor {Calabi}--{Yau} varieties},
  arXiv:1709.07736 (2017).

\bibitem{ogrady2015periods}
Kieran~G. O'Grady, \emph{Periods of double {EPW}-sextics}, Math. Z.
  \textbf{280} (2015), no.~1-2, 485--524. \MR{3343917}

\bibitem{NCHPD}
Alexander Perry, \emph{Noncommutative homological projective duality},
  arXiv:1804.00132 (2018).

\bibitem{reid-fantasy}
Miles Reid, \emph{The moduli space of {$3$}-folds with {$K=0$} may nevertheless
  be irreducible}, Math. Ann. \textbf{278} (1987), no.~1-4, 329--334.

\bibitem{thomas2015notes}
Richard Thomas, \emph{Notes on homological projective duality}, Algebraic
  geometry: {S}alt {L}ake {C}ity 2015, Proc. Sympos. Pure Math., vol.~97, Amer.
  Math. Soc., Providence, RI, 2018, pp.~585--609.

\bibitem{vdb-crepant}
Michel Van~den Bergh, \emph{Non-commutative crepant resolutions}, The legacy of
  {N}iels {H}enrik {A}bel, Springer, Berlin, 2004, pp.~749--770.

\bibitem{vdb-flops}
\bysame, \emph{Three-dimensional flops and noncommutative rings}, Duke Math. J.
  \textbf{122} (2004), no.~3, 423--455.

\end{thebibliography}
\end{document}